\documentclass[reqno]{amsart}
\pdfoutput=1
\usepackage[T1]{fontenc}
\usepackage[utf8]{inputenc} 
\usepackage{ae}
\usepackage{aecompl}
\usepackage{bbold}
\usepackage{amssymb}
\usepackage{stmaryrd}
\usepackage{xspace}
\usepackage[all]{xy} 
\usepackage[english]{babel}

\usepackage{enumitem} 
\setenumerate[1]{label=(\arabic*),leftmargin=*}

\usepackage[low-sup]{subdepth}

\usepackage{color}
\usepackage{url}
\usepackage{hyperref}
\hypersetup{bookmarksdepth=1, colorlinks, citecolor=blue,
linkcolor=blue}
\newcounter{sat}

\makeatletter
\newcommand*{\@old@slash}{}\let\@old@slash\slash
\def\slash{\relax\ifmmode\delimiter"502F30E\mathopen{}\else\@old@slash\fi}
\makeatother

\makeatletter
\DeclareRobustCommand*\cal{\@fontswitch\relax\mathcal}
\makeatother

\usepackage{tikz}
\newcommand*\numbercircled[1]{\tikz[baseline=(char.base)]{
            \node[shape=circle,draw,inner sep=1pt] (char) {#1};}}
\newtheoremstyle{plaindotless}{}{}{\itshape}{}{\bfseries}{}{1em}{}
\theoremstyle{plaindotless}
\newtheorem{thm}[sat]{Theorem}
\newtheorem{lem}[sat]{Lemma}
\newtheorem{cor}[sat]{Corollary}
\newtheorem{pro}[sat]{Proposition}

\newtheoremstyle{definitiondotless}{}{}{}{}{\bfseries}{}{1em}{}
\theoremstyle{definitiondotless}
\newtheorem{dfi}[sat]{Definition}
\newtheorem{exa}[sat]{Example}
\newtheorem{rem}[sat]{Remark}

\providecommand{\inj}{\hookrightarrow}

\providecommand{\Z}{\ensuremath{\mathbb{Z}}}
\providecommand{\Q}{\ensuremath{\mathbb{Q}}}

\providecommand{\eg}{\mbox{e.\,g.}\xspace}
\providecommand{\ie}{\mbox{i.\,e.}\xspace}

\DeclareMathOperator{\im}{im}
\DeclareMathOperator{\hm}{Hom}

\newcommand{\smallucarre}{\, \xy <0cm,-.05cm> ;<.12cm,-.05cm>: 
(0,0) ; (0,1) **\dir{-},
(1,1) ; (0,1) **\dir{-}, \endxy \, }
\newcommand{\smalldcarre}{ \xy <0cm,-.05cm> ;<.12cm,-.05cm>: 
(0,0) ; (1,0) **\dir{-},
(1,1) ; (1,0) **\dir{-}, \endxy \, }
\newcommand{\carre}{\, \xy <0cm,-.10cm> ;<.18cm,-.10cm>: 
(0,0) ; (1,0) **\dir{-},
(0,0) ; (0,1) **\dir{-},
(1,1) ; (1,0) **\dir{-},
(1,1) ; (0,1) **\dir{-}, \endxy \, }

\usepackage{cleveref}
\usepackage{mathtools}
\usepackage[backend=bibtex,style=alphabetic,backref=true]{biblatex}
\DeclareFieldFormat{postnote}{#1}
\usepackage{titlesec}
\titlelabel{\thetitle.\quad}

\newcommand{\Id}{\ensuremath{1}}
\newcommand{\unit}{\ensuremath{1}}
\newcommand{\catid}{\ensuremath{1}}
\newcommand{\Idtransf}{\ensuremath{1}}
\newcommand{\D}{\ensuremath{\mathbb{D}}}
\newcommand{\pt}{\ensuremath{\star}}
\newcommand{\one}{\ensuremath{\mathbb{1}}}
\newcommand{\ehm}[2]{\ensuremath{\langle#1,#2\rangle}}
\newcommand{\tw}[1]{\ensuremath{\mathrm{tw}(#1)}}
\newcommand{\disc}[1]{\ensuremath{\mathrm{2tw}(#1)}}
\newcommand{\eqrel}{\ensuremath{\sim}}
\newcommand{\emor}[1]{\ensuremath{\mathrm{end}(#1)}}
\newcommand{\plho}{{-}}
\newcommand{\Fun}{\ensuremath{\mathbf{Fun}_{\mathrm{lax}}}}
\DeclareMathOperator{\Cha}{char}
\DeclareMathOperator{\rad}{rad}

\xyoption{2cell}
\UseTwocells{}
\newcommand{\op}{\ensuremath{\circ}}
\newcommand\restr[2]{{
  \left.\kern-\nulldelimiterspace 
  #1 
  \right|_{#2} 
  }}
\usepackage{amsaddr}
\title{Traces in monoidal derivators, and homotopy colimits}
\author{Martin Gallauer Alves de Souza}\thanks{The author was
  supported by the Swiss National Science Foundation (SNSF)}
\address{Institut für Mathematik, Universität Zürich, Switzerland}
\email{martin.gallauer@math.uzh.ch}
\begin{document}
\begin{abstract}
  \noindent{}A variant of the trace in a monoidal category is given in
  the setting of closed monoidal derivators, which is applicable to
  endomorphisms of fiberwise dualizable objects. Functoriality of this
  trace is established. As an application, an explicit formula is
  deduced for the trace of the homotopy colimit of endomorphisms over
  finite categories in which all endomorphisms are invertible. This
  result can be seen as a generalization of the additivity of traces
  in monoidal categories with a compatible triangulation.
\end{abstract}
\maketitle{}
\section{Introduction}
\label{sec:intro}

\subsection{The additivity of traces.}

Let ${\cal C}$ be a symmetric monoidal category which in addition is
triangulated. Examples include various ``stable homotopy categories''
(such as the classical and equivariant in algebraic topology, the
motivic in algebraic geometry) or all kinds of ``derived categories''
(of modules, of perfect complexes on a scheme, etc.). Let $X$, $Y$ and
$Z$ be dualizable objects in ${\cal C}$,
\begin{equation*}
  D:\qquad X\to Y \to Z\to^{+}
\end{equation*}
a distinguished triangle, and $f$ an endomorphism of $D$. The
\emph{additivity of traces} is the statement that the following
relation holds among the traces of the components of
$f$:
\begin{equation}\label{eq:triangle}
  \mathrm{tr}(f_{Y})=\mathrm{tr}(f_{X})+\mathrm{tr}(f_{Z}).
\end{equation}
Well-known examples are the additivity of the Euler characteristic of
finite CW\nobreakdash-com\-plexes ($\chi(Y)=\chi(X)+\chi(Y/X)$ for
$X\subset Y$ a subcomplex) or the additivity of traces in short exact
sequences of finite dimensional vector spaces. The additivity of
traces should be considered as a \emph{principle}: Although incorrect
as it stands, it embodies an important idea. One should therefore try
to find the right context to formulate this idea precisely and prove
it.

In \cite{may_additivity}, J.\,Peter May made an important step in this
direction. He gave a list of axioms expressing a compatibility of the
monoidal and the triangulated structure, and proved that if they are
satisfied, then one can always replace $f$ by an endomorphism $f'$
with $f'_{X}=f_{X}$ and $f'_{Y}=f_{Y}$ such that~(\ref{eq:triangle})
holds for the components of $f'$. This result has two drawbacks
though: Firstly, there is this awkwardness of $f'$ replacing $f$, and
secondly, the axioms are rather complicated.

As noted in~\cite{GrothPontoShulman-additivity}, both these drawbacks
are related to the well-known deficiencies of triangulated
categories. Since the foremost example of a situation in which May's
compatibility axioms hold, is when $\mathcal{C}$ is the homotopy
category of a stable monoidal model category, it should not come as a
surprise that May's result can be reproved in the setting of
triangulated derivators. Moreover, since triangulated derivators
eliminate some of the problems encountered in triangulated categories,
a more satisfying formulation of the additivity of traces should be
available. We will describe it now.

Let $\D$ be a closed symmetric monoidal triangulated
derivator\footnote{See section~\ref{sec:ntt} for the definition of
  this notion.}, and
$\smallucarre$ the free category on the following graph:
\begin{equation}\label{eq:pushout-category}
  \xymatrix{(1,1)&(0,1)\ar[l]\\
    (1,0)\mathrlap{.}\ar[u]}
\end{equation}
Let $A$ be an object of $\D(\smallucarre)$ with underlying diagram
\begin{equation*}
  \xymatrix{X\ar[r]\ar[d]&Y\\
    0\mathrlap{,}}
\end{equation*}
and suppose that both $X$ and $Y$ are dualizable objects of $\D(\pt)$,
$\pt$ denoting the terminal category. Let $f$ be an endomorphism of
$A$, and denote by $p_{\smallucarre}$ the unique functor
$\smallucarre\to \pt$. Then there is a distinguished triangle
\begin{equation*}
  X\to Y\to p_{\smallucarre!}A\to^{+}
\end{equation*}
in $\D(\pt)$, $p_{\smallucarre!}A$ is also dualizable, and the
following relation holds:
\begin{equation*}
  \mathrm{tr}(f_{(0,1)})=\mathrm{tr}(f_{(1,1)})+\mathrm{tr}(p_{\smallucarre!}f).
\end{equation*}
This is the main theorem of~\cite{GrothPontoShulman-additivity}.

\subsection{The trace of the homotopy colimit.}

Another advantage of the formulation in the context of derivators is
that it immediately invites us to consider the additivity of traces as
a mere instance of a more general principle. As a first step, we see
that the condition $A_{(1,0)}=0$ is not essential. Indeed, if $A$ is
an object of $\D(\smallucarre)$ whose fibers are all dualizable
objects in $\D(\pt)$ and if $f$ is an endomorphism of $A$ then the
formula above generalizes to
\begin{equation}\label{eq:hty-pushout-result}
  \mathrm{tr}(p_{\smallucarre!}f)=\mathrm{tr}(f_{(0,1)})+\mathrm{tr}(f_{(1,0)})-\mathrm{tr}(f_{(1,1)}).
\end{equation}
And now, in the second step, it is natural to replace the category
$\smallucarre$ by other categories $I$ and try to see whether there
still is an explicit formula for $\mathrm{tr}(p_{I!}f)$. The main
result of the present article states that this is the case for finite
EI-categories, \ie{} finite categories in which all endomorphisms are
invertible (such as groups or posets), provided that the derivator is
$\Q$-linear. For each of these categories the trace of the homotopy
colimit of an endomorphism of a fiberwise dualizable object can be
computed as a linear combination of ``local traces'' (depending only
on the fibers of the endomorphism and the action of the automorphisms
of the objects in the category) with coefficients which depend only on
the category and can be computed combinatorially.

As for the proof of this result, the idea is to define the trace of
endomorphisms of objects not only living in $\D(\pt)$ but in $\D(I)$
for general categories $I$. This trace should contain enough
information to relate the trace of the homotopy colimit to the local
traces of the endomorphism. The naive approach of considering $\D(I)$
as a monoidal category and taking the usual notion of the trace
doesn't lead too far though since few objects in $\D(I)$ will be
dualizable in general even if in $\D(\pt)$ all of them are; in other
words, being fiberwise dualizable does not imply being dualizable.

This is why we will replace the ``internal'' tensor product by an
``external product''
\begin{equation*}
  \boxtimes:\D(I)\times\D(I)\to \D(I\times I)
\end{equation*}
and the internal hom by an ``external hom''
\begin{equation*}
  \ehm{\plho}{\plho}:\D(I)^{\op}\times\D(I)\to \D(I^{\op}\times I),
\end{equation*}
which has the property that for any object $A$ of $\D(I)$ and objects
$i,j$ of $I$
\begin{equation*}
  \ehm{A}{A}_{(i,j)}=[A_{i},A_{j}]
\end{equation*}
(implying that fiberwise dualizable objects will be ``dualizable with
respect to the external hom'') and which also contains enough
information to compute $[A,A]$ (among other desired formal
properties). As soon as this bifunctor is available we can mimic the
usual definition of the ``internal'' trace in a closed symmetric
monoidal category to define an ``external'' trace for any endomorphism
of a fiberwise dualizable object, replacing the internal by the
external hom everywhere. It will turn out that this new trace encodes
all local traces, and in good cases allows us to relate these to the
trace of the homotopy colimit, thus yielding the sought after formula.

A ``general additivity theorem'' for traces, supposedly similar to the
main result in this article, has been obtained by Kate Ponto and
Michael Shulman. However, their proof should be quite different from
the one presented here, relying on the technology of bicategorical
traces (personal communication, December 2012).

\subsection{Outline of the paper.}

We do not include an introduction to the theory of derivators (see for
this the references given in section~\ref{par:derivator}). However, as
the definition of a derivator varies in the literature we give the
axioms we use in section~\ref{par:derivator}. Moreover, the few
results on derivators we need in the article are either proved or
justified by a reference to where a proof can be found. In
section~\ref{sec:monder} we define the notion of a (closed) monoidal
derivator and describe its relation to the axiomatization available in
the literature. We also discuss briefly linear structures on
derivators (\ref{sec:linder}), triangulated derivators
(\ref{sec:trder}), and the interplay between triangulated and monoidal
structures on derivators (\ref{sec:trmonder}). Apart from this,
section~\ref{sec:ntt} is meant to fix the notation used in the
remainder of the article.

The main body of the text starts with section~\ref{sec:ehom} where the
construction of the external hom mentioned above is given. The proofs
of the desired formal properties of this bifunctor are lengthy and not
needed in the rest of the article so they are deferred to
appendix~\ref{sec:app1}. Next we define the external trace
(section~\ref{sec:trace-dfi}) and prove its functoriality
(section~\ref{sec:trace-nat}). As a corollary we deduce that this
trace encodes all local traces.

The main result is to be found in section~\ref{sec:formula}. First we
prove that in good cases the trace of the homotopy colimit is a
function of the external trace (again, the uninteresting part of the
story is postponed to the appendix; specifically to
appendix~\ref{sec:app2}). In the case of finite EI-categories and a
$\Q$-linear triangulated derivator, this function can be made
explicit, and this leads to the formula for the trace of the homotopy
colimit in terms of the local traces. Some technical hypotheses used
to prove this result will be eliminated in section~\ref{sec:q}.

At several points in the article the need arises for an explicit
description of an additive derivator evaluated at a finite
group. Although this description is certainly well-known, we haven't
been able to find it in the literature and have thus included it as
appendix~\ref{sec:group}.

\section*{Acknowledgments}
Joseph Ayoub not only suggested to me that there should be an explicit
formula for the trace of the homotopy colimit of an endomorphism in
the setting of a closed symmetric monoidal triangulated derivator but
also contributed several crucial ideas and constructions. In
particular, the important definition of the ``external hom'' mentioned
above is due to him. I am also grateful to Michael Shulman for
pointing out that an earlier formula was too simplistic. Finally, I
was lucky to receive a very competent and detailed report by the
referee which led to many improvements in the article. I would like to
thank him or her warmly.

\section{Conventions and preliminaries}
\label{sec:ntt}

In this section, we recall some notions and facts (mostly related to
derivators) and fix the notation used in the remainder of the article.
\subsection{}By a 2-category we mean a strict 2-category. The
2-category of (small) categories is denoted by $\mathbf{CAT}$
($\mathbf{Cat}$). Given a 2-category $\mathcal{C}$ (encompassing the
special case of a category), $\mathcal{C}^{\op}$ denotes the
2-category with the same objects, and ${\cal C}^{\op}(x,y)={\cal
  C}(y,x)$ for all objects $x,y$. The 2-category
$\mathcal{C}^{\op,\op}$ also has the same objects as ${\cal C}$ but
${\cal C}^{\op,\op}(x,y)={\cal C}(y,x)^{\op}$
(see~\cite[p.~82]{Kelly-Street}). The (possibly large) sets of
objects, 1-morphisms and 2-morphisms in a 2-category $\mathcal{C}$ are
sometimes denoted by $\mathcal{C}_{0}$, $\mathcal{C}_{1}$, and
$\mathcal{C}_{2}$ respectively.

By a 2-functor we mean a \emph{strict} 2-functor between
2-categories. Modifications are morphisms of lax natural
transformations between 2-functors
(see~\cite[p.~82]{Kelly-Street}). For fixed 2-categories $\mathcal{C}$
and $\mathcal{D}$, the 2-functors from $\mathcal{C}$ to
$\mathcal{D}$ together with lax natural transformations and
modifications form a 2-category
$\Fun(\mathcal{C},\mathcal{D})$.

\subsection{}\label{par:cat-ntt-constr}Counits and units of adjunctions are usually denoted by
$\mathrm{adj}$. Given a functor $u:I\to J$, and an object $j\in
J_{0}$, the category of \emph{objects $u$-under $j$} is (abusively)
denoted by $j\backslash I$ and the category of \emph{objects $u$-over
  $j$} by $I/j$ (see~\cite[2.6]{MacLane71-cat-work-math}). We
also need the following construction
(\cite[p.~223]{MacLane71-cat-work-math}): Given a category $I$, we
define the \emph{twisted arrow category} associated to $I$, denoted by
$\tw{I}$, as having objects the arrows of $I$ and as morphisms from $i\to
j$ to $i'\to j'$ pairs of morphisms making the following square in $I$
commutative:
\begin{equation*}
  \xymatrix{i\ar[r]&j\ar[d]\\
    i'\ar[u]\ar[r]&j'\mathrlap{.}}
\end{equation*}
There is a canonical functor $\tw{I}\to I^{\op}\times I$. In fact,
this extends canonically to a functor
$\tw{\plho}:\mathbf{Cat}\to\mathbf{Cat}$ together with a natural
transformation $\tw{\plho}\to(\plho)^{\op}\times (\plho)$.

\subsection{}\label{par:derivator}
Let us recall the notion of a derivator. For the basic theory we refer
to \cite{maltsiniotis01-intro-derivators}, \cite{cisinski-neeman},
\cite{groth_pt_stable_derivator}. For an outline of the history of the
subject see~\cite[p.~1385]{cisinski-neeman}.

A full sub-2-category $\mathbf{Dia}$ of $\mathbf{Cat}$ is called a
\emph{diagram category} if:
\begin{enumerate}[label=(Dia\arabic*)]
\item $\mathbf{Dia}$ contains the totally ordered set
  $\underline{2}=\{0<1\}$;
\item $\mathbf{Dia}$ is closed under finite products and coproducts,
  and under taking the opposite category and
  subcategories;\label{dia:closed}
\item if $I\in\mathbf{Dia}_{0}$ and $i\in I_{0}$, then $I/i\in
  \mathbf{Dia}_{0}$;
\item if $p:I\to J$ is a fibration (to be understood in the sense
  of~\cite[exposé~VI]{sga1}) whose fibers are all in $\mathbf{Dia}$,
  and if $J\in\mathbf{Dia}_{0}$, then also
  $I\in\mathbf{Dia}_{0}$.
\end{enumerate}
By~\ref{dia:closed}, the initial category $\emptyset$ and the terminal
category $\pt$ are both in $\mathbf{Dia}$. We will often use that
$\mathbf{Dia}$ is closed under pullbacks (as follows
from~\ref{dia:closed}). The smallest diagram category consists of
finite posets, other typical examples include finite categories,
finite-dimensional categories, all small posets or $\mathbf{Cat}$
itself.

A \emph{prederivator (of type $\mathbf{Dia}$)} is a 2-functor
$\D:\mathbf{Dia}^{\op,\op}\to\mathbf{CAT}$ from a diagram category
$\mathbf{Dia}$ to $\mathbf{CAT}$. If $\D$ is fixed in a context,
$\mathbf{Dia}$ always denotes the domain of $\D$.  Given a
prederivator $\D$, categories $I, J\in\mathbf{Dia}_{0}$ and a functor
$u:I\to J$, we denote by $u^{*}:\D(J)\to\D(I)$ the value of $\D$ at
$u$; if $u$ is clear from the context, we sometimes denote $u^{*}$ by
$\restr{}{I}$. Its left and right adjoint (if they exist) are denoted
by $u_{!}$ and $u_{*}$ respectively. The unique functor $I\to\pt$ is
denoted by $p_{I}$. Given an object $i\in I_{0}$, we denote also by
$i:\pt\to I$ the functor pointing $i$. Thus for an object
$A\in\D(I)_{0}$ and a morphism $f\in\D(I)_{1}$, their \emph{fiber
  over} $i$ is $i^{*}A$ and $i^{*}f$, respectively, sometimes also
denoted by $A_{i}$ and $f_{i}$, respectively. Given a natural
transformation $\eta:u\to v$ in $\mathbf{Dia}$, we denote by
$\eta^{*}$ the value of $\D$ at $\eta$. It is a natural transformation
from $v^{*}$ to $u^{*}$. In particular, if $h:i\to j$ is an arrow in
$I$ then we can consider it as a natural transformation from the
functor $i:\pt\to I$ to $j:\pt\to I$, and therefore it makes sense to
write $h^{*}$; evaluated at an object $A\in\D(I)_{0}$, it yields a
morphism of the fibers $A_{j}\to A_{i}$. The canonical ``underlying
diagram'' functor $\D(I)\to\mathbf{CAT}(I^{\op},\D(\pt))$ is denoted
by $\mathrm{dia}_{I}$. Finally, if $\D$ is a prederivator and
$J\in\mathbf{Dia}_{0}$, we denote by $\D_{J}$ the prederivator
$\D_{J}(\plho)=\D(\plho\times J)$.

A \emph{derivator (of type $\mathbf{Dia}$)} is a prederivator (of type
$\mathbf{Dia}$) $\D$ satisfying the following list of axioms:
\begin{enumerate}[label=(D\arabic*), series=derivator]
\item $\D$ takes arbitrary coproducts to products up to equivalence of
  categories.\label{der:coprod}
\item For every $I\in \mathbf{Dia}_{0}$, the family of functors
  $i^{*}:\D(I)\to\D(\pt)$ indexed by $I_{0}$ is jointly conservative.\label{der:conservative}
\item For all functors $u\in\mathbf{Dia}_{1}$, the left and right
  adjoints $u_{!}$ and $u_{*}$ to $u^{*}$ exist.
\item Given a functor $u:I\to J$ in $\mathbf{Dia}$ and an object $j\in
  J_{0}$, the ``Beck-Chevalley'' transformations associated to both comma
  squares
  \begin{align*}
    \xymatrix{j\backslash I\ar[r]^{t}\ar[d]_{p_{j\backslash I}}&I\ar[d]^{u}\\
      \pt\ar[r]_{j}&J\xtwocell[-1,-1]{}\omit}&&\text{and}&&
    \xymatrix{\xtwocell[1,1]{}\omit I/j\ar[r]^{s}\ar[d]_{p_{I/j}}&I\ar[d]^{u}\\
      \pt\ar[r]_{j}&J}
  \end{align*}
  are invertible: $p_{j\backslash
    I!}t^{*}\xrightarrow{\sim}j^{*}u_{!}$ and
  $j^{*}u_{*}\xrightarrow{\sim}p_{I/j*}s^{*}$.\label{der:kanextptwise}
\end{enumerate}
The derivator $\D$ is called \emph{strong} if in addition
\begin{enumerate}[resume*=derivator]
\item For every $J\in\mathbf{Dia}_{0}$, the functor
  $\mathrm{dia}_{\underline{2}}:\D_{J}(\underline{2})\to\mathbf{CAT}(\underline{2}^{\op},\D_{J}(\pt))$
  is full and essentially surjective.
\end{enumerate}

As an important example, if ${\cal M}$ is a model category then the
association
\begin{align*}
  \D^{{\cal M}}:\mathbf{Cat}^{\op,\op}&\longrightarrow \mathbf{CAT}\\
  I&\longmapsto {\cal M}^{I^{\op}}[{\cal W}_{I}^{-1}]
\end{align*}
defines a strong derivator, where ${\cal M}^{I^{\op}}[{\cal
  W}_{I}^{-1}]$ denotes the category obtained from ${\cal
  M}^{I^{\op}}$ by formally inverting those morphisms of presheaves
which are pointwise weak equivalences. This result is due to
Denis-Charles Cisinski (see~\cite{cisinski-imagesdirectes}). If $\D$
is a (strong) derivator and $J\in\mathbf{Dia}_{0}$ then also $\D_{J}$
is a (strong) derivator.

One consequence of the axioms we shall often have occasion to refer to
is the following result on (op)fibrations:
\begin{lem}\label{lem:fibration-exact}
  Given a derivator $\D$ of type $\mathbf{Dia}$ and given a pullback
  square
\begin{equation*}
  \label{eq:hty_exact_square}
  \xymatrix{\ar[r]^{w}\ar[d]_{v}&\ar[d]^{u}\\
    \ar[r]_{x}&}
\end{equation*}
in $\mathbf{Dia}$ with either $u$ a fibration or $x$ an opfibration,
the canonical ``Beck-Chevalley'' transformation
\begin{equation*}\label{eq:beck-chevalley}
  v_{!}w^{*}\longrightarrow x^{*}u_{!} \qquad \text{(or,
    equivalently, }
  u^{*}x_{*}\longrightarrow w_{*}v^{*}\text{)}
\end{equation*}
is invertible.
\end{lem}
For a proof see~\cite[1.30]{groth_pt_stable_derivator} or
\cite[2.7]{heller-homotopytheories}.

\subsection{}\label{sec:monder}
By a \emph{monoidal category} we always mean a symmetric unitary
monoidal category. \emph{Monoidal functors} between monoidal
categories are functors which preserve the monoidal structure up to
(coherent) natural isomorphisms; in the literature, these are
sometimes called \emph{strong} monoidal functors. \emph{Monoidal
  transformations} are natural transformations preserving the monoidal
structure in an obvious way. We thus arrive at the 2-category of
monoidal categories $\mathbf{MonCAT}$. The monoidal product is always
denoted by $\otimes$ and the unit by $\one$. If internal hom functors
exist, we arrive at its closed variant $\mathbf{ClMonCAT}$. (Notice
that functors between closed categories are not required to be
closed. In other words, $\mathbf{ClMonCAT}$ is a \emph{full}
sub-2-category of $\mathbf{MonCAT}$.) The internal hom functor is
always denoted by $[\plho,\plho]$.

A \emph{(closed) monoidal prederivator (of type $\mathbf{Dia}$)} is a
prederivator with a factorization
\begin{equation*}
  \D:\mathbf{Dia}^{\op,\op}\to\mathbf{(Cl)MonCAT}\to \mathbf{CAT},
\end{equation*}
where $\mathbf{(Cl)MonCAT}\to \mathbf{CAT}$ is the forgetful
functor. (Closed) monoidal prederivators were also discussed
in~\cite[2.1.6]{ayoub07-thesis-1} and \cite{Groth-monder}.

Let us now define the ``external product'' mentioned in the
introduction. Given a monoidal prederivator $\D$ and categories $I$,
$J$ in $\mathbf{Dia}$ we define the bifunctor
\begin{align*}
  \boxtimes:\D(I)\times\D(J)&\to\D(I\times J)\\
  (A,B)&\mapsto \restr{A}{I\times J}\otimes \restr{B}{I\times J}.
\end{align*}
Given two functors $u:I'\to I$ and $v:J'\to J$ in $\mathbf{Dia}$,
$A\in \D(I)_{0}$ and $B\in \D(J)_{0}$, we define a morphism
\begin{equation}\label{eq:eprod-nat}
  (u\times v)^{*}(A\boxtimes B)\to u^{*}A\boxtimes v^{*}B
\end{equation}
as the composition
\begin{align*}
  (u\times v)^{*}(\restr{A}{I\times J}\otimes \restr{B}{I\times J})&\xrightarrow{\sim}(u\times
  v)^{*}\restr{A}{I\times J}\otimes (u\times v)^{*}\restr{B}{I\times J}\\
  &= \restr{(u^{*}A)}{I'\times J'}\otimes \restr{(v^{*}B)}{I'\times J'}.
\end{align*}
Hence we see that~(\ref{eq:eprod-nat}) is in fact an isomorphism, and
it is clear that it is also natural in $A$ and $B$. Putting these and
similar properties together one finds that the external product
defines a pseudonatural transformation of 2-functors
\begin{equation*}
   \D\times \D\to \D\circ (\plho\times \plho),
\end{equation*}
\ie{} a 1-morphism in
$\Fun(\mathbf{Dia}^{\op,\op}\times\mathbf{Dia}^{\op,\op},\mathbf{CAT})$
with invertible 2-cell components.

Now fix a monoidal prederivator $\D$, a functor $u:I\to
J\in\mathbf{Dia}_{1}$, and $A\in \D(I)_{0}$, $B\in\D(J)_{0}$. We can
define the \emph{projection morphism}
\begin{equation}\label{eq:projection}
  u_{!}(A\otimes u^{*}B)\to u_{!}A\otimes B
\end{equation}
by adjunction as the composition
\begin{equation*}
  A\otimes u^{*}B\xrightarrow{\mathrm{adj}}u^{*}u_{!}A\otimes u^{*}B\xleftarrow{\sim} u^{*}(u_{!}A\otimes B).
\end{equation*}
It is clearly natural in $A$ and $B$. Fix a second functor $v:I'\to
J'$ in $\mathbf{Dia}$ and consider the following morphism
($A\in\D(I)_{0}$, $B\in\D(I')_{0}$):
\begin{equation}\label{eq:mon-der-external}
  (u\times v)_{!}(A\boxtimes B)\to u_{!}A\boxtimes v_{!}B,
\end{equation}
obtained by adjunction from
\begin{equation*}
  A\boxtimes B\xrightarrow{\mathrm{adj}} u^{*}u_{!}A\boxtimes
  v^{*}v_{!}B\xleftarrow{\sim}(u\times v)^{*}(u_{!}A\boxtimes
  v_{!}B).
\end{equation*}
Of course, it is also natural in $A$ and $B$.

\begin{lem}\label{lem:mon-der} Let $\D$ be a monoidal prederivator which
  satisfies the axioms of a derivator. Then the following conditions
  are equivalent:
  \begin{enumerate}
  \item The projection morphism~(\ref{eq:projection}) is invertible
    for all $u=p_{I}$, $I\in\mathbf{Dia}_{0}$.\label{lem:mon-der.proj}
  \item The projection morphism~(\ref{eq:projection}) is invertible
    for all fibrations $u$ in $\mathbf{Dia}$.\label{lem:mon-der.proj-fib}
  \item (\ref{eq:mon-der-external}) is invertible for all
    $u,v\in\mathbf{Dia}_{1}$.\label{lem:mon-der.external}
  \end{enumerate}
  If $\D$ is a closed monoidal prederivator then
  condition~\ref{lem:mon-der.proj-fib} is also equivalent to each of the
  following ones:
  \begin{enumerate}[resume]
  \item $u^{*}[B,B']\to [u^{*}B,u^{*}B']$ is invertible for all
    fibrations $u$ in $\mathbf{Dia}$.\label{lem:mon-der.closed}
  \item $[u_{!}A,B]\to u_{*}[A,u^{*}B]$ is invertible for all
    fibrations $u$ in $\mathbf{Dia}$.\label{lem:mon-der.adjunction}
  \end{enumerate}
\end{lem}

\begin{dfi} A \emph{(closed) monoidal derivator} is a (closed)
  monoidal prederivator which satisfies the axioms of a derivator as
  well as the equivalent conditions of Lemma~\ref{lem:mon-der}.
\end{dfi}

\begin{proof}[Proof of Lemma~\ref{lem:mon-der}.]
  Assume condition~\ref{lem:mon-der.proj}. Let $u:I\to J$ be a
  fibration in $\mathbf{Dia}$ and consider, for any $j\in J_{0}$, the
  following pullback square:
  \begin{equation*}
    \xymatrix{I_{j}\ar[r]^{w}\ar[d]_{p_{I_{j}}}&I\ar[d]^{u}\\
      \pt\ar[r]_{j}&J\mathrlap{.}}
  \end{equation*}
  Since $u$ is a fibration the base change morphism
  $p_{I_{j}!}w^{*}\to j^{*}u_{!}$ is an isomorphism, by
  Lemma~\ref{lem:fibration-exact}. Hence for any $A\in\D(I)_{0}$,
  $B\in\D(J)_{0}$, all vertical morphisms in the commutative diagram
  below are invertible:
  \begin{equation*}
    \xymatrix{j^{*}u_{!}(A\otimes u^{*}B)\ar[r]&j^{*}(u_{!}A\otimes B)\ar[d]\\\ar[u]
      p_{I_{j}!}w^{*}(A\otimes u^{*}B)\ar[d]&j^{*}u_{!}A\otimes
      j^{*}B\\
      p_{I_{j}!}(w^{*}A\otimes
      w^{*}u^{*}B)\ar@{=}[d]&p_{I_{j}!}w^{*}A\otimes j^{*}B\ar@{=}[d]\ar[u]\\
      p_{I_{j}!}(w^{*}A\otimes p_{I_{j}}^{*}j^{*}B)\ar[r]&p_{I_{j}!}w^{*}A\otimes j^{*}B\mathrlap{.}}
  \end{equation*}
  By assumption, the bottom horizontal arrow is an isomorphism hence
  so is the top one. Condition~\ref{lem:mon-der.proj-fib} now follows
  from~\ref{der:conservative}.

  For condition~\ref{lem:mon-der.external}, write $u\times v= (u\times
  \catid)\circ (\catid\times v)$ hence by symmetry of the monoidal
  product we reduce to the case where $u=\catid_{I}$, $v:J'\to J$. We
  use~\ref{der:conservative}, thus let $i\in I_{0}$, $j\in J_{0}$. The
  fiber of~(\ref{eq:mon-der-external}) over $(i,j)$ is easily seen to
  be the following composition ($w$ denotes the fibration $i\backslash
  I\times j\backslash J'\to i\backslash I$):
  \begin{align*}
    (i,j)^{*}(\catid_{I}\times v)_{!}(\restr{A}{I\times
        J'}\otimes\restr{B}{I\times
        J'})\xleftarrow{\sim}&\ p_{i\backslash I!}w_{!}(\restr{A}{i\backslash I\times
        j\backslash J'}\otimes\restr{B}{i\backslash I\times
        j\backslash
        J'})\\
      \xrightarrow{\sim}&\ p_{i\backslash I!}(\restr{A}{i\backslash I}\otimes w_{!}\restr{B}{i\backslash I\times
        j\backslash J'})\\
      \xrightarrow{\sim}&\ p_{i\backslash I!}(\restr{A}{i\backslash I}\otimes
      p_{i\backslash I}^{*}p_{j\backslash J'!}\restr{B}{j\backslash
        J'})\\
      \xrightarrow{\sim}&\ p_{i\backslash I!}\restr{A}{i\backslash I}\otimes
      p_{j\backslash J'!}\restr{B}{j\backslash J'}\\
      \xrightarrow{\sim}&\ i^{*}A\otimes j^{*}v_{!}B\\
      \xleftarrow{\sim}&\ (i,j)^{*}(\restr{A}{I\times
        J}\otimes \restr{(v_{!}B)}{I\times J})
  \end{align*}
  The first, the third and the fifth arrows come
  from~\ref{der:kanextptwise}, while the second and the fourth are
  invertible by condition~\ref{lem:mon-der.proj-fib}, the last one is
  clearly invertible.

  Putting $u=p_{I}$, $v=\catid_{\pt}$ in
  condition~\ref{lem:mon-der.external}, one obtains precisely
  condition~\ref{lem:mon-der.proj}. This finishes the proof of the
  first statement in the lemma.

  From now on we assume that $\D$ is a closed monoidal
  prederivator. For condition~\ref{lem:mon-der.closed}, notice that
  $u^{*}\circ [B,\plho]\to [u^{*}B,\plho]\circ u^{*}$ corresponds via
  the adjunctions
  \begin{equation*}
    u_{!}\circ (\plho\otimes u^{*}B)\dashv [u^{*}B,\plho]\circ
    u^{*}\quad\text{and}\quad (\plho\otimes B)\circ u_{!}\dashv u^{*}\circ[B,\plho]
  \end{equation*}
  to the projection morphism
  \begin{equation*}
    u_{!}(\plho\otimes u^{*}B)\to u_{!}\plho\otimes B.
  \end{equation*}
  And similarly, the morphism $[u_{!}A,\plho]\to u_{*}\circ
  [A,\plho]\circ u^{*}$ corresponds via the adjunctions
  \begin{equation*}
     u_{!}A\otimes\plho\dashv [u_{!}A,\plho]\quad\text{and}\quad
    u_{!}\circ(A\otimes \plho)\circ u^{*}\dashv u_{*}\circ [A,\plho]\circ u^{*}
  \end{equation*}
  to the projection morphism
  \begin{equation*}
    u_{!}(A\otimes u^{*}\plho)\to u_{!}A\otimes \plho.
  \end{equation*}
  Hence conditions~\ref{lem:mon-der.closed} and
  \ref{lem:mon-der.adjunction} are both equivalent to
  condition~\ref{lem:mon-der.proj-fib}. (For more details,
  see~\cite[2.1.144, 2.1.146]{ayoub07-thesis-1}.)
\end{proof}

In contrast to this, in a closed monoidal prederivator, the canonical
morphism
\begin{equation}\label{eq:pbpf-adjunction}
  [A,u_{*}B]\to u_{*}[u^{*}A,B]
\end{equation}
is \emph{always} invertible, even if $u$ is not a fibration.

If $\D$ is a (strong) derivator of type $\mathbf{Cat}$ then
precomposition with the 2-functor
$(\plho)^{\op}:\mathbf{Cat}^{\op}\to\mathbf{Cat}^{\op,\op}$ defines a
(strong) derivator $\overline{\D}$ in the sense
of~\cite{GrothPontoShulman-additivity}, and conversely starting with a
(strong) derivator in their sense, precomposition with $(\plho)^{\op}$
yields a (strong) derivator of type $\mathbf{Cat}$. By
Lemma~\ref{lem:mon-der}, $\D$ being monoidal corresponds to
$\overline{\D}$ being symmetric
monoidal. By~\cite[8.8]{GrothPontoShulman-additivity} then, $\D$ being
closed monoidal corresponds to $\overline{\D}$ being closed symmetric
monoidal. In particular, \cite[9.13]{GrothPontoShulman-additivity}
establishes that if ${\cal M}$ is a symmetric monoidal cofibrantly
generated model category then the induced derivator $\D^{{\cal M}}$ is
a closed monoidal, strong derivator (of type $\mathbf{Cat}$).

Again, if $\D$ is a (closed) monoidal derivator, then so is $\D_{J}$
for any $J\in\mathbf{Dia}_{0}$.

\subsection{}\label{sec:linder}
A few words on linear structures on derivators
(see~\cite[section~3]{Groth-monder} for the details). An
\emph{additive derivator} is a derivator $\D$ such that $\D(\pt)$ is
an additive category. It follows that $\D(I), u^{*}, u_{*}, u_{!}$
are additive for all $I\in\mathbf{Dia}_{0}$,
$u\in\mathbf{Dia}_{1}$. We define $R_{\D}$ to be the unital ring
$\D(\pt)\left(\one,\one\right)$.

If $\D$ is additive and monoidal, then $R_{\D}$ is a commutative ring
and $\D(I)$ is canonically endowed with an $R_{\D}$-linear structure
for any $I\in\mathbf{Dia}_{0}$, making $u^{*}, u_{*}, u_{!}$ all
$R_{\D}$-linear functors, $u\in\mathbf{Dia}_{1}$. Given $f\in
\D(I)\left(A,B\right)$, $\lambda\in R_{\D}$, $\lambda f$ is defined by
\begin{equation*}
  A\xleftarrow{\sim} p_{I}^{*}\one\otimes A\xrightarrow{p_{I}^{*}\lambda\otimes f}p_{I}^{*}\one\otimes B\xrightarrow{\sim} B.
\end{equation*}

\subsection{}\label{sec:trder} We now recall the notion of a
triangulated derivator. Let $\boxempty$ be the partially ordered set
considered as a category:
\begin{equation*}
  \xymatrix{(1,1)&(0,1)\ar[l]\\
    (1,0)\ar[u]&(0,0)\ar[u]\ar[l]\mathrlap{,}}
\end{equation*}
and $\smalldcarre$ the full subcategory defined by the complement of
$(1,1)$. Thus there are two canonical embeddings
$i_{\smalldcarre}:\smalldcarre\to\boxempty$ and
$i_{\smallucarre}:\smallucarre\to\boxempty$. We say that an object
$A\in\D(\boxempty)_{0}$ is cartesian (resp.\ cocartesian) if the unit
\begin{equation*}
  A\to i_{\smalldcarre*}i_{\smalldcarre}^{*}A\qquad (\text{resp.\ 
  the counit } i_{\smallucarre!}i_{\smallucarre}^{*}A\to A)
\end{equation*}
is an isomorphism.

A \emph{triangulated derivator} is a strong derivator $\D$ such that
$\D(\pt)$ is pointed and objects in $\D(\carre)$ are cartesian if and
only if they are cocartesian. If ${\cal M}$ is a stable model
category, then the derivator $\D^{{\cal M}}$ associated to ${\cal M}$
is triangulated. Also, if $\D$ is a triangulated derivator then so is
$\D_{J}$ for any $J\in\mathbf{Dia}_{0}$. The name comes from the fact
that any triangulated derivator factors canonically through the
forgetful functor $\mathbf{TrCAT}\to\mathbf{CAT}$ from triangulated
categories to $\mathbf{CAT}$. In particular, every triangulated
derivator is additive. This result is due to Georges Maltsiniotis
(\cite[Théorème~1]{maltsiniotis-k-theory-derivator}; see
also~\cite[4.15, 4.19]{groth_pt_stable_derivator}), and the
triangulated structure is given explicitly. We will need the
description of it on $\D(\pt)$. (The description on $\D(I)$ can then
be deduced by replacing $\D$ by $\D_{I}$.) Thus given an object
$A\in\D(\pt)_{0}$ one defines canonically an object in $\D(\boxempty)$
with underlying diagram
\begin{equation*}
  \xymatrix{A\ar[r]\ar[d]&0\ar[d]\\
    0\ar[r]&\Sigma A\mathrlap{,}}
\end{equation*}
some $\Sigma A\in\D(\pt)_{0}$, as $i_{\smallucarre!}(1,1)_{*}A$. Then
we can define the suspension functor $\Sigma:\D(\pt)\to\D(\pt)$ as
$(0,0)^{*}i_{\smallucarre!}(1,1)_{*}$. Moreover, if we denote by
$\boxbar$ the partial order considered as a category
\begin{equation*}
  \xymatrix{(2,1)&\ar[l](1,1)&(0,1)\ar[l]\\
    (2,0)\ar[u]&\ar[l](1,0)\ar[u]&(0,0)\ar[u]\ar[l]\mathrlap{,}}
\end{equation*}
there are three canonical embeddings $i:\boxempty\to\boxbar$ and we
say that an object $A\in\D(\boxbar)_{0}$ is a triangle if
$A_{(2,0)}\cong A_{(0,1)}\cong 0$ and $i^{*}A$ is (co)cartesian for
all three embeddings. It then follows that one has a canonical
isomorphism $A_{(0,0)}\cong\Sigma A_{(2,1)}$ (see the proof of
Lemma~\ref{lem:mon-triangulated} below) and therefore a triangle in
$\D(\pt)$:
\begin{equation}\label{eq:d.t.}
  A_{(2,1)}\to A_{(1,1)}\to A_{(1,0)}\to \Sigma A_{(2,1)}.
\end{equation}
The distinguished triangles are those isomorphic to one of the
form of~(\ref{eq:d.t.}).

\subsection{}
\label{sec:trmonder}We are also interested in some aspects of the
interplay between monoidal and triangulated structures on derivators.

\begin{dfi}
  A \emph{(closed) monoidal triangulated derivator} is a (closed)
  monoidal and triangulated derivator.
\end{dfi}

Under the correspondence $\D\leftrightsquigarrow \overline{\D}$ above,
a closed monoidal triangulated derivator of type $\mathbf{Cat}$
corresponds to a ``closed symmetric monoidal, strong, stable
derivator'' in~\cite{GrothPontoShulman-additivity}. Translating the
results in~\cite{GrothPontoShulman-additivity} back to our setting we
see that every such derivator factors canonically through
$\mathbf{ClMonTrCAT}$, the 2-category of closed monoidal categories
with a ``compatible'' triangulation (in the sense
of~\cite{may_additivity}), such that the following diagram commutes:
\begin{equation*}
  \xymatrix@R=10pt{&&\mathbf{TrCAT}\ar[rd]\\
    \mathbf{Dia}^{\op,\op}\ar[r]&\mathbf{ClMonTrCAT}\ar[ru]\ar[rd]&&\mathbf{CAT}\mathrlap{.}\\
    &&\mathbf{ClMonCAT}\ar[ur]}
\end{equation*}
Here, it is understood that following the path on the upper part of
the diagram yields the canonical factorization of the triangulated
derivator, while the path through the lower part yields the
factorization of the monoidal prederivator. All we will need from this
statement is the following lemma.
\begin{lem}[{\cite[4.1, 4.8]{GrothPontoShulman-additivity}}]\label{lem:mon-triangulated}
  Let $\D$ be a monoidal triangulated derivator and
  $I\in\mathbf{Dia}_{0}$. Then the monoidal product
  $\otimes:\D(I)\times\D(I)\to\D(I)$ is canonically triangulated in
  both variables.
\end{lem}
\begin{proof} 
  First of all, replacing $\D$ by $\D_{I}$ we reduce to the case
  $I=\pt$. Moreover, by symmetry of the monoidal product we may fix
  $B\in\D(\pt)_{0}$ and only prove $\plho\otimes B$ to be
  triangulated. Then the condition that the projection morphism
  $p_{J!}(A\otimes p_{J}^{*}B)\to p_{J!}A\otimes B$ be invertible for
  all $A\in \D(J)_{0}$ in the case of a finite discrete category $J$
  says precisely that $\plho\otimes B$ is additive.

  The following claim is also a consequence of our definition of a
  monoidal derivator:
  \begin{itemize}
  \item[(*)] Let $A\in\D(\boxempty)_{0}$ be a cocartesian object. Then
    also $A\boxtimes B$ is cocartesian.
  \end{itemize}
  Indeed, this follows from the following factorization of the counit
  morphism:
  \begin{align*}
    i_{\smallucarre!}i_{\smallucarre}^{*}(A\boxtimes
    B)&\xrightarrow{\sim}
    i_{\smallucarre!}(i_{\smallucarre}^{*}A\boxtimes B)\\
    &\xrightarrow[\sim]{(\ref{eq:mon-der-external})}i_{\smallucarre!}i_{\smallucarre}^{*}A\boxtimes
    B\\
    &\xrightarrow[\sim]{\mathrm{adj}}A\boxtimes B.
  \end{align*}

  Now let $A\in\D(\pt)_{0}$ be an arbitrary object and consider
  $C=i_{\smallucarre!}(1,1)_{*}A\in\D(\boxempty)_{0}$. Since
  $i_{\smallucarre!}$ is fully faithful (this is an easy computation;
  see \cite[7.1]{cisinski-neeman}), $C$ is cocartesian, and by (*),
  this is also true of $C\boxtimes B$. Moreover $(C\boxtimes
  B)_{(1,0)}\cong C_{(1,0)}\otimes B\cong 0$ and, similarly,
  $(C\boxtimes B)_{(0,1)}\cong 0$. It follows from the following claim
  (**) that $\Sigma(A\otimes B)\cong \Sigma(C_{(1,1)}\otimes B)$ is
  isomorphic to $C_{(0,0)}\otimes B\cong \Sigma A\otimes B$, naturally
  in $A$.
  \begin{itemize}
  \item[(**)] Let $A\in\D(\boxempty)_{0}$ be a cocartesian object with
    $A_{(1,0)}\cong A_{(0,1)}\cong 0$. Then there is a canonical
    isomorphism $\Sigma (A_{(1,1)})\cong A_{(0,0)}$, natural in $A$.
  \end{itemize}
  The condition that those fibers vanish implies that the counit of
  adjunction $i_{\smallucarre}^{*}A\to
  (1,1)_{*}(1,1)^{*}i_{\smallucarre}^{*}A$ is invertible (again, an
  easy computation, cf.~\cite[8.11]{cisinski-neeman}). But the left
  hand side becomes isomorphic to $A$ after applying
  $i_{\smallucarre!}$ by assumption, so we get the required
  isomorphism after applying $(0,0)^{*}i_{\smallucarre!}$.

  Now let $D$ be a distinguished triangle in $\D(\pt)$, \ie{}
  $D$ is associated to a triangle
  $A\in\D(\boxbar)_{0}$. Essentially by (*), $A\boxtimes B$ is again a
  triangle, and essentially by (**), the distinguished triangle
  associated to $A\boxtimes B$ is isomorphic to $D\otimes B$.
\end{proof}

\subsection{}\label{par:basic-diagram}For $I$ an object of
$\mathbf{Cat}$, throughout the article we fix the notation as in the
following diagram where both squares are pullback squares:
\begin{equation}\label{eq:pullbackdiagram}\tag{$\Delta_{I}$}\hypertarget{delta-diagram}{}
  \xymatrix{\disc{I}\ar[r]^-{r_{2}}\ar[d]_{r_{1}}&\tw{I}^{\op}\ar[d]^{q_{2}}\\
    \tw{I}\ar[r]^-{q_{1}}&I^{\op}\times
    I\ar[r]^-{p_{2}}\ar[d]_{p_{1}}&I\ar[d]^{p_{I}}\\
    &I^{\op}\ar[r]^{p_{I^{\op}}}&\pt\mathrlap{.}}
\end{equation}
Explicitly, the objects of $\disc{I}$
are pairs of arrows in $I$ of the form
\begin{equation*}
  \xymatrix{i\ar@<.5ex>[r]&i'\ar@<.5ex>[l]\mathrlap{,}}
\end{equation*}
and morphisms from this object to
$\xymatrix{j\ar@<.5ex>[r]&j'\ar@<.5ex>[l]}$ are pairs of morphisms
$(i\leftarrow j, i'\to j')$ rendering the following two squares
commutative:
\begin{align*}
  \xymatrix{i\ar[r]&\ar[d]i'\\
    j\ar[r]\ar[u]&j'\mathrlap{,}}&&  \xymatrix{i&\ar[d]i'\ar[l]\\
    j\ar[u]&\ar[l]j'\mathrlap{.}}
\end{align*}
Note that if $I$ lies in some diagram category then so does the
whole diagram~(\ref{eq:pullbackdiagram}).
\section{External hom}
\label{sec:ehom}

Fix a closed monoidal derivator $\D$ of type $\mathbf{Dia}$. As
explained in the introduction we would like to define an ``external
hom'' functor which will play an essential role in the definition of
the trace. It should behave with respect to the external product as
does the internal hom with respect to the internal product (\ie{} the
monoidal structure). As a first indication of its nature, the external
hom of $A\in \D(I)_{0}$ and $B\in\D(J)_{0}$ should be an object of
$\D(I^{\op}\times J)$, denoted by $\ehm{A}{B}$. Additionally, we would
like the fibers of $\ehm{A}{B}$ to compute the internal hom of the
fibers of $A$ and $B$, because fiberwise dualizability should imply
dualizability with respect to $\ehm{\plho}{\plho}$; moreover, $[A,B]$
should be expressible in terms of $\ehm{A}{B}$ in the case
$I=J$. These and other desired properties of the external product are
satisfied by the following construction which is due to Joseph Ayoub.

Given small categories $I$ and $J$ in $\mathbf{Dia}$, we fix the
following notation, all functors being the obvious ones:
\begin{equation*}\tag{\ensuremath{\Pi_{I,J}}}\label{eq:ehomdiagram}\hypertarget{pi-diagram}{}
  \xymatrix{\tw{I}\times J\ar[r]^-{r}\ar[rd]^{q}\ar[d]_{p}&J\\
  I^{\op}\times J&I\mathrlap{.}}
\end{equation*}
For any $A$ in $\D(I)_{0}$ and
$B$ in $\D(J)_{0}$ set
\begin{equation*}
  \ehm{A}{B}:=p_{*}[q^{*}A,r^{*}B].
\end{equation*}
This defines a bifunctor
\begin{equation*}
  \ehm{\plho}{\plho}:\D(I)^{\op}\times \D(J)\to \D(I^{\op}\times J),
\end{equation*}
whose properties we are going to list now. For the proofs the reader
is referred to appendix~\ref{sec:app1}.

\subsection*{Naturality}

For functors $u:I'\to I$ and $v:J'\to J$ in $\mathbf{Dia}$ there is an
invertible morphism
\begin{equation*}
  \Psi:(u^{\op}\times v)^{*}\ehm{A}{B}\xrightarrow{\sim}\ehm{u^{*}A}{v^{*}B},
\end{equation*}
natural in $A\in \D(I)_{0}$, $B\in \D(J)_{0}$. Moreover, $\Psi$
behaves well with respect to functors and natural transformations in
$\mathbf{Dia}$. In other words, $\ehm{\plho}{\plho}$ defines a
1-morphism in $\Fun(\mathbf{Dia}^{\op}\times
\mathbf{Dia}^{\op,\op}, \mathbf{CAT})$ from
$\D(\plho)^{\op}\times\D(\plho)$ to $\D(\plho^{\op}\times\plho)$ with
invertible 2-cell components (\ie{} a pseudonatural transformation).

\subsection*{Internal hom}
In the case $I=J$ there is an invertible morphism
\begin{equation*}
  \Theta:[A,B]\xrightarrow{\sim}
  p_{2*}q_{2*}q_{2}^{*}\ehm{A}{B}\qquad\text{(with the notation
    of~(\ref{eq:pullbackdiagram})),}
\end{equation*}
natural in $A$ and $B\in\D(I)_{0}$. Moreover, for any functor $u:I'\to
I$ in $\mathbf{Dia}$, the canonical arrow $u^{*}[A,B]\to
[u^{*}A,u^{*}B]$ is compatible with $\Psi$ via $\Theta$. In other
words, $\Theta$ defines an invertible 2-morphism in
$\Fun((\mathbf{Dia}_{\leq 1})^{\op},\mathbf{CAT})$ between
1-morphisms from $\D(\plho)^{\op}\times\D(\plho)$ to
$\D(\plho)$. Here, $\mathbf{Dia}_{\leq 1}$ is the 2-subcategory of
$\mathbf{Dia}$ obtained by removing all non-identity 2-cells.

\subsection*{External product}
Given categories $I_{(k)}$, $k=1,\ldots,4$, in $\mathbf{Dia}$, $A_{k}\in
\D(I_{(k)})_{0}$, there is a morphism
\begin{equation*}
  \Xi:\ehm{A_{1}}{A_{2}}\boxtimes \ehm{A_{3}}{A_{4}}\to \tau^{*}\ehm{A_{1}\boxtimes A_{3}}{A_{2}\boxtimes A_{4}},
\end{equation*}
natural in all four arguments, where
\begin{equation*}
  \tau:I_{(1)}^{\op}\times
  I_{(2)}\times I_{(3)}^{\op}\times I_{(4)}\to I_{(1)}^{\op}\times
  I_{(3)}^{\op}\times I_{(2)}\times I_{(4)}
\end{equation*}
interchanges the two categories in the middle. Moreover, $\Xi$ is
compatible with $\Psi$ and~(\ref{eq:eprod-nat}). In other words, it
defines a 2-morphism in
$\Fun(\mathbf{Dia}^{\op}\times\mathbf{Dia}^{\op,\op}\times\mathbf{Dia}^{\op}\times\mathbf{Dia}^{\op,\op},
\mathbf{CAT})$ between 1-morphisms from
$\D(\plho)^{\op}\times\D(\plho)\times\D(\plho)^{\op}\times\D(\plho)$
to $\D(\plho^{\op}\times\plho\times\plho^{\op}\times\plho)$.

\subsection*{Adjunction}
Given three categories in $\mathbf{Dia}$, there is an invertible
morphism
\begin{equation*}
  \Omega: \ehm{A}{\ehm{B}{C}}\xrightarrow{\sim}\ehm{A\boxtimes B}{C},
\end{equation*}
natural in all three arguments. Moreover, $\Omega$ is compatible with
$\Psi$ and~(\ref{eq:eprod-nat}). In other words, it defines an
invertible 2-morphism in
$\Fun(\mathbf{Dia}^{\op}\times\mathbf{Dia}^{\op}\times\mathbf{Dia}^{\op,\op},\mathbf{CAT})$
between 1-morphisms from
$\D(\plho)^{\op}\times\D(\plho)^{\op}\times\D(\plho)$ to
$\D(\plho^{\op}\times\plho^{\op}\times\plho)$.

\subsection*{Biduality}
For fixed $B\in \D(\pt)_{0}$, there is a morphism
\begin{equation*}
  \Upsilon:A\xrightarrow{} \ehm{\ehm{A}{B}}{B},
\end{equation*}
natural in $A\in\D(I)_{0}$. Moreover, $\Upsilon$ defines a 2-morphism
in $\Fun(\mathbf{Dia}^{\op,\op}, \mathbf{CAT})$ between
1-endomorphisms of $\D$.

\subsection*{Normalization}
Given $J\in\mathbf{Dia}_{0}$, there is an invertible morphism
\begin{align*}
  \Lambda:[p_{J}^{*}A,B]\xrightarrow{\sim}\ehm{A}{B},
\end{align*}
natural in $A\in\D(\pt)_{0}$ and $B\in \D(J)_{0}$. Again, $\Lambda$ is
compatible with $v^{*}$ for any $v:J'\to J$ in $\mathbf{Dia}$,
therefore it defines an invertible 2-morphism in
$\Fun(\mathbf{Dia}^{\op,\op}, \mathbf{CAT})$ between
1-morphisms from $\D(\pt)^{\op}\times\D(\plho)$ to
$\D(\plho)$. Moreover under this identification, all the morphisms in
the statements of the previous properties reduce to the canonical
morphisms in closed monoidal categories. (These morphisms are made
explicit in appendix~\ref{sec:app1}; see
p.~\pageref{sec:ehom-properties-proof-normalization}.)

\section{Definition of the trace}\label{sec:trace-dfi}
Recall that in a closed monoidal category ${\cal C}$, an object $A$ is
called \emph{dualizable} (sometimes also \emph{strongly dualizable})
if the canonical morphism
\begin{equation}\label{eq:dualizability}
  [A,\one]\otimes B\to [A,\one\otimes B]
\end{equation}
is invertible for all $B\in\mathcal{C}_{0}$, and in this case one
defines a \emph{coevaluation}
\begin{equation}\label{eq:coev-classical}
  \mathrm{coev}:\one\xrightarrow{\mathrm{adj}}[A,\one\otimes A]\xleftarrow{\sim} [A,\one]\otimes
  A.
\end{equation}
It has the characterizing property that the following diagram commutes
(see~\cite[1.4]{lewis-may}):
\begin{equation}\label{eq:coev-ev}
  \xymatrix@C=4pc{\ar[d]_{\sim}[A,\one]\otimes
    A\ar[r]^-{\mathrm{ev}}&\one\ar[d]^{\sim}\\
    [[A,\one]\otimes A,\one]\ar[r]_-{[\mathrm{coev},\one]}&[\one,\one]\mathrlap{.}}
\end{equation}
Here the vertical morphism on the left is defined as the composition
\begin{equation}
  \label{eq:delta}
  [A,\one]\otimes A\to [A,\one]\otimes [[A,\one],\one]\to [A\otimes
  [A,\one],\one\otimes\one]\xrightarrow{\sim}[[A,\one]\otimes A,\one],
\end{equation}
while the one on the right is
\begin{equation*}
  \one\xrightarrow{\mathrm{adj}}[\one,\one\otimes\one]\xrightarrow{\sim}[\one,\one].
\end{equation*}
$[A,\one]$ is called the \emph{dual} of $A$, and is often denoted by
$A^{\ast}$. Dualizability of $A$ implies that the canonical morphism
\begin{equation}
  \label{eq:bidual-classical}
  A\to (A^{*})^{*}
\end{equation}
is invertible.

For dualizable $A$, the \emph{trace map}
\begin{equation*}
  \mathrm{tr}:\mathcal{C}(A,A)\to \mathcal{C}(\one, \one)
\end{equation*}
sends an endomorphism $f$ to the composition
\begin{equation*}
  \one\xrightarrow{\mathrm{coev}} A^{\ast}\otimes
  A\xrightarrow{\catid\otimes f}A^{\ast}\otimes A\xrightarrow{\mathrm{ev}}\one.
\end{equation*}
More generally, \cite{Maltsiniotis1995} and \cite{jsv1996traced}
independently introduced a \emph{(twisted) trace map} for any $S$ and
$T$ in ${\cal C}$ ($A$ still assumed dualizable),
\begin{equation*}
  \mathrm{tr}:\mathcal{C}(A\otimes S,A\otimes T)\to \mathcal{C}(S,T),
\end{equation*}
which sends a ``twisted endomorphism'' $f$ to the composition
\begin{equation*}
  S\xleftarrow{\sim} \one\otimes S\xrightarrow{\mathrm{coev}\otimes \catid} A^{\ast}\otimes
  A\otimes S\xrightarrow{\catid\otimes f}A^{\ast}\otimes A\otimes
  T\xrightarrow{\mathrm{ev}\otimes \catid}\one\otimes T\xrightarrow{\sim} T.
\end{equation*}

We will mimic this definition in our derivator setting. So fix a
closed monoidal derivator $\D$ of type $\mathbf{Dia}$. First of all,
here is our translation of dualizability:

\begin{dfi}
  Let $I\in \mathbf{Dia}_{0}$, $A\in \D(I)_{0}$. We say that $A$ is
  \emph{fiberwise dualizable} if $A_{i}$ is dualizable for all $i\in
  I_{0}$. The \emph{dual} of $A$ is defined to be
  $\ehm{A}{\one_{\D(\pt)}}\in\D(I^{\op})_{0}$, also denoted by $A^{\vee}$.
\end{dfi}

Let $I$ and $A$ as in the definition, $A$ fiberwise dualizable. Then,
as was the case for dualizable objects in closed monoidal categories,
the morphisms corresponding to~(\ref{eq:dualizability})
and~(\ref{eq:bidual-classical}) are invertible (for any
$B\in\D(I)_{0}$):
\begin{align}
  \label{eq:application}
  A^{\vee}\boxtimes
  B&\cong\ehm{A}{\one}\boxtimes[p_{I}^{*}\one,B]\xrightarrow[\sim]{\Lambda}\ehm{A}{\one}\boxtimes\ehm{\one}{B}\xrightarrow[\sim]{\Xi}\ehm{A\boxtimes\one}{\one\boxtimes
    B}\cong\ehm{A}{B},\\
  \Upsilon&:A\xrightarrow{\sim} (A^{\vee})^{\vee}.
\end{align}
This follows from the naturality and the normalization properties of
the external hom. We now go about defining a coevaluation and an
evaluation morphism. This will rely on the results of the previous
section.

Using the relation between internal and external hom, we can consider
the composition
\begin{align*}
  \one_{\D(I)}\xrightarrow{\mathrm{adj}} [A,\one\otimes
  A]\xrightarrow{\sim}[A,A]\xrightarrow[\sim]{\Theta} p_{2*}q_{2*}q_{2}^{*}\ehm{A}{A}
\end{align*}
and, by adjunction, we obtain
\begin{align}\label{eq:coev-new}
  \mathrm{coev}:q_{2!}\one\to\ehm{A}{A}\xleftarrow[\sim]{(\ref{eq:application})}A^{\vee}\boxtimes A.
\end{align}

Next, inspired by~(\ref{eq:coev-ev}), we define the evaluation
morphism to be simply the dual of the coevaluation morphism. For this,
notice that $A$ being fiberwise dualizable implies that also
$A^{\vee}$ is. Hence there is an analogue of~\eqref{eq:delta}:
\begin{equation}\label{eq:ev-new-pre}
  A\boxtimes
  A^{\vee}\xrightarrow[\sim]{\Upsilon}(A^{\vee})^{\vee}\boxtimes
  A^{\vee}\xrightarrow[\sim]{\text{(\ref{eq:application})}}\ehm{A^{\vee}}{A^{\vee}}\xrightarrow[\sim]{\Omega}\ehm{A^{\vee}\boxtimes
  A}{\one_{\D(\pt)}}.
\end{equation}
Denote by $\mu:I\times I^{\op}\to I^{\op}\times I$ the canonical
isomorphism. Then we define
\begin{align*}
  \xymatrix@R=4pt@C=20pt{\mathllap{\mathrm{ev}:A^{\vee}\boxtimes A}\ar[r]^{\sim}&\mathrlap{\mu_{*}(A\boxtimes A^{\vee})}\\
  \ar[r]^{\eqref{eq:ev-new-pre}}_{\sim}&\mathrlap{\mu_{*}\ehm{A^{\vee}\boxtimes A}{\one}}\\
  \ar[r]^{\ehm{\mathrm{coev}}{\one}}&\mathrlap{\mu_{*}\ehm{q_{2!}\one}{\one}}\\
  \ar[r]^{\overline{\Psi}}& \mathrlap{\mu_{*}(q_{2})^{\op}_{*}\ehm{\one}{\one}}\\
  \ar[r]^{\sim}& \mathrlap{q_{1*}\one.}}
\end{align*}
Here, $\overline{\Psi}$ is obtained by adjunction from
$\Psi$:
\begin{align*}
  \xymatrix@R=4pt@C=15pt{\mathllap{\overline{\Psi}:\ehm{q_{2!}\one}{\one}}\ar[r]^{\mathrm{adj}}&\mathrlap{q_{2*}^{\op}q_{2}^{\op*}\ehm{q_{2!}\one}{\one}}\\
  \ar[r]_{\sim}^{\Psi}& \mathrlap{q_{2*}^{\op}\ehm{q_{2}^{*}q_{2!}\one}{\one}}\\
  \ar[r]^{\mathrm{adj}}&\mathrlap{q_{2*}^{\op}\ehm{\one}{\one}.}}
\end{align*}
It follows immediately that the following diagram commutes for any
$u:I'\to I$ in $\mathbf{Dia}$:
\begin{equation}
  \label{eq:ehom-pw}
  \xymatrix{(u\times
    u^{\op})^{*}\ehm{q_{2!}\one}{\one}\ar[r]\ar[d]_{\overline{\Psi}}&\ehm{q'_{2!}\tw{u}^{\op*}\one}{\one}\ar[d]^{\overline{\Psi}}\\
    (u\times u^{\op})^{*}q_{2*}^{\op}\ehm{\one}{\one}\ar[r]&q_{2*}'^{\op}\ehm{\tw{u}^{\op*}\one}{\one}\mathrlap{.}}
\end{equation}
In the sequel we will sometimes denote by the same symbol
$\overline{\Psi}$ other morphisms obtained by adjunction from $\Psi$
in a similar way. It is hoped that this will not cause any confusion.

Finally we can put all the pieces together and define the trace:
\begin{dfi}
  Let $I\in\mathbf{Dia}_{0}$, $A\in \D(I)_{0}$ fiberwise dualizable, and
  $S,T\in \D(I)_{0}$ arbitrary. Then we define the \emph{(twisted) trace map}
  \begin{equation*}
    \mathrm{Tr}:\D(I)\left(A\otimes S,A\otimes T\right)\to\D(I^{\op}\times I)\left(q_{2!}\one\otimes
    p_{2}^{*}S,q_{1*}\one\otimes p_{2}^{*}T\right)
  \end{equation*}
  as the association which sends a twisted endomorphism $f$ to the
  composition
  \begin{equation*}
   \xymatrix@C=40pt{q_{2!}\one\otimes
    p_{2}^{*}S\ar[r]^-{\mathrm{coev}\otimes\catid}&(A^{\vee}\boxtimes A) \otimes
    p_{2}^{*}S\ar[r]^-{\sim}& A^{\vee}\boxtimes (A\otimes
    S)\ar[d]^{\catid\boxtimes f}\\
    q_{1*}\one\otimes p_{2}^{*}T& (A^{\vee}\boxtimes A) \otimes
    p_{2}^{*}T\ar[l]^-{\mathrm{ev}\otimes\catid}&\ar[l]^-{\sim}A^{\vee}\boxtimes (A\otimes
    T)}
  \end{equation*}
  called the \emph{(twisted) trace of $f$}.
\end{dfi}

\begin{rem}\label{rem:components}
  Although defined in this generality, we will be interested mainly in
  traces of endomorphisms twisted by ``constant'' objects, \ie{}
  coming from objects in $\D(\pt)$. In this case
  ($S,T\in\D(\pt)_{0}$), the trace map is an association
  \begin{equation*}
    \D(I)
    \left(A\boxtimes S,A\boxtimes T\right)\to \D(I^{\op}\times
    I)
    \left(q_{2!}\one\boxtimes S,q_{1*}\one\boxtimes T\right).
  \end{equation*}

  Now, let $g$ be an element of the target of this map. It induces the
  composite
  \begin{equation*}
    q_{2!}\restr{S}{\tw{I}^{\op}}\xleftarrow{\sim}q_{2!}\one\boxtimes S\xrightarrow{g}q_{1*}\one
\boxtimes T\to q_{1*}\restr{T}{\tw{I}}
  \end{equation*}
  and by adjunction
  \begin{equation*}
    q_{2}^{*}\restr{S}{I^{\op}\times I}\to
    q_{2}^{*}q_{1*}q_{1}^{*}\restr{T}{I^{\op}\times
      I}\xrightarrow{\sim}
    r_{2*}r_{1}^{*}q_{1}^{*}\restr{T}{I^{\op}\times I}
  \end{equation*}
  or, by another adjunction, a morphism
  \begin{equation}\label{eq:input2}
    \restr{S}{\disc{I}}\to \restr{T}{\disc{I}}.
  \end{equation}
  Applying the functor $\mathrm{dia}_{\disc{I}}$ we obtain an element
  of
  \begin{equation*}
    \D(\pt)^{(\disc{I})^{\op}}\left(\mathrm{dia}_{\disc{I}}(\restr{S}{\disc{I}}),
    \mathrm{dia}_{\disc{I}}(\restr{T}{\disc{I}})\right)\cong\prod_{\pi_{0}(\disc{I})}\D(\pt)\left(S,T\right).
  \end{equation*}
  The component corresponding to $\gamma\in\pi_{0}(\disc{I})$ is called
  the \emph{$\gamma$-component of $g$}.
\end{rem}

\begin{lem}\label{lem:input_eq}
  Suppose that the following hypotheses are satisfied:
  \begin{enumerate}[label=(H\arabic*), series=hypotheses]
  \item\label{hyp:q_1} $q_{1*}\one\boxtimes A\to q_{1*}(\one\boxtimes
    A)$ is invertible for all $A\in\D(\pt)_{0}$;
  \item for each connected component $\gamma$ of $\disc{I}$, the
    functor $p_{\gamma}^{*}$ is fully faithful.\label{hyp:dia}
  \end{enumerate}
  Then the map
  \begin{equation*}
    \D(I^{\op}\times
    I)
    \left(q_{2!}\one\boxtimes S,q_{1*}\one\boxtimes T\right)\to\prod_{\pi_{0}(\disc{I})}\D(\pt)\left(S,T\right)
  \end{equation*}
  defined above is a bijection. In particular, any morphism
  $q_{2!}\one\boxtimes S\to q_{1*}\one\boxtimes T$ is uniquely
  determined by its $\gamma$-components, $\gamma\in\pi_{0}(\disc{I})$.
\end{lem}
\begin{proof}
  \ref{hyp:q_1} implies that the morphism $g:q_{2!}\one\boxtimes
  S\to q_{1*}\one\boxtimes T$ in the remark above can equivalently be
  described by~(\ref{eq:input2}). Moreover, the following square commutes:
  \begin{equation*}
    \xymatrix@C=45pt{\D(\disc{I})(\restr{S}{\disc{I}},\restr{T}{\disc{I}})\ar[r]^-{\mathrm{dia}_{\disc{I}}}\ar[d]_{\sim}&\D(\pt)^{\disc{I}^{\op}}(S_{\mathrm{cst}},T_{\mathrm{cst}})\ar[d]^{\sim}\\
    \displaystyle\prod_{\gamma\in\pi_{0}(\disc{I})}\D(\gamma)(\restr{S}{\gamma},\restr{T}{\gamma})&\displaystyle\prod_{\gamma\in\pi_{0}(\disc{I})}\D(\pt)(S,T)\ar[l]_-{(p_{\gamma}^{*})_{\gamma}}\mathrlap{.}}
  \end{equation*}
  Here, the left vertical arrow is invertible
  by~\ref{der:coprod}. \ref{hyp:dia} now implies that the horizontal
  arrow on the top is a bijection.
\end{proof}

  In particular we see that in favorable cases 
  (and these are the only ones we will have much to say about) the
  seemingly complicated twisted trace is encoded simply by a family of
  morphisms over the terminal category. The goal of the following
  section is to determine these morphisms.

\section{Functoriality of the trace}
\label{sec:trace-nat}
Our immediate goal is to describe the components $S\to
T\in\D(\pt)_{1}$ associated to the trace of a (twisted) endomorphism
of a fiberwise dualizable object as explained in the previous
section. However, we take the opportunity to establish a more general
functoriality property of the trace
(Proposition~\ref{pro:functoriality}).
Our immediate goal will be achieved as a corollary to this result.

Throughout this section we fix a category $I\in\mathbf{Dia}_{0}$. An
object of $\disc{I}$ is a pair of arrows
\begin{equation}\label{eq:dI-object}
  \xymatrix{i\ar@<.5ex>[r]^{h_{1}}&j\ar@<.5ex>[l]^{h_{2}}}
\end{equation}
in $I$ (cf.~\ref{par:basic-diagram}). There is always a
morphism in $\disc{I}$ from an object of the form
\begin{equation*}
  (i,h):\qquad \xymatrix{i\ar@<.5ex>[r]^{\catid_{i}}&i\ar@<.5ex>[l]^{h}}
\end{equation*}
to (\ref{eq:dI-object}), given by the pair of arrows
$(\catid_{i},h_{1})$ if $h=h_{2}h_{1}$. Hence we can take some of the
$(i,h)$ as representatives for $\pi_{0}(\disc{I})$ and it is sufficient
to describe the component $S\to T$ corresponding to $(i,h)$. This
motivates the following more general functoriality statement.

Let $u:I'\to I$ be a functor, $\eta:u\to u$ a natural transformation in
$\mathbf{Dia}$; consider the basic
diagram~(\ref{eq:pullbackdiagram}). Notice that this diagram is
functorial in $I$ hence there is a canonical morphism of diagrams
\hyperlink{delta-diagram}{$(\Delta_{I'})$}$\to$\hyperlink{delta-diagram}{$(\Delta_{I})$}
and we will use the convention that the arrows in
\hyperlink{delta-diagram}{$(\Delta_{I'})$} will be distinguished from
their $I$-counterparts by being decorated with a prime.
\begin{dfi}
  Let $S$, $T\in \D(I)_{0}$. Define a pullback map
  \begin{multline*}
    (u,\eta)^{*}:\D(I^{\op}\times I)\left(q_{2!}\one\otimes
    p_{2}^{*}S,q_{1*}\one\otimes p_{2}^{*}T\right)\longrightarrow\\
    \D(I^{\prime\op}\times I^{\prime})\left(q'_{2!}\one\otimes
    p_{2}'^{*}u^{*}S,q'_{1*}\one\otimes p_{2}'^{*}u^{*}T\right)
  \end{multline*}
  by sending a morphism $g$ to the composition
  \begin{multline*}
    q'_{2!}\tw{u}^{\op*}\one\otimes p_{2}'^{*}u^{*}S\to
    (u^{\op}\times u)^{*}(q_{2!}\one\otimes
    p_{2}^{*}S)\xrightarrow{g}(u^{\op}\times u)^{*}(q_{1*}\one\otimes
    p_{2}^{*}T)\\
    \xrightarrow{(\Id\times \eta)^{*}}(u^{\op}\times u)^{*}(q_{1*}\one\otimes
    p_{2}^{*}T)\to q'_{1*}\tw{u}^{*}\one\otimes p_{2}'^{*}u^{*}T.
  \end{multline*}
\end{dfi}

\begin{pro}\label{pro:functoriality}
  Let $u$, $\eta$, $S$, $T$ as above, assume $A\in \D(I)_{0}$ is
  fiberwise dualizable. For any $f:A\otimes S\to A\otimes T$, we have
  \begin{equation}\label{eq:pro-functoriality}
    (u,\eta)^{*}\mathrm{Tr}(f)=\mathrm{Tr}(\eta^{*}\circ u^{*}f),
  \end{equation}
  where $\eta^{*}\circ u^{*}f$ is any of the two paths from the top
  left to the bottom right in the following commutative square:
  \begin{equation*}
    \xymatrix{u^{*}A\otimes
      u^{*}S\ar[r]^{u^{*}f}\ar[d]_{\eta^{*}_{A\otimes
          S}}&u^{*}A\otimes u^{*}T\ar[d]^{\eta^{*}_{A\otimes
          T}}\\
      u^{*}A\otimes u^{*}S\ar[r]_{u^{*}f}&u^{*}A\otimes u^{*}T\mathrlap{.}}
  \end{equation*}
\end{pro}
\begin{proof}
  The two outer paths in the following diagram are exactly the two
  sides of~(\ref{eq:pro-functoriality}):
  \begin{equation*}
    \xymatrix{q'_{2!}\tw{u}^{\op*}\one\otimes p_{2}'^{*}u^{*}S\ar[r]\ar[d]_{\mathrm{coev}}&(u^{\op}\times
      u)^{*}(q_{2!}\one\otimes p_{2}^{*}S)\ar[d]^{\mathrm{coev}}\\
      (u^{*}A)^{\vee}\boxtimes u^{*}A\otimes
      p_{2}'^{*}u^{*}S\ar[d]_{\Id\boxtimes u^{*}f}&(u^{\op}\times
      u)^{*}(A^{\vee}\boxtimes A\otimes
      p_{2}^{*}S)\ar[l]_{\Psi}\ar[d]^{\Id\boxtimes f}\\
      (u^{*}A)^{\vee}\boxtimes u^{*}A\otimes
      p_{2}'^{*}u^{*}T\ar[d]_{\Id\boxtimes \eta^{*}}&(u^{\op}\times
      u)^{*}(A^{\vee}\boxtimes A\otimes
      p_{2}^{*}T)\ar[l]_{\Psi}\ar[d]^{\Id\times \eta^{*}}\\
      (u^{*}A)^{\vee}\boxtimes u^{*}A\otimes p_{2}'^{*}u^{*}T\ar[d]_{\sim}&(u^{\op}\times
      u)^{*}(A^{\vee}\boxtimes A\otimes p_{2}^{*}T)\ar[l]_{\Psi}\ar[d]^{\sim}\\
      \mu'_{*}\ehm{(u^{*}A)^{\vee}\boxtimes u^{*}A}{\one}\otimes p_{2}'^{*}u^{*}T\ar[d]_{\ehm{\mathrm{coev}}{\one}}&(u^{\op}\times
      u)^{*}(\mu_{*}\ehm{A^{\vee}\boxtimes A}{\one}\otimes p_{2}^{*}T)\ar[l]_{\Psi}\ar[d]^{\ehm{\mathrm{coev}}{\one}}\\
      \mu'_{*}\ehm{q'_{2!}\one}{\one}\otimes p_{2}'^{*}u^{*}T\ar[d]&(u^{\op}\times
      u)^{*}(\mu_{*}\ehm{q_{2!}\one}{\one}\otimes
      p_{2}^{*}T)\ar[l]_{\Psi}\ar[d]\\
      q'_{1*}\tw{u}^{*}\one\otimes
      p_{2}'^{*}u^{*}T&\ar[l](u^{\op}\times
      u)^{*}(q_{1*}\one\otimes p_{2}^{*}T)\mathrlap{.}}
  \end{equation*}
  Hence it suffices to prove the commutativity of this diagram. The
  second and third square clearly commute, the fourth and sixth square
  do so by the functoriality statements in section~\ref{sec:ehom} (use
  also~(\ref{eq:ehom-pw})). The fifth square commutes if the first does
  so we are left to show commutativity of the first one.

  By definition, $\mathrm{coev}$ is the composition
  \begin{equation*}
    q_{2!}\one\to\ehm{A}{A}\xleftarrow{\sim}A^{\vee}\boxtimes A
  \end{equation*}
  and we already know that the second arrow behaves well with respect
  to functors in $\mathbf{Dia}$. Thus it suffices to prove that the
  following diagram commutes:
  \begin{equation*}
    \xymatrix{q'_{2!}q_{2}^{\prime*}p_{2}^{\prime*}\one\ar[d]_{\sim}\ar[r]&q'_{2!}q_{2}^{\prime*}p_{2}^{\prime*}p_{2*}'q_{2*}'q_{2}'^{*}\ehm{u^{*}A}{u^{*}A}\ar[r]^-{\mathrm{adj}}&\ehm{u^{*}A}{u^{*}A}\\
      q'_{2!}q_{2}^{\prime*}p_{2}^{\prime*}u^{*}\one\ar[d]\ar[r]& \ar[d]\ar[u]^{\Psi}
      q'_{2!}q_{2}^{\prime*}p_{2}^{\prime*}u^{*}p_{2*}q_{2*}q_{2}^{*}\ehm{A}{A}\\
      \ar[r](u^{\op}\times
      u)^{*}q_{2!}q_{2}^{*}p_{2}^{*}\one&(u^{\op}\times
      u)^{*}q_{2!}q_{2}^{*}p_{2}^{*}p_{2*}q_{2*}q_{2}^{*}\ehm{A}{A}\ar[r]_-{\mathrm{adj}}&(u^{\op}\times u)^{*}\ehm{A}{A}\ar[uu]_{\Psi}\mathrlap{.}}
  \end{equation*}
  The top left square commutes by the internal hom property in
  section~\ref{sec:ehom}, the bottom left square clearly commutes, and
  the right rectangle is also easily seen to commute.
\end{proof}

Of course, in the Proposition we can take $u=i$ to be an object of
$I$, and $\eta$ to be the identity transformation. Denote the pullback
morphism $(i,\Idtransf)^{*}$ by $i^{*}$.
\begin{cor}\label{cor:trace-fiber}
  Let $i\in I_{0}$. For any $A,S,T\in\D(I)_{0}$, $A$ fiberwise
  dualizable, and for any $f:A\otimes S\to A\otimes T$, we have
  \begin{equation*}
    i^{*}\mathrm{Tr}(f)=\mathrm{tr}(f_{i})
  \end{equation*}
  modulo the obvious identifications.
\end{cor}
\begin{proof}
  By the proposition, $i^{*}\mathrm{Tr}(f)=\mathrm{Tr}(i^{*}f)$.  It
  remains to prove that in the case $I=\pt$, the maps $\mathrm{Tr}$
  and $\mathrm{tr}$ coincide. Thus assume $I=\pt$ and consider the
  following diagram:
  \begin{equation*}
    \xymatrix{\catid_{!}\one\otimes
      S\ar[d]_{\sim}\ar[r]^-{\mathrm{coev}}&A^{\vee}\otimes
      A\otimes S\ar[r]^{f}&A^{\vee}\otimes
      A\otimes
      T\ar[r]^-{\mathrm{ev}}&\catid_{*}\one\otimes
      T\\
      \one\otimes S\ar[r]_-{\mathrm{coev}}&A^{*}\otimes A\otimes
      S\ar[r]_{f}\ar[u]^{\sim}_{\Lambda}&A^{*}\otimes A\otimes
      T\ar[r]_-{\mathrm{ev}}\ar[u]_{\sim}^{\Lambda}&\one\otimes T\ar[u]_{\sim}\mathrlap{.}}
  \end{equation*}
  The composition of the top horizontal arrows is $\mathrm{Tr}(f)$
  while the composition of the bottom horizontal arrows is
  $\mathrm{tr}(f)$. The middle square clearly commutes. The left
  square commutes by the normalization property of the external hom,
  and commutativity of the right square can be deduced from this
  and~(\ref{eq:coev-ev}).
\end{proof}

Let us come back to the situation considered at the beginning of this
section. Here the proposition implies:

\begin{cor}\label{cor:trace-nat}
  Let $A\in\D(I)_{0}$ fiberwise dualizable, $S,T\in\D(\pt)_{0}$, $i\in
  I_{0}$, $h\in I(i,i)$, and $f:A\boxtimes S\to A\boxtimes T\in
  \D(I)_{1}$. Then, modulo the obvious identifications, the
  $(i,h)$-component of $\mathrm{Tr}(f)$ is $\mathrm{tr}(h^{*}\circ
  f_{i})$.
\end{cor}
\begin{proof}$h$ defines a natural transformation $i\to i$ and we have
  \begin{align*}
    (i,h)^{*}\mathrm{Tr}(f)&=\mathrm{Tr}(h^{*}\circ i^{*}f)&&\text{by the
  proposition above,}\\
&=\mathrm{tr}(h^{*}\circ i^{*}f)&&\text{by the
  previous corollary.}
  \end{align*}
  We need to prove that the left hand side computes the
  $(i,h)$-component. The pair $(\catid_{i},h)$ defines an arrow in
  $\tw{I}^{\op}$ from $h$ to $\catid_{i}$. The composition of the
  vertical arrows on the left of the following diagram is the
  $(i,h)$-component of $\mathrm{Tr}(f)$ while the composition of the
  vertical arrows on the right is $(i,h)^{*}\mathrm{Tr}(f)$:
  \begin{equation*}
    \xymatrix{(i,h)^{*}\restr{S}{\disc{I}}\ar[d]_{\mathrm{adj}}&h^{*}\restr{S}{\tw{I}^{\op}}\ar@{=}[l]\ar[d]_{\mathrm{adj}}&\catid_{i}^{*}\restr{S}{\tw{I}^{\op}}\ar@{=}[l]\ar[d]^{\mathrm{adj}}&\\
  (i,h)^{*}r_{2}^{*}q_{2}^{*}q_{2!}(\one\boxtimes
  S)\ar[d]_{\sim}&h^{*}q_{2}^{*}q_{2!}(\one\boxtimes S)\ar@{=}[l]\ar[d]_{\sim}&\ar[l]_-{(\catid_{i},h)^{*}}\catid_{i}^{*}q_{2}^{*}q_{2!}(\one\boxtimes S)\ar[d]^{\sim}\\
  (i,h)^{*}r_{2}^{*}q_{2}^{*}(q_{2!}\one\boxtimes
  S)\ar[d]_{\mathrm{Tr}(f)}&h^{*}q_{2}^{*}(q_{2!}\one\boxtimes S)\ar@{=}[l]\ar[d]_{\mathrm{Tr}(f)}&\ar[l]_-{(\catid_{i},h)^{*}}\catid_{i}^{*}q_{2}^{*}(q_{2!}\one\boxtimes S)\ar[d]^{\mathrm{Tr}(f)}\\
(i,h)^{*}r_{2}^{*}q_{2}^{*}(q_{1*}\one\boxtimes
T)\ar[d]&\ar@{=}[l]h^{*}q_{2}^{*}(q_{1*}\one\boxtimes
T)\ar[d]&\ar[l]_-{(\catid_{i},h)^{*}}\catid_{i}^{*}q_{2}^{*}(q_{1*}\one\boxtimes
T)\ar[d]\\
(i,h)^{*}r_{2}^{*}q_{2}^{*}q_{1*}(\one\boxtimes
T)\ar[d]_{\sim}&\ar@{=}[l]h^{*}q_{2}^{*}q_{1*}(\one\boxtimes
T)\ar[d]_{\sim}&\ar[l]_-{(\catid_{i},h)^{*}}\catid_{i}^{*}q_{2}^{*}q_{1*}(\one\boxtimes
T)\ar[d]^{(\catid_{i}\times h)^{*}}\\
(i,h)^{*}r_{1}^{*}q_{1}^{*}q_{1*}\restr{T}{\tw{I}}\ar[d]_{\mathrm{adj}}&\ar@{=}[l]\catid_{i}^{*}q_{1}^{*}q_{1*}\restr{T}{\tw{I}}\ar@{=}[r]\ar[d]_{\mathrm{adj}}&\catid_{i}^{*}q_{1}^{*}q_{1*}\restr{T}{\tw{I}}\ar[d]^{\mathrm{adj}}\\
(i,h)^{*}\restr{T}{\disc{I}}&\catid_{i}^{*}\restr{T}{\tw{I}}\ar@{=}[l]\ar@{=}[r]&\catid_{i}^{*}\restr{T}{\tw{I}}\mathrlap{.}}
  \end{equation*}
  The unlabeled arrows are the canonical ones; all squares clearly
  commute.
\end{proof}

Knowing the components of the trace we now give a better description
of the indexing set $\pi_{0}(\disc{I})$, at least for ``EI-categories'':
\begin{dfi}
  An \emph{EI-category} $I$ is a category whose endomorphisms are all
  invertible, \ie{} such that for all $i\in I_{0}$, $G_{i}:=I(i,i)$ is
  a group.
\end{dfi}
EI-categories have been of interest in studies pertaining to different
fields of mathematics, especially in representation theory and
algebraic topology; closest to our discussion in the sequel is their
role in the study of the Euler characteristic of a category
(see~\cite{fls-euler-hocolim}, \cite{leinster-euler}). We will see
examples of EI-categories below.

Let $I$ be an EI-category; we define the \emph{endomorphism category}
$\emor{I}$ associated to $I$ to be the category whose objects are
endomorphisms in $I$ and an arrow from $h\in I(i,i)$ to $k\in I(j,j)$
is a morphism $m\in I(i,j)$ such that $mh=km$. The
object $h\in I(i,i)$ is sometimes denoted by $(i,h)$. There is also a
canonical functor
\begin{equation*}
  \disc{I}\to \emor{I}
\end{equation*}
which takes a typical object~(\ref{eq:dI-object}) of $\disc{I}$ to its
composition $h_{1}h_{2}\in I(j,j)$. Notice that it takes $(i,h)$ to
$(i,h)$.

\begin{lem}\label{lem:eqrel-iso-ei}
  Let $I$ be an EI-category, $h\in I(i,i)$, $k\in I(j,j)$. Then
  $(i,h)$ and $(j,k)$ lie in the same connected component of $\disc{I}$
  if and only if $h\cong k$ as objects of $\emor{I}$. In other words,
  the functor defined above induces a bijection
  \begin{equation*}
    \pi_{0}(\disc{I})\longleftrightarrow \emor{I}_{0}\slash\cong.
  \end{equation*}
\end{lem}
\begin{proof}
  If $m:h\to k$ is an isomorphism in $\emor{I}$ then
  $(m^{-1},m)$ defines a morphism in $\disc{I}$ from $(i,h)$
  to $(j,k)$.

  For the converse we notice that $\disc{I}$ is a groupoid. Indeed, it
  follows from the definition of an EI-category that in a typical
  object~(\ref{eq:dI-object}) of $\disc{I}$, both $h_{1}$ and $h_{2}$
  must be isomorphisms. From this, and using a similar argument, one
  deduces that the components of any morphism in $\disc{I}$ are
  invertible.

  Given now a morphism $(m_{1},m_{2})$ from $(i,h)$ to $(j,k)$ in $\disc{I}$,
  we must have $m_{1}=m_{2}^{-1}$ and therefore $m_{2}$ defines an isomorphism
  from $h$ to $k$ in $\emor{I}$.
\end{proof}

\begin{exa}\label{exa:indexing-set}\mbox{}
  \begin{enumerate}
  \item Let $I$ be a preordered set considered as a category. Clearly
    this is an EI-category, and $\emor{I}=I$. It follows that we have
    $\pi_{0}(\disc{I})=I_{0}\slash\cong$, the isomorphism classes of
    objects in $I$, or in other words, the (underlying set of the)
    poset associated to $I$. If the hypotheses of
    Lemma~\ref{lem:input_eq} are satisfied then the trace of an
    endomorphism $f$ is just the family of the traces of the fibers
    $(\mathrm{tr}(f_{i}))_{i}$, indexed by isomorphism classes of
    objects in $I$.
  \item Let $G$ be a group. We can consider $G$ canonically as a
    category with one object, the morphisms being given by $G$ itself,
    the composition being the multiplication in $G$. Again, this is an
    EI-category. Given $h$ and $k$ in $G$, an element $m\in G$ defines
    a morphism $m:h\to k$ if and only if it satisfies $mhm^{-1}=k$, so
    $h$ and $k$ are connected (and therefore isomorphic) in $\emor{I}$
    if and only if they are conjugate in $G$. It follows that
    $\pi_{0}(\disc{I})$ can be identified with the set of conjugacy
    classes of $G$. If the hypotheses of Lemma~\ref{lem:input_eq} are
    satisfied then the trace of an endomorphism $f$ with unique fiber
    $e^{*}f$ is just the family of traces $(\mathrm{tr}(h^{*}\circ
    e^{*}f))_{[h]}$, indexed by the conjugacy classes of $G$.
  \item Generalizing the two previous examples, for an arbitrary
    EI-category $I$, $\emor{I}_{0}\slash\cong$ can be identified with
    the disjoint union of the sets $C_{i}$ of conjugacy classes of the
    groups $G_{i}=I(i,i)$ for representatives $i$ of the isomorphism
    classes in $I$, \ie{}
    \begin{equation*}
      \pi_{0}(\disc{I})\longleftrightarrow\coprod_{i\in I_{0}\slash\cong}C_{i}.
    \end{equation*}
    If the hypotheses of Lemma~\ref{lem:input_eq} are satisfied then
    the trace of an endomorphism $f$ is just the family of traces
    $(\mathrm{tr}(h^{*}\circ f_{i}))_{i,[h]}$.
  \end{enumerate}
\end{exa}

\begin{rem}
  One can define the category $\emor{I}$ without the hypothesis that
  $I$ be an EI-category but the previous lemma does not remain true
  without it. However, there is the following general alternative
  description of $\pi_{0}(\disc{I})$: Let $\eqrel$ be the equivalence
  relation on the set $\coprod_{i\in I_{0}}I(i,i)$ generated by the
  relation $m_{1}m_{2}\eqrel m_{2}m_{1}$, $m_{1},m_{2}\in I_{1}$
  composable.  Then $(i,h)$ and $(j,k)$ lie in the same connected
  component of $\disc{I}$ if and only if $h\eqrel k$. It follows that
  for arbitrary $I$, there is a bijection
  \begin{equation*}
    \pi_{0}(\disc{I})\longleftrightarrow \left(\coprod_{i\in I_{0}}I(i,i)\right)\slash\eqrel.
  \end{equation*}
\end{rem}

\section{The trace of the homotopy colimit}
\label{sec:formula}Given a closed monoidal derivator $\D$, a category
$I$ in the domain $\mathbf{Dia}$ of $\D$ and objects $A$ of $\D(I)$
fiberwise dualizable, $S$ and $T$ of $\D(\pt)$, we can associate to
every $f:A\boxtimes S\to A\boxtimes T$ in $\D(I)$ its homotopy colimit
$p_{I!}f:p_{I!}A\boxtimes S\to p_{I!}A\boxtimes T$ by requiring that
the following square commutes:
\begin{equation*}
  \xymatrix{p_{I!}(A\boxtimes S)\ar[r]^{f}\ar[d]^{\sim}_{(\ref{eq:mon-der-external})}&p_{I!}(A\boxtimes
    T)\ar[d]_{\sim}^{(\ref{eq:mon-der-external})}\\
    p_{I!}A\boxtimes S\ar[r]_{p_{I!}f}&p_{I!}A\boxtimes T\mathrlap{.}}
\end{equation*}
We will now show that, in good cases, the trace of $f$ as defined
above contains enough information to compute the trace of the homotopy
colimit of $f$.

\begin{dfi}\label{dfi:Phi}  Given a morphism $g:q_{2!}\one\boxtimes
  S\to q_{1*}\one\boxtimes T$ as in Remark~\ref{rem:components} (or,
  under the hypotheses in Lemma~\ref{lem:input_eq}, the family of its
  $\gamma$-components, $\gamma\in\pi_{0}(\disc{I})$), we associate to
  it a new morphism $\Phi(g):S\to T$, provided that the morphism
  $p_{I!}p_{2*}\to p_{I^{\op}*}p_{1!}$ is invertible. (This latter
  morphism is obtained by adjunction from the composition
  \begin{equation*}
    p_{I^{\op}}^{*}p_{I!}p_{2*}\xleftarrow{\sim} p_{1!}p_{2}^{*}p_{2*}\xrightarrow{\mathrm{adj}}p_{1!},
  \end{equation*}
  where for the first isomorphism one uses
  Lemma~\ref{lem:fibration-exact}.) In this case $\Phi(g)$ is defined
  by the requirement that the following rectangle commutes:
  \begin{equation}\label{eq:main_diagram}
      \xymatrix{S\ar[rr]^{\Phi(g)}\ar[d]_{\mathrm{adj}}&&T\\
      p_{I^{\op}*}p_{I^{\op}}^{*}S&&p_{I!}p_{I}^{*}T\ar[u]_{\mathrm{adj}}\ar[d]^{\mathrm{adj}}_{\sim}\\
      \ar[d]^{\sim}\ar[u]^{\mathrm{adj}}_{\sim}p_{I^{\op}*}p_{1!}q_{2!}q_{2}^{*}p_{1}^{*}p_{I^{\op}}^{*}S&&p_{I!}p_{2*}q_{1*}q_{1}^{*}p_{2}^{*}p_{I}^{*}T\\
      p_{I^{\op}*}p_{1!}q_{2!}(\one\boxtimes
      S)\ar[d]^{\sim}&&p_{I!}p_{2*}q_{1*}(\one\boxtimes T)\ar[u]^{\sim}\\
      p_{I^{\op}*}p_{1!}(q_{2!}\one\boxtimes
      S)\ar[r]_-{g}&p_{I^{\op}*}p_{1!}(q_{1*}\one\boxtimes
      T)&\ar[l]^-{\sim}\ar[u]p_{I!}p_{2*}(q_{1*}\one\boxtimes T)\mathrlap{.}}
  \end{equation}
\end{dfi}

Here, the two (co)units of adjunctions going in the ``wrong''
direction are invertible by Lemma~\ref{lem:isos}.

\begin{rem}\label{rem:phi-naturality}
  Suppose that the conditions~\labelcref{hyp:q_1,hyp:dia} of
  Lemma~\ref{lem:input_eq} are satisfied, thus $\Phi=\Phi_{S,T}$ can
  be identified with a map
  $\prod_{\pi_{0}(\disc{I})}\D(\pt)\left(S,T\right)\to\D(\pt)\left(S,T\right)$. The
  observation is that this map is natural in both arguments, in the
  following sense: Given morphisms $S\to S'$ and $T\to T'$, the
  following diagram commutes:
  \begin{equation*}
    \xymatrix@C=35pt{\prod_{\pi_0(\disc{I})}\D(\pt)\left(S',T\right)\ar[r]^-{\Phi_{S',T}}\ar[d]&\D(\pt)\left(S',T\right)\ar[d]\\
      \prod_{\pi_0(\disc{I})}\D(\pt)\left(S,T\right)\ar[r]^-{\Phi_{S,T}}\ar[d]&\D(\pt)\left(S,T\right)\ar[d]\\
      \prod_{\pi_0(\disc{I})}\D(\pt)\left(S,T'\right)\ar[r]^-{\Phi_{S,T'}}&\D(\pt)\left(S,T'\right)\mathrlap{.}}
  \end{equation*}
  This follows immediately from the definition of $\Phi$.
\end{rem}

\begin{pro}\label{pro:main}
  Let $I\in \mathbf{Dia}_{0}$, and suppose that the following
  conditions are satisfied:
  \begin{enumerate}[resume*=hypotheses]
  \item the morphism $p_{I!}p_{2*}\to p_{I^{\op}*}p_{1!}$ is
    invertible;\label{hyp:limcolim}
  \item the morphism $p_{I^{\op}*}\plho\otimes\plho\to
    p_{I^{\op}*}(\plho\otimes p_{I^{\op}}^{*}\plho)$ is invertible.\label{hyp:p_Iop}
  \end{enumerate}
  If $A\in\D(I)_{0}$ is fiberwise dualizable, $S$, $T\in\D(\pt)_{0}$,
  and $f:A\boxtimes S\to A\boxtimes T$, then the object $p_{I!}A$ is
  dualizable in $\D(\pt)$ and the following equality holds:
  \begin{equation*}
    \Phi(\mathrm{Tr}(f))=\mathrm{tr}(p_{I!}f).
  \end{equation*}
\end{pro}
\begin{proof}\labelcref{hyp:p_Iop} implies that $p_{I!}$ preserves fiberwise dualizable
  objects. Then the proof proceeds by
  decomposing~\eqref{eq:main_diagram} into smaller pieces; since it is
  rather long and not very enlightening we defer it to
  appendix~\ref{sec:app2}.
\end{proof}

\begin{rem}
  It is worth noting that the particular shape of
  diagram~\eqref{eq:main_diagram} is of no importance to us. All we
  will use in the sequel is that there exists a map $\Phi$, natural in
  the sense of Remark~\ref{rem:phi-naturality}, and which takes the
  trace of a (twisted) endomorphism to the trace of its homotopy
  colimit. The idea is the following: Suppose $\D$ is additive, and
  let $I$ be a category
  satisfying~\labelcref{hyp:q_1,hyp:dia,hyp:limcolim,hyp:p_Iop}. Then
  Corollary~\ref{cor:trace-nat} tells us that $\mathrm{Tr}(f)$ is
  completely determined by the local traces $\mathrm{tr}(h^{*}\circ
  f_{i})$, $(i,h)\in\pi_0(\disc{I})$. If $\pi_0(\disc{I})$ is finite
  then, by Remark~\ref{rem:phi-naturality}, we can think of $\Phi$ as
  a linear map which takes the input $(\mathrm{tr}(h^{*}\circ
  f_{i}))_{(i,h)}$ and outputs
  $\sum_{(i,h)}\lambda_{(i,h)}\mathrm{tr}(h^{*}\circ
  f_{i})=\mathrm{tr}(p_{I!}f)$. We will obtain a formula for the trace
  of the homotopy colimit by determining these coefficients
  $\lambda_{(i,h)}$.
\end{rem}

Let $I$ be a finite category. The \emph{$\zeta$-function on $I$} is
defined as the association
\begin{align*}
  \zeta_{I}:I_{0}\times I_{0}&\to \Z\\
  (i,j)&\mapsto \#I(i,j).
\end{align*}
Following~\cite{leinster-euler} we define an \emph{$R$-coweighting on
  $I$}, $R$ a commutative unitary ring, to be a family
$(\lambda_{i})_{i\in I_{0}}$ of elements of $R$ such that the
following equality holds for all $j\in I_{0}$:
\begin{equation}\label{eq:coweighting}
  1=\sum_{i\in I_{0}}\lambda_{i}\zeta_{I}(i,j).
\end{equation}
Not all finite categories possess an $R$-coweighting; and if one such
exists it might not be unique. Preordered sets always possess an
$R$-coweighting (and it is unique if and only if the preorder is a
partial order), groups possess one if and only if their order is
invertible in $R$ (and in this case it is unique). One trivial reason
why a coweighting may not be unique is the existence of isomorphic
distinct objects in a category. For in this case any modification of
the family $(\lambda_{i})_{i}$ which doesn't change the sum of the
coefficients $\lambda_{i}$ for isomorphic objects leaves the right
hand side of~(\ref{eq:coweighting}) unchanged. On the other hand, this
also means that any coweighting $(\lambda_{i})_{i}$ on $I$ induces a
coweighting $(\rho_{j})_{j}$ on the core of $I$ by setting
$\rho_{j}=\sum_{i\in I_{0}, i\cong j}\lambda_{i}$. (Here, ``the''
\emph{core of $I$} is any equivalent subcategory of $I$ which is
\emph{skeletal}, \ie{} has no distinct isomorphic objects.)
Conversely, any coweighting on the core induces a coweighting on $I$
by choosing all additional coefficients to be 0. We therefore say that
$I$ admits an \emph{essentially unique $R$-coweighting} if there is a
unique $R$-coweighting on its core. In this case we sometimes speak
abusively of \emph{the} $R$-coweighting, especially if the context
makes it clear which core is to be chosen.

For an EI-category $I$ we continue to denote by $G_{i}$, $i\in I_{0}$,
the group $I(i,i)$, and by $C_{i}$ the set of conjugacy classes of
$G_{i}$ (cf.~Example~\ref{exa:indexing-set}). Given $h\in G_{i}$, we
denote by $[h]\in C_{i}$ the conjugacy class of $h$ in $G_{i}$.
\begin{dfi}
  Let $I$ be a finite EI-category. We define its \emph{characteristic},
  denoted by $\Cha(I)$, to be the product of distinct prime factors
  dividing the order of the automorphism group of some object in the
  category, \ie{}
  \begin{equation*}
    \Cha(I)=\rad
    \left(\prod_{i\in I_{0}}\#G_{i}
    \right).
  \end{equation*}
\end{dfi}

\begin{lem}[\protect{cf.~\cite[1.4]{leinster-euler}}]\label{lem:coweighting}
  Let $I$ be a finite EI-category and $R$ a commutative unitary ring.
  If $\Cha(I)$ is invertible in $R$ then there is an essentially
  unique $R$-coweighting on $\emor{I}$. It is given as follows:

  Choose a core $J\subset\emor{I}$ of objects $\{(i,h)\}$. Then
  \begin{equation*}
    \lambda_{(j,k)}=\sum_{(i,h)\in J_{0}}\sum_{n\geq
      0}(-1)^{n}\sum\frac{\#[h_{0}]}{\#G_{i_{0}}}\cdots\frac{\#[h_{n}]}{\#G_{i_{n}}},
  \end{equation*}
  where the last sum is over all non-degenerate paths
  \begin{equation*}
    (i,h)=(i_{0},h_{0})\to (i_{1},h_{1})\to\cdots\to
    (i_{n},h_{n})=(j,k)
  \end{equation*}
  from $(i,h)$ to $(j,k)$ in $J$ (\ie{} the $(i_{l},h_{l})$ are
  pairwise distinct, or, equivalently, the $i_{l}$ are pairwise
  distinct, or, also equivalently, none of the arrows is the
  identity).
\end{lem}
\begin{proof}
  The data $(\zeta_{J}(h,k))_{h,k\in J_{0}}$ can be identified in an
  obvious way with a square matrix $\zeta_{J}$ with coefficients in
  $\Z$. For the first claim in the lemma, it suffices to prove that
  $\zeta_{J}$ is an invertible matrix in $R$, for then
  \begin{align*}
    \begin{bmatrix}
      \cdots&\lambda_{(i,h)}&\cdots
    \end{bmatrix}&=    \begin{bmatrix}
      \cdots&\lambda_{(i,h)}&\cdots
    \end{bmatrix}
    \left(\zeta_{J}\zeta_{J}^{-1}\right)\\
    &=\left(\begin{bmatrix}
      \cdots&\lambda_{(i,h)}&\cdots
    \end{bmatrix}\zeta_{J}\right)\zeta_{J}^{-1}\\
    &=\begin{bmatrix}
      \cdots&1&\cdots
    \end{bmatrix}\zeta_{J}^{-1}.
  \end{align*}
  For any $(i,h)\in J_{0}$, the endomorphism monoid is
  $G_{(i,h)}=\mathrm{C}_{G_{i}}(h)$, the centralizer of $h$, hence $J$
  is also a finite EI-category. This implies that we can find an
  object $(i,h)\in J_{0}$ which has no incoming arrows from other
  objects. Proceeding inductively we can thus choose an ordering of
  $J_{0}$ such that the matrix $\zeta_{J}$ is upper
  triangular. Consequently, $\det(\zeta_{J})=\prod_{(i,h)\in
    J_{0}}\#C_{G_{i}}(h)$ is invertible in $R$ by assumption.

  The proof in~\cite[1.4]{leinster-euler} goes through word for
  word to establish the formula given in the lemma (the relation
  between ``Möbius inversion'' and coweighting is given in~\cite[p.~28]{leinster-euler}).
\end{proof}

\begin{exa}\mbox{}
  \begin{enumerate}
  \item Let $I$ be a finite skeletal category with no non-identity
    endomorphisms (e.g.\ a partially ordered set). Then for any ring
    $R$ there is a unique $R$-coweighting on $I=\emor{I}$ given by
    (cf.~\cite[1.5]{leinster-euler})
    \begin{equation*}
      \lambda_{j}=\sum_{i\in I_{0}}\sum_{n\geq
        0}(-1)^{n}\#\{\text{non-degenerate paths of length }n\text{ from }i\text{ to } j\}
    \end{equation*}
    for any $j\in I_{0}$.
  \item Let $I=G$ be a finite group. By
    Example~\ref{exa:indexing-set}, the objects of the core of
    $\emor{G}$ can be identified with the conjugacy classes of
    $G$. For a $\Z[1/\#G]$-algebra $R$, the $R$-coweighting on
    $\emor{G}$ is given by
    \begin{equation*}
      \lambda_{[k]}=\frac{\#[k]}{\#G}
    \end{equation*}
    for any conjugacy class $[k]$ of $G$.
  \end{enumerate}
\end{exa}

\begin{exa}
  Let us go back to the situation considered in the introduction: Let
  $\smallucarre$ be the category of~\eqref{eq:pushout-category}. It
  follows from the first example above that for any ring $R$, the
  unique $R$-coweighting on $\smallucarre=\emor{\smallucarre}$ is given
  by
  \begin{equation*}
    \xymatrix{-1&1\ar[l]\\
      1\ar[u]}
  \end{equation*}
  and one notices that these are precisely the coefficients in the
  formula for the trace of the homotopy
  colimit~\eqref{eq:hty-pushout-result}. This is an instance of the
  following theorem.
\end{exa}

\begin{thm}\label{cor:hocolim_formula}
  Let $\D$ be a closed monoidal triangulated derivator of type
  $\mathbf{Dia}$, let $I$ be a finite EI-category in $\mathbf{Dia}$
  and suppose that $\Cha(I)$ is invertible in $R_{\D}$. If $S,T\in
  \D(\pt)_{0}$, $f:A\boxtimes S\to A\boxtimes T\in\D(I)_{1}$, with
  $A\in\D(I)_{0}$ fiberwise dualizable, then the object $p_{I!}A$ is
  dualizable in $\D(\pt)$, and we have
  \begin{align*}
    \mathrm{tr}(p_{I!}f)=\sum_{\substack{i\in I_{0}\slash\cong\\ [h]\in C_{i}}}\lambda_{(i,h)}\mathrm{tr}(h^{*}\circ
    f_{i})
  \end{align*}
  where $(\lambda_{(i,h)})_{(i,h)}$ is the $R_{\D}$-coweighting on
  $\emor{I}$.
\end{thm}
We will prove the theorem under the additional assumption that all of
the hypotheses~\labelcref{hyp:q_1,hyp:dia,hyp:limcolim,hyp:p_Iop} are
satisfied. In the next section we will show that they in fact
automatically hold (Proposition~\ref{pro:eliminate-hypo}).
\begin{proof}
  We view $\pi_{0}(\disc{I})$ as the set of pairs $(i,h)$ where $i$
  runs through a full set of representatives for the isomorphism
  classes of objects of $I$, and $h$ runs through a full set of
  representatives for the conjugacy classes of $G_{i}$ (use
  Example~\ref{exa:indexing-set}).

  Lemma~\ref{lem:input_eq} tells us that we may consider $\Phi$ as a
  group homomorphism
  \begin{equation*}
    \prod_{\pi_{0}(\disc{I})}\D(\pt)\left(S,T\right)\to\D(\pt)\left(S,T\right).
  \end{equation*}
  We first assume $S=T$, set $R=\D(\pt)\left(S,S\right)$. In this
  case, Remark~\ref{rem:phi-naturality} tells us that $\Phi$ is both
  left and right $R$-linear hence there exist $\lambda_{(i,h)}\in
  Z(R)$, the center of $R$, such that for every
  $g=(g_{(i,h)})_{(i,h)}$ in the domain,
  \begin{equation*}
    \Phi(g)=\sum_{(i,h)\in\pi_{0}(\disc{I})}\lambda_{(i,h)}g_{(i,h)}.
  \end{equation*}
  In particular, if $g=\mathrm{Tr}(f)$ we get
  \begin{align}\label{eq:fund-eq}
    \mathrm{tr}(p_{I!}f)&=\Phi(\mathrm{Tr}(f))&&\text{by Proposition~\ref{pro:main},}\notag\\
    &=\sum_{(i,h)}\lambda_{(i,h)}\mathrm{tr}(h^{*}\circ
    f_{i})&&\text{by Corollary~\ref{cor:trace-nat}.}
  \end{align}
  Now, fix $(j,k)\in\pi_{0}(\disc{I})$. Below we will define a
  specific endomorphism $f$ satisfying
  \begin{align}
    \mathrm{tr}(p_{I!}f)=\catid_{S}\label{eq:hty-trace}\intertext{and}
    \mathrm{tr}(h^{*}\circ f_{i})=\zeta_{\emor{I}}(h,k)\label{eq:local-traces}
  \end{align}
  for any $(i,h)\in\pi_{0}(\disc{I})$. Letting
  $(j,k)\in\pi_{0}(\disc{I})$ vary, (\ref{eq:fund-eq}) thus says that
  the $\lambda_{(i,h)}$ define a $Z(R)$-coweighting on the core of
  $\emor{I}$ and by Lemma~\ref{lem:coweighting} this is unique (by
  assumption, $\Cha(I)$ is invertible in $R_{\D}$ but then it must
  also be invertible in $Z(R)$). It must therefore be (the image of)
  the unique $R_{\D}$-coweighting on the core of $\emor{I}$ and this
  would complete the proof of the theorem in the case $S=T$.

  Before we come to the construction of $f$, let us explain how the
  general case (\ie{} when not necessarily $S=T$) can be deduced. Set
  $U=S\oplus T$ and denote by $\iota:S\to U$ and $\pi:U\to T$ the
  canonical inclusion and projection, respectively. By
  Remark~\ref{rem:phi-naturality}, the following diagram commutes:
  \begin{equation*}
    \xymatrix@C=35pt{\prod_{\pi_{0}(\disc{I})}\D(\pt)\left(U,U\right)\ar[r]^-{\Phi_{U,U}}\ar[d]_{\iota^{*}}&\D(\pt)\left(U,U\right)\ar[d]^{\iota^{*}}\\
      \prod_{\pi_{0}(\disc{I})}\D(\pt)\left(S,U\right)\ar[r]^-{\Phi_{S,U}}\ar[d]_{\pi_{*}}&\D(\pt)\left(S,U\right)\ar[d]^{\pi_{*}}\\
      \prod_{\pi_{0}(\disc{I})}\D(\pt)\left(S,T\right)\ar[r]^-{\Phi_{S,T}}&\D(\pt)\left(S,T\right)\mathrlap{.}}
  \end{equation*}
  Given a family $(g_{(i,h)})_{(i,h)}$ in the bottom left, there is a
  canonical lift $(\widetilde{g_{(i,h)}})_{(i,h)}$ in the top left,
  similarly for the right hand side. In particular, given $S,T,A,f$ as
  in the statement of the theorem,
  \begin{align*}
    \mathrm{tr}(p_{I!}f)&=\Phi_{S,T}(\mathrm{Tr}(f))&&\text{by the
      proposition,}\\
    &=\Phi_{S,T}(\pi_{*}\iota^{*}(\widetilde{\mathrm{Tr}(f)_{(i,h)}})_{(i,h)})\\
    &=\pi_{*}\iota^{*}\Phi_{U,U}((\widetilde{\mathrm{Tr}(f)_{(i,h)}})_{(i,h)})\\
    &=\pi_{*}\iota^{*}{\textstyle\sum_{(i,h)}}\lambda_{(i,h)}\widetilde{\mathrm{tr}(h^{*}\circ
      f_i)}&&\text{by the previous
      argument,}\\
    &={\textstyle\sum_{(i,h)}}\pi\lambda_{(i,h)}\widetilde{\mathrm{tr}(h^{*}\circ
      f_{i})}\iota\\
    &={\textstyle\sum_{(i,h)}}\lambda_{(i,h)}\pi\widetilde{\mathrm{tr}(h^{*}\circ
      f_{i})}\iota\\
    &={\textstyle\sum_{(i,h)}}\lambda_{(i,h)}\mathrm{tr}(h^{*}\circ f_{i}).
  \end{align*}
  This completes the argument in the general case.

  Now we come to the construction of the endomorphism $f$ mentioned
  above. We will freely use the fact that for any
  finite group $G\in \mathbf{Dia}_{0}$ whose order is invertible in
  $R_{\D}$ (such as $G_{i}$ for all $i\in I_{0}$ by assumption), the
  underlying diagram functor
  \begin{equation*}
    \mathrm{dia}_{G}:\D(G)\to\mathbf{CAT}(G^{\op}, \D(\pt))
  \end{equation*}
  is fully faithful. We postpone the proof of this to
  appendix~\ref{sec:group}.

  Fix $(j,k)\in\pi_{0}(\disc{I})$. Denote by $e_{j}:\pt\to G_{j}$ the
  unique functor; by~\ref{dia:closed}, this is a functor in
  $\mathbf{Dia}$. Then $e_{j!}S$ is the right regular representation
  of $G_{j}^{\op}$ associated to $S$ (for more details, see
  appendix~\ref{sec:group}); we denote the action by $r_{(\plho)}$.
  Left translation by $k$, $l_{k}$, defines a
  $G_{j}^{\op}$-endomorphism of $e_{j!}S$. By transitivity of the
  action,
  \begin{equation*}
    p_{G_{j}!}l_{k}:S=p_{G_{j}!}e_{j!}S=e_{j!}S/G_{j}^{\op}\to e_{j!}S/G_{j}^{\op}=p_{G_{j}!}e_{j!}S=S
  \end{equation*}
  is just the identity.

  Let $\overline{j}:G_{j}\to I$ be the fully faithful inclusion
  pointing $j$ and set $f=\overline{j}_{!}l_{k}$. To be
  completely precise, we should set $A=j_{!}\one$, and $f$ to be the
  endomorphism of $A\boxtimes S$ induced by
  $\overline{j}_{!}l_{k}$ via the canonical isomorphism
  \begin{equation*}
    j_{!}\one\boxtimes S\xleftarrow[\sim]{(\ref{eq:mon-der-external})} j_{!}(\one\boxtimes S)\xrightarrow{\sim}j_{!}S\xrightarrow{\sim}\overline{j}_{!}e_{j!}S.
  \end{equation*}
  However, for the sake of clarity, we will continue to use this
  identification implicitly.

  Then we have
  \begin{align*}
    \mathrm{tr}(p_{I!}\overline{j}_{!}l_{k}) &=\mathrm{tr}(p_{G_{j}!}l_{k})\\
    &=\mathrm{tr}(\catid_{S})&&\text{as seen above,}\\
    &=\catid_{S}
  \end{align*}
  \ie{}~(\ref{eq:hty-trace}) holds.

  For~(\ref{eq:local-traces}) we must understand $h^{*}\circ
  i^{*}\overline{j}_{!}l_{k}$. Write $\mathrm{S}(m)$ for the
  stabilizer subgroup of $m\in I(i,j)$ in $G_{j}$ and consider the
  following comma square in $\mathbf{Dia}$
  \begin{equation*}
    \xymatrix{\xtwocell[1,1]{}\omit{^{\eta}}\coprod_{m}\mathrm{S}(m)\ar[r]^-{w}\ar[d]_{p}&G_{j}\ar[d]^{\overline{j}}\\
      \pt\ar[r]_{i}&I}
  \end{equation*}
  where the disjoint union is indexed by a full set of representatives
  for the $G_{j}$-orbits of $I(i,j)$, $w$ is the canonical inclusion
  on each component, and $\eta$ is $m$ on the component of
  $m$. Under the identification $i^{*}\overline{j}_{!}\cong
  p_{!}w^{*}$ (by~\ref{der:kanextptwise}),
  \begin{equation*}
    i^{*}\overline{j}_{!}e_{j!}S\cong p_{!}w^{*}e_{j!}S\cong\bigoplus_{m}
    \left(e_{j!}S/\mathrm{S}(m)\right)\cong \bigoplus_{m}\bigoplus_{G_{j}/\mathrm{S}(m)}S,
    \end{equation*}
    and $i^{*}\overline{j}_{!}l_{k}$ corresponds to the morphism which
    takes the $g\mathrm{S}(m)$-summand identically to the
    $k^{-1}g\mathrm{S}(m)$-summand. It follows that under the
    identification $i^{*}\overline{j}_{!}e_{j!}S\cong i^{*}j_{!}S\cong
    \oplus_{I(i,j)}S$ (again by~\ref{der:kanextptwise}), it
    corresponds to the morphism which takes the $m$-summand
    identically to the $k^{-1}m$-summand.

    Writing out explicitly the Beck-Chevalley transformation
    above we obtain the horizontal arrows in
    the following diagram:
    \begin{equation*}
      \xymatrix{\oplus_{I(i,j)}\ar[r]^-{\mathrm{adj}}\ar[d]_{m\mapsto
        m
        h}&\oplus_{I(i,j)}j^{*}j_{!}\ar[r]^{(m^{*})_{m}}\ar[d]^{m\mapsto
      m
      h}&\oplus_{I(i,j)}i^{*}j_{!}\ar[r]^-{\sum}&i^{*}j_{!}\ar[d]^{h^{*}}\\
  \oplus_{I(i,j)}\ar[r]^-{\mathrm{adj}}&\oplus_{I(i,j)}j^{*}j_{!}\ar[r]^{(m^{*})_{m}}&\oplus_{I(i,j)}i^{*}j_{!}\ar[r]^-{\sum}&i^{*}j_{!}\mathrlap{.}}
    \end{equation*}
    Obviously, the diagram is commutative. In total we get that
    $h^{*}\circ i^{*}\overline{j}_{!}l_{k}$ corresponds to the morphism
    which takes the $m$-summand identically to the $k^{-1}m
    h$-summand. It follows that the trace of this composition is equal
    to
    \begin{align*}
      \mathrm{tr}(h^{*}\circ
      i^{*}\overline{j}_{!}l_{k})&=\#\{m\in I(i,j)\mid k^{-1}m h=m\}\\
      &=\#\emor{I}(h,k)\\
      &=\zeta_{\emor{I}}(h,k).
    \end{align*}
  \end{proof}

\section{$\Q$-linearity and triangulation}
\label{sec:q}
Let $\D$ be a monoidal triangulated derivator. In this section we will
show that for any finite EI-category $I\in\mathbf{Dia}_{0}$, if
$\Cha(I)$ is invertible in $R_{\D}$ then all
hypotheses~\labelcref{hyp:q_1,hyp:dia,hyp:limcolim,hyp:p_Iop}
automatically hold. The main tool used in the proof is
Lemma~\ref{lem:generators} below, in essence suggested to me by Joseph
Ayoub, where it is shown how invertibility of $\Cha(I)$ in $R_{\D}$
and $\D$ being triangulated imply the existence of nice generators for
$\D(I)$. In fact, this is the only place in the article where the
triangulated structure plays any role.

Recall that a subcategory of a triangulated category is called
\emph{thick} if it is a triangulated subcategory and closed under
direct factors. If $\mathcal{T}$ is a triangulated category and
$S\subset \mathcal{T}_{0}$ a family of objects we denote by
\begin{equation*}
  \langle S\rangle \qquad (\text{resp. }\langle S\rangle_{s})
\end{equation*}
the triangulated (resp. thick) subcategory generated by $S$.

Let $I$ be an EI-category. If $i\in I_{0}$ is an object we denote by
$G_{i}$ its automorphism group $I(i,i)$, and by $\overline{i}:G_{i}\to
I$ the fully faithful embedding of the ``point'' $i$ into $I$. This is
to distinguish it from the inclusion $i:\pt\to I$.
\begin{lem}\label{lem:generators} Let $\D$ be a triangulated
  derivator, and let $I\in\mathbf{Dia}_{0}$ be a finite
  EI-category. Then we have the following equality:
  \begin{align*}
    \D(I)&=\langle \overline{i}_{!}A\mid i\in I_{0}, A\in
    \D(G_{i})_{0}\rangle.
  \end{align*}
  Suppose that for all $i\in I_{0}$, the canonical functor
  $e_{i}:\pt\to G_{i}$ induces a faithful functor
  $e_{i}^{*}:\D(G_{i})\to\D(\pt)$. Then we also have the following
  equality:
  \begin{align*}
    \D(I)&=\langle i_{!}A\mid i\in I_{0}, A\in\D(\pt)_{0}\rangle_{s}.
  \end{align*}

  All these statements remain true if we replace $(\plho)_{!}$ by
  $(\plho)_{*}$ everywhere.
\end{lem}
\begin{rem}\label{rem:faithful-fiber}
  We will prove in appendix~\ref{sec:group} that if $n$ is invertible
  in $R_{\D}$ then $e:\pt\to G$ induces a faithful functor
  $e^{*}:\D(G)\to \D(\pt)$ for any group $G\in\mathbf{Dia}_{0}$ of
  order $n$. In particular, if $\Cha(I)$ is invertible in $R_{\D}$
  then the second equality in Lemma~\ref{lem:generators} holds.
\end{rem}
\begin{proof}[Proof of Lemma~\ref{lem:generators}]
  Note that since $I\in\mathbf{Dia}_{0}$ so is $G_{i}$, $i\in I_{0}$,
  by~\ref{dia:closed}. Therefore, the statement of the lemma at least
  makes sense.

  The first equality is proved by induction on the number $n$ of
  objects in $I$. Clearly, we may assume $I$ to be skeletal. If $n=1$,
  the claim is obviously true. If $n>1$ we find an object $i\in I_{0}$
  which is maximal in the sense that the implication
  $I(i,j)\neq\emptyset\Rightarrow i=j$ holds. For any
  $B\in\D(I)_{0}$, consider the morphism
  \begin{equation*}
    \mathrm{adj}:\overline{i}_{!}\overline{i}^{*}B\to B
  \end{equation*}
  and let $C$ be the cone. One checks easily that $i^{*}\mathrm{adj}$
  is an isomorphism hence $i^{*}C=0$ which implies that $C$ is of the
  form $u_{!}B'$, some $B'\in\D(U)_{0}$ where $u:U\inj I$ is the open
  embedding of the full subcategory of objects different from $i$ in
  $I$ (see~\cite[8.11]{cisinski-neeman}). By induction, $B'\in\langle
  \overline{j}_{!}A\mid j\in U_{0}, A\in
  \D(G_{j})_{0}\rangle$, hence it suffices to prove
  \begin{equation*}
    u_{!}\langle \overline{j}_{!}A\mid j\in U_{0}, A\in
    \D(G_{j})_{0}\rangle\subset \langle \overline{j}_{!}A\mid j\in I_{0}, A\in
    \D(G_{j})_{0}\rangle.
  \end{equation*}
  But this follows from the fact that $u_{!}$ is a triangulated
  functor and $u_{!}\overline{j}_{!}=\overline{j}_{!}$.

  For the second equality, it will follow from the first as soon as we
  prove, for each $i\in I_{0}$,
  \begin{equation}
    \D(G_{i})=\langle e_{i!}A\mid A\in\D(\pt)_{0}\rangle_{s}.\label{eq:group-alt-description}
  \end{equation}
  So let $B\in\D(G_{i})_{0}$ and consider the counit
  $e_{i!}e_{i}^{*}B\to B$. By assumption this is an epimorphism. But
  in a triangulated category every epimorphism is complemented, \ie{}
  \begin{equation*}
    e_{i!}e_{i}^{*}B\cong B\oplus B',
  \end{equation*}
  some $B'\in\D(G_{i})_{0}$. This proves~(\ref{eq:group-alt-description})
  and hence the second equality.

  The last claim of the lemma can be established by dualizing the
  whole proof.
\end{proof}

\begin{pro}\label{pro:eliminate-hypo}
  Let $\D$ be a monoidal triangulated derivator, let $I$ be a finite
  EI-category in $\mathbf{Dia}$, and suppose that $\Cha(I)$ is
  invertible in $R_{\D}$. Then all
  hypotheses~\labelcref{hyp:q_1,hyp:dia,hyp:limcolim,hyp:p_Iop} are
  satisfied.
\end{pro}
\begin{proof}
  \begin{enumerate}
  \item For~\labelcref{hyp:q_1} we will prove more generally that
    \begin{equation}
      q_{1*}A\otimes
      \restr{T}{I^{\op}\times I}\to q_{1*}(A\otimes \restr{T}{\tw{I}})\label{eq:h1-arrow}
    \end{equation}
    is invertible for all $T\in\D(\pt)_{0}$, $A\in\D(\tw{I})_{0}$. Also we
    will only need that $I$ has finite $\hm$-sets.

    Fix $i,j\in I_{0}$ and consider the following pullback square:
    \begin{equation*}
      \xymatrix{I(i,j)\ar[r]^-{r_{i,j}}\ar[d]_{p_{I(i,j)}}&\tw{I}\ar[d]^{q_{1}}\\
        \pt\ar[r]_-{(i,j)}&I^{\op}\times I\mathrlap{.}}
    \end{equation*}
    Since $q_{1}$ is an opfibration, Lemma~\ref{lem:fibration-exact}
    tells us that the first vertical morphisms on the left and on the
    right in the following diagram are invertible:
    \begin{equation*}
        \xymatrix{(i,j)^{*}(q_{1*}A\otimes
      \restr{T}{I^{\circ}\times I})\ar[r]\ar[d]_{\sim}& (i,j)^{*}q_{1*}(A\otimes
      \restr{T}{\mathrm{tw}(I)})\ar[d]^{\sim}\\
      p_{I(i,j)*}r_{i,j}^{*}A\otimes T\ar[d]_{\sim}\ar[r]&p_{I(i,j)*}r_{i,j}^{*}(A\otimes\restr{T}{\mathrm{tw}(I)})\ar[d]^{\sim}\\
      \prod_{h\in I(i,j)}h^{*}r_{i,j}^{*}A\otimes T\ar[r]&\prod_{h\in I(i,j)}(h^{*}r_{i,j}^{*}A\otimes T)\mathrlap{.}}
    \end{equation*}
    Here, $h:\pt\to I(i, j)$ is the functor defined by the object $h$
    of the discrete category $I(i, j)$, and the axiom~\ref{der:coprod}
    is used for the second vertical morphisms on the left and on the
    right. Clearly both squares commute. Moreover, the bottom
    horizontal arrow is invertible since $I(i,j)$ is finite and the
    internal product in $\D(\pt)$ is additive
    (Lemma~\ref{lem:mon-triangulated}). Therefore also the top
    horizontal arrow is invertible which implies
    (by~\ref{der:conservative}, and letting $i$ and $j$ vary)
    that~(\ref{eq:h1-arrow}) is.
  \item For~\labelcref{hyp:dia}, let $\gamma$ be a connected component
    of $\disc{I}$. Since $I$ is a finite EI-category, $\gamma$ is
    equivalent to a finite group $G$ whose order divides $\Cha(I)$. As
    explained in Remark~\ref{rem:faithful-fiber}, this implies that
    $e_{G}^{*}$ is faithful ($e_{G}:\pt\to G$). Since $p_{G}^{*}$ is a
    section of $e_{G}^{*}$, it follows that $p_{G}^{*}$ is fully
    faithful.
  \item \labelcref{hyp:limcolim} states that $p_{I!}p_{2*}\to
    p_{I^{\op}*}p_{1!}$ is invertible. Since also $I^{\op}\times I$ is
    a finite EI-category and since $\Cha(I^{\op}\times I)=\Cha(I)$ we
    may prove this on objects in the image of $(\catid_{I^{\op}}\times
    i)_{!}$, where $i\in I_{0}$, by the previous lemma. Consider the
    following square:
    \begin{equation*}
      \xymatrix{p_{I!}p_{2*}(\catid_{I^{\op}}\times
        i)_{!}\ar[r]&p_{I^{\op}*}p_{1!}(\catid_{I^{\op}}\times i)_{!}\ar[d]^{\sim}\\
        p_{I!}i_{!}p_{I^{\op}*}\ar[r]_-{\sim}\ar[u]&p_{I^{\op}*}\mathrlap{.}}
    \end{equation*}
    It is easy to see that it commutes hence it suffices to prove
    invertibility of the left vertical arrow.

    For this we use~\ref{der:conservative}, so fix $j\in I_{0}$ an
    object. Then
    \begin{align*}
      j^{*}i_{!}p_{I^{\op}*}&\cong\oplus_{I(j,i)}p_{I^{\op}*},\\
      j^{*}p_{2*}(\catid_{I^{\op}}\times i)_{!}&\cong
      p_{I^{\op}*}(\catid_{I^{\op}}\times
      j)^{*}(\catid_{I^{\op}}\times i)_{!}\\
      &\cong p_{I^{\op}*}\oplus_{I(j,i)}.
    \end{align*}
    The claim follows since $p_{I^{\op}*}$ is additive. (It is easy to see
    that this identification is compatible with the vertical arrow
    above.)
  \item Since also $I^{\op}$ is a finite EI-category and
    $\Cha(I^{\op})=\Cha(I)$, we may replace $I$ by
    $I^{\op}$. \labelcref{hyp:p_Iop} then is the statement that
    \begin{equation*}
      p_{I*}A\otimes B\to p_{I*}(A\otimes p_{I}^{*}B)
    \end{equation*}
    is invertible, and by Lemma~\ref{lem:generators} we may assume
    $A=i_{*}C$, some $C\in\D(\pt)_{0}$ and $i\in I_{0}$ (here we use
    that $\plho\otimes B$ and $\plho\otimes p_{I}^{*}B$ both take
    distinguished triangles to distinguished triangles, by
    Lemma~\ref{lem:mon-triangulated}).

    Clearly, the following square commutes:
    \begin{equation*}
      \xymatrix{\ar[d]_{\sim}p_{I*}i_{*}C\otimes B\ar[r]& p_{I*}(i_{*}C\otimes
        p_{I}^{*}B)\ar[d]\\
        C\otimes B&\ar[l]^-{\sim}p_{I*}i_{*}(C\otimes B)\mathrlap{,}}
    \end{equation*}
    hence it suffices to prove invertible the vertical arrow on the
    right. Again we use~\ref{der:conservative}, so let $j\in I_{0}$
    an object. Then:
    \begin{align*}
      j^{*}i_{*}C\otimes
      j^{*}p_{I}^{*}B&\cong\oplus_{I(i,j)}C\otimes B,\\
      j^{*}i_{*}(C\otimes B)&\cong\oplus_{I(i,j)}(C\otimes B).
    \end{align*}
    Again, the claim follows from the additivity of the functor
    $\plho\otimes B$.
  \end{enumerate}
\end{proof}

\appendix{}
\titlelabel{Appendix~\thetitle.\quad}
\section{Properties of the external hom}
\label{sec:app1}
In this section we want to give proofs for the properties of the
external hom listed in section~\ref{sec:ehom}. We take them up one by
one. Throughout the section we fix a closed monoidal derivator $\D$ of
type $\mathbf{Dia}$.
\subsection*{Naturality}
Given $u:I'\to I$ and $v:J'\to J$ in $\mathbf{Dia}$ there is an
induced morphism of diagrams
\hyperlink{pi-diagram}{$(\Pi_{I',J'})$}$\to$\hyperlink{pi-diagram}{$(\Pi_{I,J})$}
and we distinguish the morphisms in the former from their counterparts
in the latter by decorating them with a prime. We deduce a morphism
\begin{align*}
  \xymatrix@R=4pt@C=10pt{\mathllap{\Psi^{u,v}_{A,B}:(u^{\op}\times v)^{*}\ehm{A}{B}}\ar@{=}[r]&\mathrlap{(u^{\op}\times  v)^{*}p_{*}[q^{*}A,r^{*}B]}\\
    \ar[r]&\mathrlap{p'_{*}(\tw{u}\times v)^{*}[q^{*}A,r^{*}B]}\\
    \ar[r]&\mathrlap{p'_{*}[(\tw{u}\times v)^{*}q^{*}A,(\tw{u}\times v)^{*}r^{*}B]}\\
    \ar@{=}[r]&\mathrlap{\ehm{u^{*}A}{v^{*}B}.}}
\end{align*}
Clearly, this morphism is natural in $A$ and $B$, moreover it behaves
well with respect to identities and composition of functors as well as
natural transformations in
$\mathbf{Dia}^{\op}\times\mathbf{Dia}^{\op,\op}$ so that we have
defined a lax natural transformation. The following proposition thus
concludes the proof of the naturality property.

\begin{pro}
  For $u$,$v$ and $A$, $B$ as above the morphism $\Psi^{u,v}_{A,B}$ is
  invertible.
\end{pro}
\begin{proof}
  We proceed in several steps.
  \begin{enumerate}
  \item Let $i\in I'_{0}$, $j\in J'_{0}$. It suffices to prove that
    $(i,j)^{*}$ applied to the morphism $\Psi^{u,v}_{A,B}$ is
    invertible. But this means that it suffices to prove that
    $\Psi^{ui,vj}_{\plho,\plho}$ and $\Psi^{i,j}_{\plho,\plho}$ are
    invertible; in other words we may assume $I'=J'=\pt$, $u=i\in
    I_{0}$, $v=j\in J_{0}$.
  \item We factor
    $(i,j):\pt\xrightarrow{i}I\xrightarrow{\catid_{I}\times j} I\times
    J$, and first deal with $\Psi^{\catid_{I},j}_{\plho,\plho}$. In
    this case, the square
    \begin{equation*}
      \xymatrix{\tw{I}\ar[r]^-{\catid_{\tw{I}}\times j}\ar[d]_{p'}&\tw{I}\times
        J\ar[d]^{p}\\
        I^{\op}\ar[r]_-{\catid_{I^{\op}}\times j}&I^{\op}\times J}
    \end{equation*}
    is a pullback square, and $p$ an opfibration, therefore the first
    arrow in the definition of $\Psi$ is invertible
    (Lemma~\ref{lem:fibration-exact}). For the second arrow in the
    definition, it suffices to prove invertible
    \begin{align*}
      (\catid_{\tw{I}}\times j)^{*}[(\catid_{\tw{I}}\times
      p_{J})^{*}\plho,\plho]\to [\plho,(\catid_{\tw{I}}\times
      j)^{*}\plho].
    \end{align*}
    By passing to $\D_{\tw{I}}$ we may thus assume $I=\pt$ and prove
    instead invertible
    \begin{equation*}
      j^{*}[p_{J}^{*}\plho,\plho]\to [\plho,j^{*}\plho].
    \end{equation*}
    By adjunction, this corresponds to the morphism
    \begin{equation*}
      (\catid_{\pt}\times j)_{!}(\plho\boxtimes \plho)\xrightarrow{(\ref{eq:mon-der-external})} \plho\boxtimes j_{!}\plho
    \end{equation*}
    which we know to be invertible.
  \item Thus from now on we may assume $J=\pt$. Factor
    $i:\pt\xrightarrow{\catid_{i}}i\backslash
    I\xrightarrow{t}I$. Exactly the same argument as in the previous
    step shows that the first arrow in the definition of
    $\Psi^{t,\catid_{\pt}}$ is invertible. Moreover, $\tw{t}$ is a
    fibration hence, by Lemma~\ref{lem:mon-der}, also the second arrow
    in the definition of $\Psi^{t,\catid_{\pt}}$ is invertible.
  \item From now on, we may assume
    that $I$ has initial object $i$ and we need to prove
    $\Psi^{i,\catid_{\pt}}$ invertible. Consider the following diagram:
    \begin{align*}
      \scalebox{.95}{\xymatrix{i^{*}p_{*}[q^{*}\plho,p_{\tw{I}}^{*}\plho]\ar[r]\ar[rd]_{\sim}&\catid_{i}^{*}[q^{*}\plho,p_{\tw{I}}^{*}\plho]\ar[r]&[\catid_{i}^{*}q^{*}\plho,\catid_{i}^{*}p_{\tw{I}}^{*}\plho]\ar@{=}[r]&[i^{*}\plho,\plho]\\
        &p_{\tw{I}*}[q^{*}\plho,p_{\tw{I}}^{*}\plho]\ar[u]&\ar[u]\ar[l]^-{\sim}[p_{\tw{I}!}q^{*}\plho,\plho]&\ar[u]_-{\sim}\ar[l]^-{\sim}[p_{I!}\plho,\plho]\mathrlap{.}}}
    \end{align*}
    The composition of the top horizontal arrows is nothing but
    $\Psi^{i,\catid_{\pt}}$. The triangle on the left arises from the
    Beck-Chevalley transformations associated to the squares
    \begin{equation*}
      \xymatrix{ \pt\ar[r]^-{\catid_{i}}\ar@{=}[d]&\xtwocell[1,1]{}\omit\tw{I}\ar@{=}[r]\ar[d]_{p_{\tw{I}}}&\tw{I}\ar[d]^{p}\\
        \pt\ar@{=}[r]&\pt\ar[r]_{i}&I^{\op}\mathrlap{.}}
    \end{equation*}
    It follows that the triangle commutes and the slanted morphism is
    invertible by \ref{der:kanextptwise}. The first bottom horizontal
    arrow is invertible by Lemma~\ref{lem:mon-der}, the second one
    arises from the counit $q_{!}q^{*}\to \catid$ which is invertible
    by Lemma~\ref{lem:isos}. The middle vertical arrow is induced by
    the ``dual'' of the left vertical arrow,
    $\catid_{i}^{*}\xrightarrow{\mathrm{adj}}\catid_{i}^{*}p_{\tw{I}}^{*}p_{\tw{I}!}=
    p_{\tw{I}!}$. The commutativity of the left square is therefore
    immediate, as is the commutativity of the right square (the right
    vertical arrow is induced by the canonical identification
    $i^{*}\cong p_{I!}$ as $i$ is an initial object of $I$).
  \end{enumerate}
\end{proof}

\subsection*{Internal hom}

We now want to show that in case $I=J\in\mathbf{Dia}_{0}$, internal
hom can be expressed in terms of external hom. Consider the following
category $3I$: Objects are two composable arrows in $I$ and morphisms
from the top to the bottom are of the form:
  \begin{equation*}
    \xymatrix{i_{2}\ar[d]\ar[r]&i_{1}\ar[r]&i_{0}\ar[d]\\
      j_{2}\ar[r]&j_{1}\ar[r]\ar[u]&j_{0}\mathrlap{.}}
  \end{equation*}
  We have canonical functors $t_{k}:3I\to I$, $k=0,2$. Moreover, there
  are functors $p':3I\to \tw{I}^{\op}$ and $q':3I\to \tw{I}\times I$,
  the first one forgetting the 0-component, the second one mapping the
  two components 0 and 1 to $\tw{I}$ and component 2 to $I$. It is
  easy to see that one gets a pullback square:
  \begin{equation*}
    \xymatrix{3I\ar[r]^-{p'}\ar[d]_{q'}&\tw{I}^{\op}\ar[d]^{q_{2}}\\
      \tw{I}\times I\ar[r]_-{p}&I^{\op}\times I\mathrlap{.}}
  \end{equation*}
  Notice that there is a canonical natural transformation $t_{2}\to
  t_{0}$ and hence one can define the following morphism:
    \begin{equation}\label{eq:phi}
      \xymatrix@R=4pt@C=15pt{\mathllap{\Theta^{I}_{A,B}:[A,B]}\ar[r]^{\mathrm{adj}}&\mathrlap{[t_{2!}t_{2}^{*}A,B]}\\
      \ar[r]&\mathrlap{[t_{2!}t_{0}^{*}A,B]}\\
      \ar[r]&\mathrlap{t_{2*}[t_{0}^{*}A,t_{2}^{*}B]}\\
      &\ar[l]_-{\sim}\mathrlap{p_{2*}q_{2*}p'_{*}[q'^{*}q^{*}A,q'^{*}r^{*}B]}\\
      &\ar[l]_-{\sim}\mathrlap{p_{2*}q_{2*}p'_{*}q'^{*}[q^{*}A,r^{*}B]}\\
      &\ar[l]_-{\sim}\mathrlap{p_{2*}q_{2*}q_{2}^{*}p_{*}[q^{*}A,r^{*}B]\mathrlap{.}}}
    \end{equation}
    Here the last isomorphism is due to
    Lemma~\ref{lem:fibration-exact} and $q_{2}$ being a
    fibration. Therefore also $q'$ is a fibration and
    Lemma~\ref{lem:mon-der} gives us the second to last isomorphism.

    Again, $\Theta^{I}_{A,B}$ is clearly natural in $A$ and $B$ and
    one checks easily (if tediously) that the following diagram
    commutes for any $u:I'\to I$ in $\mathbf{Dia}_{1}$:
    \begin{equation*}
      \xymatrix{u^{*}[A,B]\ar[r]^-{\Theta^{I}}\ar[d]&u^{*}p_{2*}q_{2*}q_{2}^{*}\ehm{A}{B}\ar@{.>}[r]&p'_{2*}q'_{2*}\tw{u}^{\op*}q_{2}^{*}\ehm{A}{B}\ar@{=}[d]\\
        [u^{*}A,u^{*}B]\ar[r]_-{\Theta^{I'}}&p'_{2*}q'_{2*}q_{2}'^{*}\ehm{u^{*}A}{u^{*}B}&p'_{2*}q'_{2*}q_{2}'^{*}(u^{\op}\times
        u)^{*}\ehm{A}{B}\ar@{.>}[l]^{\Psi}\mathrlap{.}}
    \end{equation*}    
    It follows that if we take the composition of the dotted arrows in
    the diagram as components of the 2-cells for the lax natural
    transformation $p_{2*}q_{2*}q_{2}^{*}\ehm{\plho}{\plho}$, then
    $\Theta$ defines a modification as claimed in
    section~\ref{sec:ehom}. It now remains to prove that it is
    invertible.
\begin{pro}
  $\Theta^{I}_{A,B}$ is invertible for all $I$, $A$ and $B$ as above.
\end{pro}
\begin{proof}
  It is easy to see that $t_{2}$ is a fibration. Hence it follows from
  Lemma~\ref{lem:mon-der} that the third arrow in~(\ref{eq:phi}) is
  invertible, and it now suffices to prove that
  \begin{equation}
    \label{eq:adj-nattrans}
    t_{2!}t_{0}^{*}\to t_{2!}t_{2}^{*}\xrightarrow{\mathrm{adj}} 1
  \end{equation}
  is invertible. Let $i\in I_{0}$ be an arbitrary object. We will show
  that $i^{*}$ applied to~(\ref{eq:adj-nattrans}) is invertible which
  is enough for the claim by~\ref{der:conservative}.

  Consider the following two diagrams:
  \begin{align*}
    \xymatrix{3I_{i}\ar[r]^{w}\ar[d]_{p_{3I_{i}}}&3I\ar[d]^{t_{2}}\\
      \pt\ar[r]_{i}&I\mathrlap{,}}&&
    \xymatrixcolsep{5pc}\xymatrix{3I_{i}\ar[r]^{w}\ar[d]_{u}&3I\dtwocell_{t_{2}}^{t_{0}}{^}\\
      i\backslash I\rtwocell_{v}^{ip_{i\backslash I}}&I\mathrlap{.}}
  \end{align*}
  The first one is a pullback square, in the second one $u$ is defined
  by $u(i\to i_{1}\to i_{0})=i\to i_{0}$, while $v(i\to i_{0})=i_{0}$
  and $ip_{i\backslash I}\to v$ is the canonical natural
  transformation. This second diagram is commutative in the sense that
  $ip_{i\backslash I}u\to vu$ is equal to $t_{2}w\to
  t_{0}w$. Consequently the second inner square on the left of the
  following diagram commutes:
  \begin{equation*}
    \xymatrix{i^{*}t_{2!}t_{0}^{*}\ar[r]&i^{*}t_{2!}t_{2}^{*}\ar[r]^-{\mathrm{adj}}&i^{*}\\
      p_{3I_{i}!}w^{*}t_{0}*\ar[r]\ar[u]^{\sim}\ar[d]_{\sim}&p_{3I_{i}!}w^{*}t_{2}^{*}\ar[u]^{\sim}\ar[d]_{\sim}\ar@{=}[r]&**[l]p_{3I_{i}!}p_{3I_{i}}^{*}i^{*}\ar[u]^{\mathrm{adj}}\ar[dl]^{\sim}\\
      p_{i\backslash
        I!}u_{!}u^{*}v^{*}\ar[r]\ar[d]_{\mathrm{adj}}&p_{i\backslash
        I!}u_{!}u^{*}p_{i\backslash I}^{*}i^{*}\ar[d]_{\mathrm{adj}}\\
      p_{i\backslash I!}v^{*}\ar[r]_-{\sim}&p_{i\backslash
        I!}p_{i\backslash I}^{*}i^{*}\ar@/_5pc/[ruuu]_{\mathrm{adj}}^{\sim}\mathrlap{.}}
  \end{equation*}
  The rest is clearly commutative. Moreover, the top row is the fiber
  of~(\ref{eq:adj-nattrans}) over $i$. The isomorphism of functors
  $p_{i\backslash I!}\cong\catid_{i}^{*}$ ($\catid_{i}$ being the
  initial object of $i\backslash I$) implies that the bottom
  horizontal as well as the bent arrow induced by the counit of the
  adjunction $p_{i\backslash I!}\dashv p_{i\backslash I}^{*}$ are
  invertible, hence it suffices to prove $u_{!}u^{*}\to \Id$ an
  isomorphism. But this is true since $u$ admits a fully faithful
  right adjoint
  \begin{align*}
    i\backslash I&\longrightarrow 3I_{i}\\
    (i\to j)&\longmapsto (i\xrightarrow{\catid_{i}}i\to j).
  \end{align*}
\end{proof}

\subsection*{External product}

Recall that for any closed monoidal category there is a canonical
morphism
\begin{align}\label{eq:prod_ihom}
  [A_{1},A_{2}]\otimes [A_{3},A_{4}]\to [A_{1}\otimes
  A_{3},A_{2}\otimes A_{4}]
\end{align}
defined by adjunction as follows:
\begin{align*}
  \xymatrix@R=4pt@C=20pt{\mathllap{([A_{1},A_{2}]\otimes [A_{3},A_{4}]) \otimes (A_{1}\otimes
  A_{3})}\ar[r]^{\sim}&\mathrlap{([A_{1},A_{2}]\otimes A_{1})\otimes
  ([A_{3},A_{4}] \otimes A_{3})}\\
  \ar[r]^{\mathrm{ev}\otimes\mathrm{ev}}&\mathrlap{A_{2}\otimes A_{4}\mathrlap{.}}}
\end{align*}
From this we deduce for $A_{1},A_{3}\in \D(I)_{0}$, $A_{2},A_{4}\in \D(J)_{0}$
($I,J\in\mathbf{Dia}_{0}$):
\begin{align}\label{eq:prod_ehom}
  &\xymatrix@R=4pt@C=20pt{\mathllap{\ehm{A_{1}}{A_{2}}\otimes
  \ehm{A_{3}}{A_{4}}}\ar@{=}[r]&\mathrlap{p_{*}[q^{*}A_{1},r^{*}A_{2}]\otimes
  p_{*}[q^{*}A_{3},r^{*}A_{4}]}\\
  \ar[r]&\mathrlap{p_{*}([q^{*}A_{1},r^{*}A_{2}]\otimes
  [q^{*}A_{3},r^{*}A_{4}])}\\
\ar[r]^{(\ref{eq:prod_ihom})}&\mathrlap{p_{*}([q^{*}A_{1}\otimes
  q^{*}A_{3},r^{*}A_{2}\otimes r^{*}A_{4}])}\\
  \ar[r]^{\sim}&\mathrlap{\ehm{A_{1}\otimes A_{3}}{A_{2}\otimes A_{4}}\mathrlap{.}}}&
\end{align}
Now, fix categories $I_{(k)}$, $k=1,\ldots 4$ in $\mathbf{Dia}$ and
objects $A_{k}\in \D(I_{(k)})_{0}$. Set $K=I_{(1)}^{\op}\times I_{(2)}\times
I_{(3)}^{\op}\times I_{(4)}$. We can now finally
define the morphism $\Xi$:
\begin{multline}\label{eq:prod_ehom2}
\Xi^{I_{(1)},I_{(2)},I_{(3)},I_{(4)}}_{A_{1},A_{2},A_{3},A_{4}}:\quad\ehm{A_{1}}{A_{2}}\boxtimes
  \ehm{A_{3}}{A_{4}}=\restr{\ehm{A_{1}}{A_{2}}}{K}\otimes
  \restr{\ehm{A_{3}}{A_{4}}}{K}\\
  \xrightarrow[\sim]{\Psi}\tau^{*}\ehm{\restr{A_{1}}{I_{(1)}\times
      I_{(3)}}}{\restr{A_{2}}{I_{(2)}\times I_{(4)}}}\otimes\tau^{*}\ehm{\restr{A_{3}}{I_{(1)}\times
      I_{(3)}}}{\restr{A_{4}}{I_{(2)}\times I_{(4)}}}\\
  \xleftarrow[\sim]{}\tau^{*}
  \left(
\ehm{\restr{A_{1}}{I_{(1)}\times
      I_{(3)}}}{\restr{A_{2}}{I_{(2)}\times I_{(4)}}}\otimes\ehm{\restr{A_{3}}{I_{(1)}\times
      I_{(3)}}}{\restr{A_{4}}{I_{(2)}\times I_{(4)}}}
  \right)\\
  \xrightarrow{(\ref{eq:prod_ehom})}\tau^{*}\ehm{A_{1}\boxtimes
    A_{3}}{A_{2}\boxtimes A_{4}}.
\end{multline}
Clearly, $\Xi^{I_{(1)},I_{(2)},I_{(3)},I_{(4)}}$ is a natural
transformation. To conclude the proof of the external product property
it remains to verify the following lemma.
\begin{lem} Let $u_{k}:I'_{(k)}\to I_{(k)}$, $k=1,\ldots,4$. Then the
  following diagram commutes:
  \begin{equation*}
    \scalebox{0.9}{\xymatrix{(u_{1}^{\op}\times u_{3}^{\op}\times u_{2}\times u_{4})^{*}(\ehm{A_{1}}{A_{2}}\boxtimes
      \ehm{A_{3}}{A_{4}})\ar[r]^-{\Xi}\ar[d]_{\Psi}^{\sim}&(u_{1}^{\op}\times
      u_{3}^{\op}\times u_{2}\times
      u_{4})^{*}\tau^{*}\ehm{A_{1}\boxtimes A_{3}}{A_{2}\boxtimes A_{4}}\ar[d]^{\Psi}_{\sim}\\
      \ehm{u_{1}^{*}A_{1}}{u_{2}^{*}A_{2}}\boxtimes
      \ehm{u_{3}^{*}A_{3}}{u_{4}^{*}A_{4}}\ar[r]_-{\Xi}&\tau^{*}\ehm{u_{1}^{*}A_{1}\boxtimes u_{3}^{*}A_{3}}{u_{2}^{*}A_{2}\boxtimes u_{4}^{*}A_{4}}\mathrlap{.}}}
  \end{equation*}
\end{lem}
\begin{proof}
  By decomposing the horizontal arrows according to their definition
  in~(\ref{eq:prod_ehom2}) one immediately reduces to showing
  that~(\ref{eq:prod_ehom}) behaves well with respect to the functors
  $u_{k}$; in other words one reduces to showing that for
  $A_{1},A_{3}\in \D(I)_{0}$, $A_{2},A_{4}\in \D(J)_{0}$ and functors $u:I'\to I$,
  $v:J'\to J$, the following diagram commutes:
  \begin{equation*}
    \xymatrix{(u^{\op}\times
      v)^{*}(\ehm{A_{1}}{A_{2}}\otimes\ehm{A_{3}}{A_{4}})\ar[r]^{(\ref{eq:prod_ehom})}\ar[d]_{\Psi}^{\sim}&(u^{\op}\times
      v)^{*}\ehm{A_{1}\otimes A_{3}}{A_{2}\otimes
        A_{4}}\ar[d]^{\Psi}_{\sim}\\
      \ehm{u^{*}A_{1}}{v^{*}A_{2}}\otimes\ehm{u^{*}A_{3}}{v^{*}A_{4}}\ar[r]_{(\ref{eq:prod_ehom})}&\ehm{u^{*}A_{1}\otimes
        u^{*}A_{3}}{v^{*}A_{2}\otimes v^{*}A_{4}}\mathrlap{.}}
  \end{equation*}
  Since the unit and counit of the adjunction $p^{*}\dashv p_{*}$
  behave well with respect to pulling back along $u^{\op}\times v$ and
  $\tw{u}\times v$ one reduces further to showing
  that~\eqref{eq:prod_ihom} is functorial in this sense which is
  clear.
\end{proof}

\subsection*{Adjunction}

Fix three categories $I$, $J$, $K$ in $\mathbf{Dia}$, and objects
$A\in \D(I)_{0}, B\in\D(J)_{0}, C\in\D(K)_{0}$. Fix also the following
notation:
\begin{equation*}
  \xymatrix{J&I\times J\ar[l]\ar[r]&I\\
    \tw{J}\times K\ar@/_1.5pc/[dd]_{p}\ar[u]^{q}\ar[d]_{r}&\tw{I\times J}\times
    K\ar[d]_{p''}\ar[dl]^{r''}\ar[u]^{q''}\ar[r]^-{\alpha}\ar[l]_-{\beta}&\tw{I}\times
    J^{\op}\times K\ar[u]_{q'}\ar[dl]^{p'}\ar@/^2pc/[ddll]^{r'}\\
    K&I^{\op}\times J^{\op}\times K\\
  J^{\op}\times K\mathrlap{.}}
\end{equation*}
Then the morphism in the statement of the adjunction property is given
by:
\begin{align*}
  \Omega^{I,J,K}_{A,B,C}:p'_{*}[q'^{*}A,r'^{*}p_{*}[q^{*}B,r^{*}C]]&\xrightarrow{\sim}p'_{*}[q'^{*}A,\alpha_{*}\beta^{*}[q^{*}B,r^{*}C]]\\
  &\xrightarrow{\sim} p'_{*}\alpha_{*}[\alpha^{*}q'^{*}A,[\beta^{*}q^{*}B,\beta^{*}r^{*}C]]\\
  &\xrightarrow{\sim}p'_{*}\alpha_{*}[\alpha^{*}q'^{*}A\otimes
  \beta^{*}q^{*}B,\beta^{*}r^{*}C] \\
  &\xrightarrow{\sim} p''_{*}[q''^{*}(\restr{A}{I\times J}\otimes
  \restr{B}{I\times J}),r''^{*}C].
\end{align*}
It is clear that this morphism is natural in the three
arguments. Moreover, as above it is straightforward to check that it
behaves well with respect to functors $u:I'\to I$, $v:J'\to J$,
$w:K'\to K$.

\subsection*{Biduality}

Fix $B\in \D(\pt)_{0}$, $I\in \mathbf{Dia}_{0}$ and $A\in
\D(I)_{0}$. We also fix the following notation:
\begin{equation*}
  \xymatrix{&\pt&\\
    \ar[rrd]_(.8){\overline{p}}\ar[d]_{\overline{q}}\tw{I^{\op}}\ar[ru]^{\overline{r}}\ar[rr]^{\mu}&&\tw{I}\ar[lu]_{r}\ar[lld]^(.8){p}\ar[d]^{q}\\
    I^{\op} &&I\mathrlap{.}}
\end{equation*}
Here, $\mu$ is the isomorphism of categories taking $j\to i$ in
$I^{\op}$ to $i\to j$ in $I$. We then define the morphism mentioned in
the statement of the biduality property,
\begin{equation}
  \label{eq:dual}
  \Upsilon^{I}_{A}:A\to \ehm{\ehm{A}{B}}{B},
\end{equation}
by adjunction as follows:
\begin{align*}
  \xymatrix@R=4pt@C=15pt{\mathllap{\overline{p}^{*}A\otimes
  \overline{q}^{*}p_{*}[q^{*}A,r^{*}B]}\ar@{=}[r]&
  \mathrlap{\overline{p}^{*}A\otimes
  \mu^{*}p^{*}p_{*}[q^{*}A,r^{*}B]}\\
  \ar[r]^{\mathrm{adj}}&
  \mathrlap{\overline{p}^{*}A\otimes
  \mu^{*}[q^{*}A,r^{*}B]}\\
\ar[r]&
  \mathrlap{\overline{p}^{*}A\otimes
  [\overline{p}^{*}A,\overline{r}^{*}B]}\\
  \ar[r]^{\mathrm{ev}}&\mathrlap{\overline{r}^{*}B\mathrlap{.}}}
\end{align*}
This is clearly natural in $A$. If $u:I'\to I$ is a functor in
$\mathbf{Dia}$ we define a morphism
\begin{equation*}
  u^{*}\ehm{\ehm{A}{B}}{B}\xrightarrow[\sim]{\Psi} \ehm{u^{\op*}\ehm{A}{B}}{B}\xleftarrow[\sim]{\Psi} \ehm{\ehm{u^{*}A}{B}}{B}.
\end{equation*}
As we know by the naturality property, this morphism is invertible,
natural in $A$, and behaves well with respect to identity and
composition of functors as well as natural transformations in
$\mathbf{Dia}$. Therefore we have defined a pseudonatural
transformation $\ehm{\ehm{\plho}{B}}{B}$. To check that
(\ref{eq:dual}) defines a modification of pseudonatural
transformations as claimed in section~\ref{sec:ehom} it suffices to
prove the following lemma.
\begin{lem} With the notation above the following diagram commutes:
  \begin{equation*}
    \xymatrix{u^{*}A\ar[r]^-{\Upsilon}\ar[d]_{\Upsilon}&u^{*}\ehm{\ehm{A}{B}}{B}\ar[d]^{\Psi}_{\sim}\\
      \ehm{\ehm{u^{*}A}{B}}{B}\ar[r]_-{\Psi}^-{\sim}&\ehm{u^{\op*}\ehm{A}{B}}{B}\mathrlap{.}}
  \end{equation*}
\end{lem}
\begin{proof}
  Using adjunction, the square can be equivalently written as the
  outer rectangle of the following diagram:
 \begin{multline*}
 \scalebox{0.95}{\xymatrix{\overline{p}'^{*}u^{*}A\otimes\overline{q}'^{*}u^{\op*}p_{*}[q^{*}A,r^{*}B]\ar[r]\ar[d]&\tw{u^{\op}}^{*}(\overline{p}^{*}A\otimes
     \overline{q}^{*}p_{*}[q^{*}A,r^{*}B])\ar[d]\ar@{-}[r]&\cdots\\
     \overline{p}'^{*}u^{*}A\otimes
     [\overline{p}'^{*}u^{*}A,\overline{r}'^{*}B]\ar[d]_{\mathrm{ev}}&\tw{u^{\op}}^{*}(\overline{p}^{*}A\otimes[\overline{p}^{*}A,\overline{r}^{*}B])\ar[d]_{\mathrm{ev}}&\\
   \overline{r}'^{*}B&\ar@{=}[l]\tw{u^{\op}}^{*}\overline{r}^{*}B\ar@{=}[r]&\cdots}}\\\\
\scalebox{0.95}{\xymatrix@C=12pt{\cdots\ar[r]& \tw{u^{\op}}^{*}(\overline{p}^{*}\ehm{\ehm{A}{B}}{B}\otimes
     \overline{q}^{*}\ehm{A}{B})\ar[d]&\ar[l]_-{\sim}
     \overline{p}'^{*}u^{*}\ehm{\ehm{A}{B}}{B}\otimes
     \overline{q}'^{*}u^{\op*}\ehm{A}{B}\ar[d]\\
     &\tw{u^{\op}}^{*}([\overline{q}^{*}\ehm{A}{B},\overline{r}^{*}B]\otimes
     \overline{q}^{*}\ehm{A}{B})\ar[d]_{\mathrm{ev}}&[\overline{q}'^{*}u^{\op*}\ehm{A}{B},\overline{r}'^{*}B]\otimes
     \overline{q}'^{*}u^{\op*}\ehm{A}{B}\ar[d]^{\mathrm{ev}}\\
     \cdots \ar@{=}[r]&\tw{u^{\op}}^{*}\overline{r}^{*}B\ar@{=}[r]&\overline{r}'^{*}B\mathrlap{.}}}
 \end{multline*}
 All three parts are easily seen to commute.
\end{proof}

\subsection*{Normalization}
\label{sec:ehom-properties-proof-normalization}
Given $J\in\mathbf{Dia}_{0}$, $A\in\D(\pt)_{0}$ and $B\in\D(J)_{0}$,
the morphism $\Lambda^{J}_{A,B}$ is the canonical identification
induced by the strict functoriality of $\D$:
\begin{equation*}
  [p_{J}^{*}A,B]\xrightarrow{\sim}\catid_{J*}[p_{J}^{*}A,B]=\ehm{A}{B}.
\end{equation*}
Clearly, this is natural in $A$ and $B$, and behaves well with respect
to functors $v:J'\to J$. The last claim in section~\ref{sec:ehom}
about $\Lambda$ explicitly amounts to the following:
\begin{itemize}
\item for $A,B\in\D(\pt)_{0}$, $\Theta$ is the canonical composition
  \begin{equation*}
    [A,B]\xrightarrow[\sim]{}\catid_{*}[A,B]\xrightarrow[\sim]{}\catid_{*}\catid_{*}[A,B]\xrightarrow[\sim]{\Lambda}\catid_{*}\catid_{*}\ehm{A}{B}
  \end{equation*}
  where $\catid$ is the unique endofunctor of the terminal category
  $\pt$;
\item for $A,C\in\D(\pt)_{0}$, $B\in\D(I)_{0}$, $D\in\D(J)_{0}$,
  $\Xi$ fits into the commutative diagram:
  \begin{equation*}
    \xymatrix{\ehm{A}{B}\boxtimes\ehm{C}{D}\ar[r]^{\Xi}&\ehm{A\boxtimes
        C}{B\boxtimes D}\\
      \ar[u]^{\Lambda}_{\sim}\restr{[p_{I}^{*}A,B]}{I\times J}\otimes
      \restr{[p_{J}^{*}C,D]}{I\times J}\ar[d]_{\sim}&[p_{I\times
        J}^{*}(A\otimes C),\restr{B}{I\times J}\otimes
      \restr{D}{I\times J}]\ar[d]^{\sim}\ar[u]_{\Lambda}^{\sim}\\
      [p_{I\times J}^{*}A,\restr{B}{I\times J}]\otimes
      [p_{I\times J}^{*}C,\restr{D}{I\times J}]\ar[r]^{(\ref{eq:prod_ihom})}&[p_{I\times
        J}^{*}A\otimes p_{I\times J}^{*}C,\restr{B}{I\times J}\otimes
      \restr{D}{I\times J}]\mathrlap{.}}
  \end{equation*}
\item for $A,B\in\D(\pt)_{0}$ and $C\in \D(J)_{0}$, $\Omega$ fits into
  the commutative diagram:
  \begin{equation*}
    \xymatrix{\ehm{A}{\ehm{B}{C}}\ar[r]_{\sim}^{\Omega}&\ehm{A\otimes
        B}{C}\\
      \ar[u]^{\Lambda}_{\sim}[p_{J}^{*}A,[p_{J}^{*}B,C]]\ar[dr]_{\sim}&[p_{J}^{*}(A\otimes
      B),C]\ar[d]^{\sim}\ar[u]_{\Lambda}^{\sim}\\
      &[p_{J}^{*}A\otimes
      p_{J}^{*}B,C]\mathrlap{.}}
  \end{equation*}
\item for $A,B\in\D(\pt)_{0}$, $\Upsilon$ is identified with the
  morphism $A\to[[A,B],B]$ which by adjunction corresponds to
  $\mathrm{ev}: A\otimes [A,B]\to B$.
\end{itemize}
All these statements follow easily from the constructions in this
section.

\section{The external trace and homotopy colimits}
\label{sec:app2}
In this section the proof of Proposition~\ref{pro:main} will be
given. Throughout we fix a closed monoidal derivator $\D$ of type
$\mathbf{Dia}$. We start with a preliminary result, already needed to
define the association $\Phi$ on page~\pageref{dfi:Phi}.
\begin{lem}\label{lem:isos}Let  $I\in\mathbf{Dia}_{0}$. Then the following three morphisms are
  invertible:
  \begin{enumerate}
  \item $p_{1!}q_{2!}q_{2}^{*}p_{1}^{*}\to\Id$ (counit of adjunction),
  \item $\Id\to p_{2*}q_{1*}q_{1}^{*}p_{2}^{*}$ (unit of adjunction),
  \item $\overline{\Psi}:[p_{I!}A,B]\to p_{I^{\op}*}\ehm{A}{B}$ for
    $A\in\D(I)_{0}$, $B\in\D(\pt)_{0}$.
  \end{enumerate}
\end{lem}
\begin{proof}

  For the first morphism, fix $i\in I_{0}$ and consider the following
  pullback square:
  \begin{equation*}
    \xymatrix{\tw{I}^{\op}_{i}\ar[r]\ar[d]_{p_{i}}&\tw{I}^{\op}\ar[d]^{p_{1}q_{2}}\\
      \pt\ar[r]_{i}&I^{0}\mathrlap{.}}
  \end{equation*}
  Since $q_{2}$ and $p_{1}$ are both fibrations so is their
  composition and by Lemma~\ref{lem:fibration-exact} the
  Beck-Chevalley transformation corresponding to the square above is
  invertible. It follows that for the counit
  $p_{1!}q_{2!}q_{2}^{*}p_{1}^{*}\to \Id$ to be invertible it is
  necessary and sufficient that $p_{i!}p_{i}^{*}\to\Id$ is (for all
  $i\in I_{0}$, by~\ref{der:conservative}). This is equivalent to
  $\Id\to p_{i*}p_{i}^{*}$ being invertible, and this is true since
  $\tw{I}^{\op}_{i}=I/i$ and thus $p_{i*}=\catid_{i}^{*}$. The second morphism in
  the statement of the Lemma is treated in the same way.

  For the last morphism, we consider the following factorization:
  \begin{equation*}
     \xymatrix@C=18pt{[p_{I!}A,B]\ar[r]^-{\mathrm{adj}}\ar[d]^{\sim}&p_{I^{\op}*}p_{I^{\op}}^{*}[p_{I!}A,B]\ar[r]^{\Psi\circ\Lambda}\ar[d]^{\sim}&
    p_{I^{\op}*}\ehm{p_{I}^{*}p_{I!}A}{B}\ar[d]^{\mathrm{adj}}\\
    p_{I*}[A,p_{I}^{*}B]\ar[d]_{\Theta}^{\sim}\ar[r]^-{\mathrm{adj}}&
    p_{I^{\op}*}p_{I^{\op}}^{*}p_{I*}[A,p_{I}^{*}B]\ar[d]_{\Theta}^{\sim}&
    p_{I^{\op}*}\ehm{A}{B}\ar[dd]^{\mathrm{adj}}\\
    p_{I*}p_{2*}q_{2*}q_{2}^{*}\ehm{A}{p_{I}^{*}B}\ar[r]^-{\mathrm{adj}}&p_{I^{\op}*}p_{I^{\op}}^{*}p_{I*}p_{2*}q_{2*}q_{2}^{*}\ehm{A}{p_{I}^{*}B}\\
    p_{I^{\op}*}p_{1*}q_{2*}q_{2}^{*}p_{1}^{*}\ehm{A}{B}\ar[u]_{\sim}^{\Psi}&\ar[u]_{\sim}^{\Psi}p_{I^{\op}*}p_{I^{\op}}^{*}p_{I^{\op}*}p_{1*}q_{2*}q_{2}^{*}p_{1}^{*}\ehm{A}{B}\ar[l]_-{\mathrm{adj}}&p_{I^{\op}*}p_{1*}q_{2*}q_{2}^{*}p_{1}^{*}\ehm{A}{B}\ar[l]_-{\mathrm{adj}}\mathrlap{.}}
  \end{equation*}
  Notice that all the vertical arrows on the left are invertible (the
  first one by Lemma~\ref{lem:mon-der}, the second and third by the
  results of section~\ref{sec:ehom}) as is the vertical arrow on the
  bottom right by part~1 of the lemma. And the composition of the
  horizontal arrows at the bottom is the identity so we only need to
  prove commutativity of the diagram.
  
  This is clear for the left half of the diagram while the right half
  may be decomposed as follows:
  \begin{multline*}
   \xymatrix{
p_{I^{\op}}^{*}[p_{I!}A,B]\ar[r]^{\Lambda}_{\sim}\ar[d]_{\mathrm{adj}}&p_{I^{\op}}^{*}\ehm{p_{I!}A}{B}\ar@{-}[r]^-{\Psi}\ar[d]^{\mathrm{adj}}&\cdots\\
      p_{I^{\op}}^{*}p_{I*}p_{I}^{*}[p_{I!}A,B]\ar[d]\ar[r]^{\Lambda}_{\sim}&\ar@{}[d]|{\numbercircled{1}}p_{I^{\op}}^{*}p_{I*}p_{I}^{*}\ehm{p_{I!}A}{B}\ar@{-}[r]^-{\mathrm{adj}}&\cdots\\
      p_{I^{\op}}^{*}p_{I*}[p_{I}^{*}p_{I!}A,p_{I}^{*}B]\ar[r]^-{\Theta}_-{\sim}\ar[d]_{\mathrm{adj}}&p_{I^{\op}}^{*}p_{I*}p_{2*}q_{2*}q_{2}^{*}\ehm{p_{I}^{*}p_{I!}A}{p_{I}^{*}B}\ar[d]_{\mathrm{adj}}\ar@{-}[r]^-{\sim}&\cdots\\
      p_{I^{\op}}^{*}p_{I*}[A,p_{I}^{*}B]\ar[r]^-{\Theta}_-{\sim}&p_{I^{\op}}^{*}p_{I*}p_{2*}q_{2*}q_{2}^{*}\ehm{A}{p_{I}^{*}B}\ar@{-}[r]^-{\sim}&\cdots}\\\\
    \xymatrix{\cdots\ar[r]&\ehm{p_{I}^{*}p_{I!}A}{B}\ar[r]^-{\mathrm{adj}}&\ehm{A}{B}\ar[d]^{\mathrm{adj}}\\
 \cdots\ar[r]&     p_{I^{\op}}^{*}p_{I*}p_{2*}q_{2*}q_{2}^{*}(p_{I}^{\op}\times
      p_{I})^{*}\ehm{p_{I!}A}{B}\ar[d]^{\Psi}_{\sim}&p_{1*}q_{2*}q_{2}^{*}p_{1}^{*}\ehm{A}{B}\ar[dd]^{\Psi}\\
\cdots\ar[r]&      p_{I^{\op}}^{*}p_{I^{\op}*}p_{1*}q_{2*}q_{2}^{*}\ehm{p_{I}^{*}p_{I!}A}{p_{I}^{*}B}\ar[d]^{\mathrm{adj}}\\
\cdots\ar[r]&      p_{I^{\op}}^{*}p_{I^{\op}*}p_{1*}q_{2*}q_{2}^{*}\ehm{A}{p_{I}^{*}B}\ar[r]^-{\mathrm{adj}}&p_{1*}q_{2*}q_{2}^{*}\ehm{A}{p_{I}^{*}B}\mathrlap{.}}
  \end{multline*}
  Everything except possibly \numbercircled{1} clearly commutes; and
  \numbercircled{1} does so by the internal hom property in
  section~\ref{sec:ehom}.
\end{proof}

From now on we take the assumptions of Proposition~\ref{pro:main} to
be satisfied. First we prove:
\begin{lem}
  $p_{I!}A$ is dualizable.
\end{lem}
\begin{proof}
  We are given an object $B$ in $\D(\pt)$ and we need to show that the
  top arrow in the following diagram is invertible:
  \begin{equation*}
    \xymatrix{[p_{I!}A,\one]\otimes
      B\ar[r]\ar[d]^{\sim}_{\overline{\Psi}}&[p_{I!}A,\one\otimes
      B]\ar[d]^{\overline{\Psi}}_{\sim}\\
      p_{I^{\op}*}\ehm{A}{\one}\otimes
      B\ar[d]^{\sim}&p_{I^{\op}*}\ehm{A}{\one\otimes B}\\
      p_{I^{\op}*}(\ehm{A}{\one}\boxtimes B)\ar[r]^-{\sim}_-{\Xi}&p_{I^{\op}*}\ehm{A}{\one\boxtimes B}\ar@{=}[u]\mathrlap{.}}
  \end{equation*}
  The two arrows labeled $\overline{\Psi}$ are invertible by the
  previous lemma, as is the vertical arrow on the bottom left by
  hypothesis~\labelcref{hyp:p_Iop}. Given $i\in I_{0}$, the fiber over
  $i$ of the morphism $\Xi:\ehm{A}{\one}\boxtimes B\to
  \ehm{A}{\one\boxtimes B}$ corresponds to the morphism
  $[i^{*}A,\one]\otimes B\to [i^{*}A,\one\otimes B]$ by the external
  product and normalization properties in section~\ref{sec:ehom}. The
  latter morphism is invertible since $A$ is fiberwise dualizable
  hence also the bottom horizontal arrow in the diagram is invertible
  (by~\ref{der:conservative}). It now suffices to prove its
  commutativity which we leave as an easy exercise.
\end{proof}

To prove commutativity of the diagram~(\ref{eq:main_diagram}) with
$g=\mathrm{Tr}(f)$ and the top horizontal arrow replaced by
$\mathrm{Tr}(p_{I!}f)$ we decompose $\mathrm{Tr}(f)$ into
coevaluation, the morphism induced by $f$ and evaluation, and
similarly for $\mathrm{Tr}(p_{I!}f)$.  Schematically:
\begin{multline}\label{eq:colim-schema}
  \xymatrix{S\ar[r]^-{\mathrm{coev}}\ar[d]&(p_{I!}A)^{*}\otimes
    p_{I!}A\otimes S\ar@{-}[r]^-{p_{I!}f}\ar[d]&\cdots\\
    p_{I^{\op}*}p_{1!}(q_{2!}\one\otimes
    \restr{S}{I^{\op}\times I})\ar[r]_-{\mathrm{coev}}&
    p_{I^{\op}*}p_{1!}(A^{\vee}\boxtimes A\otimes
    \restr{S}{I^{\op}\times I})\ar@{-}[r]_-{f}&\cdots}\\\\
  \xymatrix{\cdots\ar[r]&(p_{I!}A)^{*}\otimes
    p_{I!}A\otimes T\ar[r]^-{\mathrm{ev}}&T\\
    \cdots\ar[r]&      \ar[u]p_{I^{\op}*}p_{1!}(A^{\vee}\boxtimes A\otimes
    \restr{T}{I^{\op}\times I})\ar[r]_-{\sim\circ\mathrm{ev}} &\ar[u]p_{I!}p_{2*}(q_{1*}\one\otimes
    \restr{T}{I^{\op}\times I})\mathrlap{.}}
\end{multline}
The vertical morphisms in the middle will be described below but we
can already say here that they will be easily seen to make the square
in the middle commute. Now the fact that we have isomorphisms
\begin{align*}
  p_{I^{\op}*}(\plho\otimes p_{I^{\op}}^{*}\plho)\cong p_{I^{\op}*}\plho\otimes\plho,&&    p_{1!}(\plho\otimes p_{1}^{*}\plho)\cong p_{1!}\plho\otimes\plho
\end{align*}
allows us to neglect the twisting:
\begin{lem}
  We may assume $S=T=\one$.
\end{lem}
\begin{proof}
  Consider the following diagram:
  \begin{equation*}
    \xymatrix{\one\otimes
      S\ar[r]^-{\mathrm{coev}}\ar[d]&(p_{I!}A)^{*}\otimes
      p_{I!}A\otimes S\ar[d]\\
      p_{I^{\op}*}p_{1!}q_{2!}\one\otimes
      S\ar[r]^-{\mathrm{coev}}\ar[d]_{\sim}&p_{I^{\op}*}p_{1!}(A^{\vee}\boxtimes
      A)\otimes S\ar[d]^{\sim}\\
      p_{I^{\op}*}p_{1!}(q_{2!}\one\otimes
      p_{2}^{*}p_{I}^{*}S)\ar[r]^-{\mathrm{coev}}&p_{I^{\op}*}p_{1!}(A^{\vee}\boxtimes
      A\otimes p_{2}^{*}p_{I}^{*}S)\mathrlap{.}}      
  \end{equation*}
  It is easy to check that the composition of the two vertical
  morphisms on the left equals the left vertical morphism
  in~(\ref{eq:colim-schema}). Moreover the bottom square clearly
  commutes thus we are left to prove the commutativity of the top
  square but this does not depend on $S$. A similar argument shows
  that we may assume $T=\one$.
\end{proof}

\begin{lem}
  The left square in~(\ref{eq:colim-schema}) commutes.
\end{lem}
\begin{proof}
  By the previous lemma we may assume $S=\one$. Again, we factor the
  coevaluation morphisms on the top and bottom into two parts as
  in~(\ref{eq:coev-classical}) and~(\ref{eq:coev-new})
  respectively. This decomposes the left square
  in~(\ref{eq:colim-schema}) into two parts which we consider
  separately.

  By adjunction, the first one may be expanded as follows (the arrows
  labeled with a small Greek letter will be defined below):
  \begin{multline*}
    \xymatrix{p_{I^{\op}}^{*}\one\ar[r]^-{\mathrm{adj}}&p_{I^{\op}}^{*}[p_{I!}A,p_{I!}A]\ar@{-}[r]^-{\alpha}&\cdots\\
    \ar[d]_{\sim}\ar[u]^{\mathrm{adj}}_{\sim}(p_{1}q_{2})_{!}(p_{1}q_{2})^{*}\one\ar[r]^-{\mathrm{adj}}&\ar@{}[rd]|{\numbercircled{1}}    \ar@{=}[d]\ar[u]^{\sim}_{\mathrm{adj}}(p_{1}q_{2})_{!}(p_{1}q_{2})^{*}p_{I^{\op}}^{*}[p_{I!}A,p_{I!}A]\ar@{-}[r]^-{\alpha}&\cdots\\
(p_{1}q_{2})_{!}(p_{2}q_{2})^{*}\one\ar[r]^-{\mathrm{adj}}&(p_{1}q_{2})_{!}(p_{2}q_{2})^{*}p_{I}^{*}[p_{I!}A,p_{I!}A]\ar@{-}[r]^-{\beta}&\cdots\\
&&\\
    \ar@{=}[uu](p_{1}q_{2})_{!}(p_{2}q_{2})^{*}\one\ar@{-}[rr]^-{\Theta\circ\mathrm{adj}}&&\cdots
  }\\\\
  \xymatrix{\cdots\ar[r]&\ehm{p_{I}^{*}p_{I!}A}{p_{I!}A}\ar[r]^-{\mathrm{adj}}&\ehm{A}{p_{I!}A}\\
\cdots\ar[r]&    (p_{1}q_{2})_{!}(p_{1}q_{2})^{*}\ehm{p_{I}^{*}p_{I!}A}{p_{I!}A}\ar[u]^{\sim}_{\mathrm{adj}}&\\
\cdots\ar[r]&(p_{1}q_{2})_{!}(p_{2}q_{2})^{*}(p_{2}q_{2})_{*}q_{2}^{*}\ehm{p_{I}^{*}p_{I!}A}{p_{I}^{*}p_{I!}A}\ar[d]^{\mathrm{adj}}\ar[u]_{\gamma}&p_{1!}p_{1}^{*}\ehm{A}{p_{I!}A}\ar[d]^{\Psi}_{\sim}\ar[uu]_{\mathrm{adj}}\\
&(p_{1}q_{2})_{!}(p_{2}q_{2})^{*}(p_{2}q_{2})_{*}q_{2}^{*}\ehm{A}{p_{I}^{*}p_{I!}A}\ar[r]^-{\mathrm{adj}}&p_{1!}\ehm{A}{p_{I}^{*}p_{I!}A}\\
\cdots\ar[r]&(p_{1}q_{2})_{!}(p_{2}q_{2})^{*}(p_{2}q_{2})_{*}q_{2}^{*}\ehm{A}{A}\ar[u]_{\mathrm{adj}}\ar[r]_-{\mathrm{adj}}&p_{1!}\ehm{A}{A}\ar[u]_{\mathrm{adj}}\mathrlap{,}}
  \end{multline*}
  and the second one as follows:
  \begin{equation*}
    \xymatrix{p_{I^{\op}}^{*}[p_{I!}A,p_{I!}A]\ar@{}[rd]|{\numbercircled{2}}\ar[d]_{\mathrm{adj}\circ\alpha}&p_{I^{\op}}^{*}([p_{I!}A,\one]\otimes
      p_{I!}A)\ar[d]^{\delta}\ar[l]_-{\sim}\\\ar@{}[rd]|{\numbercircled{3}}
      \ehm{A}{p_{I!}A}&\ehm{A}{\one}\otimes
      p_{I^{\op}}^{*}p_{I!}A\ar[l]_-{\sim}^-{(\ref{eq:application})}\\
     \ar[u]_{\sim}^{\overline{\Psi}} p_{1!}\ehm{A}{A}&\ar[l]_-{\sim}^-{(\ref{eq:application})}p_{1!}(p_{1}^{*}\ehm{A}{\one}\otimes p_{2}^{*}A)\ar[u]_{\sim}\mathrlap{.}}
  \end{equation*}
  Notice first that these two diagrams indeed ``glue'' together. Thus
  it suffices to show commutativity of the rectangles marked with a
  number (the other ones are easily seen to commute).

  \numbercircled{1} may be expanded as follows (set $B=p_{I!}A$):
  \begin{equation*}
    \xymatrix{q_{2}^{*}p_{1}^{*}p_{I^{\op}}^{*}[B,B]\ar@{=}[d]\ar[r]^{\Lambda}_{\sim}&q_{2}^{*}p_{1}^{*}p_{I^{\op}}^{*}\ehm{B}{B}\ar[r]^{\Psi}_{\sim}&q_{2}^{*}p_{1}^{*}\ehm{p_{I}^{*}B}{B}\ar[d]_{\sim}^{\Psi}\\
      q_{2}^{*}p_{2}^{*}p_{I}^{*}[B,B]\ar[r]_{\sim}^{\Lambda}&q_{2}^{*}p_{2}^{*}p_{I}^{*}\ehm{B}{B}\ar[r]^{\Psi}_{\sim}&q_{2}^{*}\ehm{p_{I}^{*}B}{p_{I}^{*}B}\\
 q_{2}^{*}p_{2}^{*}p_{I}^{*}[B,B]\ar@{=}[u]\ar[r]_{\sim}&q_{2}^{*}p_{2}^{*}[p_{I}^{*}B,p_{I}^{*}B]\ar[r]^-{\Theta}_-{\sim}&q_{2}^{*}p_{2}^{*}p_{2*}q_{2*}q_{2}^{*}\ehm{p_{I}^{*}B}{p_{I}^{*}B}\ar[u]_{\mathrm{adj}}\mathrlap{.}}
  \end{equation*}
  The top rectangle commutes by the naturality property, the bottom
  rectangle by the internal hom property of section~\ref{sec:ehom}.

  For \numbercircled{2} consider the following decomposition (by
  adjunction again):
  \begin{equation*}
    \xymatrixcolsep{10pt}
    \scalebox{.95}{\xymatrix{[p_{I!}A,p_{I!}A]\ar[d]_{\Lambda\circ\mathrm{adj}}&\ar[l]_{\sim}[p_{I!}A,\one]\otimes
      p_{I!}A\ar[d]\ar@{=}[r]&[p_{I!}A,\one]\otimes p_{I!}A\ar[d]^{\Lambda\circ\mathrm{adj}}\\
    p_{I^{\op}*}p_{I^{\op}}^{*}\ehm{p_{I!}A}{p_{I!}A}\ar[dd]^{\sim}_{\Psi}&p_{I^{\op}*}p_{I^{\op}}^{*}(\ehm{p_{I!}A}{\one}\otimes
    p_{I!}A)\ar[l]_-{\sim}^-{(\ref{eq:application})}\ar[dr]_{\sim}&p_{I^{\op}*}p_{I^{\op}}^{*}\ehm{p_{I!}A}{\one}\otimes
    p_{I!}A\ar[d]^{\sim}
\\
    &&p_{I^{\op}*}(p_{I^{\op}}^{*}\ehm{p_{I!}A}{\one}\otimes
    p_{I^{\op}}^{*}p_{I!}A)\ar[d]^{\Psi}_{\sim}\\
    p_{I^{\op}*}\ehm{p_{I}^{*}p_{I!}A}{p_{I!}A}\ar[d]_{\mathrm{adj}}&&p_{I^{\op}*}(\ehm{p_{I}^{*}p_{I!}A}{\one}\otimes
    p_{I^{\op}}^{*}p_{I!}A)\ar[d]^{\mathrm{adj}}\ar[ll]_{\sim}^{(\ref{eq:application})}\\
    p_{I^{\op}*}\ehm{A}{p_{I!}A}&&p_{I^{\op}*}(\ehm{A}{\one}\otimes
    p_{I^{\op}}^{*}p_{I!}A)\ar[ll]_{\sim}^{(\ref{eq:application})}\mathrlap{.}}}
  \end{equation*}
  The top left square commutes by the normalization property, the
  pentagon in the middle by the external product and normalization
  properties of section~\ref{sec:ehom}. The rest is clearly
  commutative. (One also needs here Lemma~\ref{lem:isos} to ensure
  that the morphism corresponding to $\delta$ under adjunction is
  invertible.)

  Next, we may decompose \numbercircled{3} by adjunction as follows:
  \begin{equation*}
    \xymatrix{    p_{1}^{*}\ehm{A}{p_{I!}A}\ar[d]^{\sim}_{\Psi}&p_{1}^{*}(\ehm{A}{\one}\boxtimes
      p_{I!}A)\ar[l]^-{\sim}_-{(\ref{eq:application})}\ar[d]^{\sim}\\
      \ehm{A}{p_{I}^{*}p_{I!}A}&\ehm{A}{\one}\boxtimes
      p_{I}^{*}p_{I!}A\ar[l]^-{\sim}_-{(\ref{eq:application})}\\
      \ehm{A}{A}\ar[u]^{\mathrm{adj}}&\ehm{A}{\one}\boxtimes
      A\ar[l]_-{(\ref{eq:application})}^-{\sim}\ar[u]_{\mathrm{adj}}\mathrlap{.}}
  \end{equation*}
  Both squares commute by the external product property in
  section~\ref{sec:ehom}.
\end{proof}
The following lemma completes the proof of Proposition~\ref{pro:main}.
\begin{lem}
  The right square in~(\ref{eq:colim-schema}) commutes.
\end{lem}
\begin{proof}
  Again, we may assume $T=\one$ by the lemma above. First,
  (\ref{eq:coev-ev}) lets us replace the evaluation morphism on the
  top by the following composition (the arrows labeled with a small Greek
  letter will be defined below):
  \begin{equation*}
    \xymatrixcolsep{14pt}
\scalebox{0.9}{\xymatrix{\ar@{}[rd]|{\numbercircled{4}}\ar[d]^{\sim}_{\Lambda}(p_{I!}A)^{*}\otimes
      p_{I!}A\ar[r]_-{\sim}&\ar@{}[rd]|{\numbercircled{6}}(p_{I!}A)^{*}\otimes
      (p_{I!}A)^{**}\ar[r]\ar[d]^{\sim}_{\overline{\Psi}\circ\Lambda}&\ar@{}[rd]|{\numbercircled{7}}((p_{I!}A)^{*}\otimes
      p_{I!}A)^{*}\ar[r]^-{\mathrm{coev}}&\one\\
      \ar@{}[rd]|{\numbercircled{5}}\ar[r]_-{\sim}^-{\Upsilon}p_{I^{\op}*}A^{\vee}\otimes
      p_{I!}A&p_{I!}A^{\vee\vee}\otimes
      p_{I^{\op}*}A^{\vee}&\ar[u]^{\theta}_{\sim}p_{I!}(A^{\vee}\boxtimes
      p_{I!}A)^{\vee}\ar[d]^{\sim}_{\overline{\Psi}}&\\
    \ar[r]_-{\sim}\ar[u]^{\varepsilon}_{\sim}p_{I^{\op}*}p_{1!}(A^{\vee}\boxtimes
    A)&\ar[u]^{\eta}_{\sim}\ar[r]_{\Xi}p_{I!}p_{2*}\mu_{*}(A^{\vee\vee}\boxtimes
    A^{\vee})&\ar[r]_{\mathrm{coev}}p_{I!}p_{2*}\mu_{*}(A^{\vee}\boxtimes A)^{\vee}&p_{I!}p_{2*}\mu_{*}\ehm{q_{2!}\one}{\one}\ar[uu]\mathrlap{.}}
}    
  \end{equation*}
  The commutativity of \numbercircled{4} can be checked on each
  tensor factor separately; only one of them is possibly non-obvious:
  \begin{equation*}
    \xymatrix{A\ar[r]^-{\mathrm{adj}}\ar[d]_{\Upsilon}^{\sim}&p_{I}^{*}p_{I!}A\ar[d]_{\Upsilon}^{\sim}\ar[r]^-{\Upsilon}_-{\sim}&p_{I}^{*}\ehm{\ehm{p_{I!}A}{\one}}{\one}\ar[d]^{\Psi}_{\sim}\\
      \ehm{\ehm{A}{\one}}{\one}\ar[d]_{\mathrm{adj}}\ar[r]^-{\mathrm{adj}}&\ehm{\ehm{p_{I}^{*}p_{I!}A}{\one}}{\one}\ar[r]_-{\sim}^-{\Psi}&\ehm{p_{I^{\op}}^{*}\ehm{p_{I!}A}{\one}}{\one}\\
      \ehm{p_{I^{\op}}^{*}p_{I^{\op}*}\ehm{A}{\one}}{\one}\ar[rru]^{\overline{\Psi}}&&p_{I}^{*}\ehm{p_{I^{\op}*}\ehm{A}{\one}}{\one}\ar[ll]_{\sim}^{\Psi}\ar@/{}_{5pc}/[uu]_{\overline{\Psi}}\mathrlap{.}}
  \end{equation*}
  The two squares in the top row commute by the biduality property of
  section~\ref{sec:ehom} while the rest is clearly commutative.

  \numbercircled{5} may be decomposed as follows:
  \begin{equation*}
    \xymatrixcolsep{10pt}
    \scalebox{0.88}{\xymatrix{\ar[d]_{}\ar@{=}[r]p_{I^{\op}*}A^{\vee}\otimes
      p_{I!}A&p_{I^{\op}*}A^{\vee}\otimes
      p_{I!}A\ar[r]&p_{I!}A\otimes p_{I^{\op}*}A^{\vee}\ar[r]^-{\Upsilon}&p_{I!}A^{\vee\vee}\otimes
      p_{I^{\op}*}A^{\vee}\\
      p_{I^{\op}*}(A^{\vee}\otimes
      p_{I^{\op}}^{*}p_{I!}A)&\ar[d]_{}\ar[u]^{}p_{I!}(p_{I}^{*}p_{I^{\op}*}A^{\vee}\otimes
      A)\ar[r]&\ar[r]^-{\Upsilon}\ar[d]_{}\ar[u]^{}p_{I!}(A\otimes p_{I}^{*}p_{I^{\op}*}A^{\vee})&\ar[d]_{}\ar[u]^{}p_{I!}(A^{\vee\vee}\otimes p_{I}^{*}p_{I^{\op}*}A^{\vee})\\
      \ar[u]^{}p_{I^{\op}*}(A^{\vee}\otimes
      p_{1!}p_{2}^{*}A)&\ar[d]\ar[r]p_{I!}(p_{2*}p_{1}^{*}A^{\vee}\otimes
      A)&\ar[d]\ar[r]^-{\Upsilon}p_{I!}(A\otimes p_{1*}'p_{2}'^{*}A^{\vee})&p_{I!}(A^{\vee\vee}\otimes p_{1*}'p_{2}'^{*}A^{\vee})\ar[d]\\
      \ar[u]^{}p_{I^{\op}*}p_{1!}(A^{\vee}\boxtimes
    A)&\ar[r]\ar[l]^{}p_{I!}p_{2*}(A^{\vee}\boxtimes
    A)&\ar[r]^-{\Upsilon}p_{I!}p_{2*}\mu_{*}(A\boxtimes
    A^{\vee})&p_{I!}p_{2*}\mu_{*}(A^{\vee\vee}\boxtimes
    A^{\vee})\mathrlap{.}}}
  \end{equation*}
  Here, $p_{1}'$ and $p_{2}'$ are the projections onto the factors of
  $I\times I^{\op}$ and all arrows are invertible. All rectangles of
  this diagram are easily seen to commute (for the leftmost one may
  use \cite[2.1.105]{ayoub07-thesis-1}).

  Next we turn to \numbercircled{6}. In the decomposition of it (use the
  normalization property of section~\ref{sec:ehom} for the top
  horizontal arrow),
  \begin{equation*}
    \xymatrix{(p_{I!}A)^{\vee\vee}\boxtimes (p_{I!}A)^{\vee}\ar[r]^-{\Xi}&((p_{I!}A)^{\vee}\boxtimes
      p_{I!}A)^{\vee}\\
    (p_{I^{\op}*}A^{\vee})^{\vee}\boxtimes (p_{I!}A)^{\vee}\ar[u]^{\overline{\Psi}}\ar[r]^-{\Xi}&(p_{I^{\op}*}A^{\vee}\boxtimes
      p_{I!}A)^{\vee}\ar[u]_{\overline{\Psi}}\\
    p_{I!}(A^{\vee\vee})\boxtimes (p_{I!}A)^{\vee}\ar[u]^{\overline{\Psi}}&(p_{I^{\op}*}(A^{\vee}\boxtimes
      p_{I!}A))^{\vee}\ar[u]\\
    p_{I!}(A^{\vee\vee}\boxtimes (p_{I!}A)^{\vee})\ar[u]\ar[r]^-{\Xi}\ar[d]&p_{I!}(A^{\vee}\boxtimes
    p_{I!}A)^{\vee}\ar[u]_{\overline{\Psi}}\ar[d]\\
    p_{I!}(A^{\vee\vee}\boxtimes
    p_{I^{\op}*}A^{\vee})\ar[d]&p_{I!}(A^{\vee}\otimes p_{1!}p_{2}^{*}A)^{\vee}\ar[d]\\
    p_{I!}(A^{\vee\vee}\otimes
    p_{1*}'p_{2}'^{*}A^{\vee})\ar[d]&p_{I!}(p_{1!}(A^{\vee}\boxtimes
      A))^{\vee}\ar[d]^{\overline{\Psi}}\\
  \ar[r]^-{\Xi}p_{I!}p_{2*}\mu_{*}(A^{\vee\vee}\boxtimes
  A^{\vee})&p_{I!}p_{2*}\mu_{*}(A^{\vee}\boxtimes A)^{\vee}\mathrlap{,}}
  \end{equation*}
  everything commutes by the external product property of
  section~\ref{sec:ehom} (and adjunction). All vertical arrows are
  invertible. 

  It remains to prove the commutativity of \numbercircled{7}. In the
  diagram
  \begin{equation*}
    \xymatrix{\ehm{(p_{I!}A)^{*}\otimes
        p_{I!}A}{\one}\ar[r]^-{\mathrm{coev}}&\ehm{\one}{\one}\ar[r]&\one\\
\ar[u]      \ehm{p_{I^{\op}*}p_{1!}(A^{\vee}\boxtimes
        A)}{\one}\ar[d]_{\overline{\Psi}}\ar[r]^-{\mathrm{coev}}&\ar[u]\ehm{p_{I^{\op}*}p_{1!}q_{2!}\one}{\one}\ar[d]_{\overline{\Psi}}&p_{I!}p_{2*}q_{1*}q_{1}^{*}p_{2}^{*}p_{I}^{*}\ehm{\one}{\one}\ar[lu]\ar[d]^{\Psi}_{\sim}\\
      p_{I!}p_{2*}\mu_{*}\ehm{A^{\vee}\boxtimes
        A}{\one}\ar[r]^{\mathrm{coev}}&p_{I!}p_{2*}\mu_{*}\ehm{q_{2!}\one}{\one}\ar[r]^{\overline{\Psi}} &p_{I!}p_{2*}q_{1*}\ehm{\one}{\one}\mathrlap{,}}
  \end{equation*}
  the top left square is simply $\ehm{\plho}{\one}$ applied to the left
  square in~(\ref{eq:colim-schema}). It follows that this square is
  commutative. Moreover it is easy to see that the composition of the
  left vertical arrows is the same as of the ones in
  \numbercircled{7}. Thus this diagram is a decomposition
  of~\numbercircled{7}. The rest of the diagram clearly commutes.
\end{proof}
\section{$\D(G)$ for $G$ a finite group}
\label{sec:group}
The question, given a category $I$, whether $I$-diagrams and morphisms
of such in the homotopy categories can be lifted (and if so whether
uniquely) to the homotopy categories of $I$-diagrams has always been
of interest (see \eg{}~\cite[chapitre~IV]{grothendieck-derivators}
or~\cite[p.~2]{heller-homotopytheories}). The goal of this last section
is to give a proof for the (well-known) answer in the case of $I$ a
finite group.
\begin{pro}\label{pro:finite-group}
  Let $\D$ be an additive derivator of type $\mathbf{Dia}$, let $G$ be
  a finite group in $\mathbf{Dia}$ and assume that $\#G$ is invertible
  in $R_{\D}$. Then the canonical functor
  \begin{equation*}
    \mathrm{dia}_{G}:\D(G)\to \mathbf{CAT}(G^{\op},\D(\pt))
  \end{equation*}
  is fully faithful. If, in addition, $\D(G)$ is pseudo-abelian then
  the functor is an equivalence of categories.
\end{pro}
\begin{rem}
  Suppose that $\D$ is triangulated and that $\mathbf{Dia}$ contains
  countable discrete categories. In this case $\D(G)$ has countable
  direct sums, and it follows from~\cite[1.6.8]{neeman01-trcat} that
  $\D(G)$ is pseudo-abelian.
\end{rem}
\begin{proof}[Proof of Proposition~\ref{pro:finite-group}]
  We need to understand the two adjunctions $e_{!}\dashv e^{*}$ and
  $e^{*}\dashv e_{*}$ where $e:\pt\to G$ is the unique functor.

  Consider the following comma square where $\eta$ on the component
  corresponding to $g\in G$ is $g$:
  \begin{equation*}
    \xymatrix{\xtwocell[1,1]{}\omit{^{\eta}}\coprod_{G}\pt\ar[d]_{p}\ar[r]^-{p}&\pt\ar[d]^{e}\\
      \pt\ar[r]_-{e}&G\mathrlap{,}}
  \end{equation*}
  By~\ref{der:kanextptwise}, the two compositions
  \begin{align*}
    \xymatrix@R=5pt{p_{!}p^{*}\ar[r]^-{\mathrm{adj}}&p_{!}p^{*}e^{*}e_{!}\ar[r]^{\eta^{*}}&p_{!}p^{*}e^{*}e_{!}\ar[r]^-{\mathrm{adj}}&e^{*}e_{!},\\
      p_{*}p^{*}&p_{*}p^{*}e^{*}e_{*}\ar[l]^-{\mathrm{adj}}&p_{*}p^{*}e^{*}e_{*}\ar[l]^{\eta^{*}}&e^{*}e_{*}\ar[l]^-{\mathrm{adj}}}
  \end{align*}
  are invertible, yielding identifications
  \begin{align*}
    e^{*}e_{!}\cong {\textstyle\coprod_{G}}, && e^{*}e_{*}\cong {\textstyle\prod_{G}},
  \end{align*}
  and therefore a canonical morphism $e^{*}e_{!}\to e^{*}e_{*}$ which
  is invertible if $G$ is finite.

  Under these identifications the (contravariant) action of $G$ on
  $e^{*}e_{!}$ (obtained by applying $\mathrm{dia}_{G}$ to $e_{!}$)
  is given by right translation, and on $e^{*}e_{*}$ by left
  translation. Indeed, let $A\in \D(\pt)_{0}$ be an arbitrary object
  and set $B=e^{*}e_{!}A$, fix also $g\in G$. Then the following
  diagram commutes where $r_{g}((x_{h})_{h})=(x_{h})_{hg}$:
  \begin{align*}
    \xymatrix{\coprod_{h\in G}A\ar[r]^-{\mathrm{adj}}\ar[d]_{r_{g}}&\coprod_{h\in
        G}B\ar[r]^{\coprod_{h}h^{*}}\ar[d]_{r_{g}}&\coprod_{h\in
        G}B\ar[r]^-{\sum}&B\ar[d]^{g^{*}}\\
      \coprod_{h\in G}A\ar[r]_-{\mathrm{adj}}&\coprod_{h\in G}B\ar[r]_{\coprod_{h}h^{*}}&\coprod_{h\in G}B\ar[r]_-{\sum}&B\mathrlap{.}}
  \end{align*}
  Thus the claim in the case of $e^{*}e_{!}$; the case of
  $e^{*}e_{*}$ is proved in a similar way.

  Next, we would like to describe the units and counits of the
  adjunctions. We first deal with the unit of $e_{!}\dashv
  e^{*}$. Let $i:\pt\to \coprod_{G}\pt$ be the inclusion of the
  component corresponding to $\unit_{G}$.
  \begin{equation*}
    \xymatrix{p_{!}p^{*}\ar[r]^-{\mathrm{adj}}&p_{!}p^{*}e^{*}e_{!}\ar[r]^{\eta^{*}}&p_{!}p^{*}e^{*}e_{!}\ar[dd]^-{\mathrm{adj}}\\
      \ar[d]_{\sim}p_{!}i_{!}i^{*}p^{*}\ar[u]^{\mathrm{adj}}\ar[r]^-{\mathrm{adj}}&p_{!}i_{!}i^{*}p^{*}e^{*}e_{!}\ar[u]^{\mathrm{adj}}\ar[dr]^-{\sim}\ar[ru]^{\mathrm{adj}}\\
      \Id\ar[rr]_-{\mathrm{adj}}&&e^{*}e_{!}\mathrlap{.}}
  \end{equation*}
  The diagram clearly commutes and hence the unit $\Id\to e^{*}e_{!}$
  is given by the inclusion of the unit component into
  $\coprod_{G}$. Similarly, the counit $e^{*}e_{*}\to\Id$ is the
  projection onto the component corresponding to
  $\unit_{G}$.

  Next, we want to describe the other two (co)units (at least after
  applying $e^{*}$). For this consider the composition of the unit and
  the counit of the adjunction,
  \begin{equation*}
    e^{*}\to {\textstyle\coprod_{G}}e^{*}\to e^{*},
  \end{equation*}
  which we know to be the identity. By the description of the first
  morphism above we see that the $\unit_{G}$-component of the second
  morphism has to be the identity. But this second morphism is also
  $G$-equivariant so the description of the $G$-action above implies
  that the morphism is the action of $g$ on the $g$-component for any
  $g\in G$. Similarly, the counit $e^{*}\to \prod_{G}e^{*}$ is given
  by the action of $g$ on the $g$-component.

  We now have enough information to describe the composition
  \begin{equation*}
    \xi:e_{!}\xrightarrow{\mathrm{adj}} e_{*}e^{*}e_{!}\to e_{*}e^{*}e_{*}\xrightarrow{\mathrm{adj}} e_{*}
  \end{equation*}
  after applying $e^{*}$. Indeed, it can then be identified with the
  following one:
  \begin{align*}
    \xymatrix@R=1pt{\coprod_{G}\ar[r]&
    \prod_{G}\coprod_{G}\ar[r]&\prod_{G}\prod_{G}\ar[r]&\prod_{G}\\
    (x_{h})_{h}\ar@{|->}[r]&((x_{hg^{-1}})_{h})_{g}\ar@{|->}[r]&
    ((x_{hg^{-1}})_{h})_{g}\ar@{|->}[r]& (x_{g^{-1}})_{g}\mathrlap{.}}
  \end{align*}
  Since this morphism is invertible and $e^{*}$ conservative
  (by~\ref{der:conservative}), also $\xi$ is invertible, and it thus
  makes sense to consider the composition
  \begin{align}\label{eq:mult-G}
    \Id\to e_{*}e^{*}\xrightarrow{\xi^{-1}}e_{!}e^{*}\to \Id.
  \end{align}
  After applying $e^{*}$ it can be identified with
  \begin{align*}
    \xymatrix@R=1pt{e^{*}\ar[r]&\prod_{G}e^{*}\ar[r]&\coprod_{G}e^{*}\ar[r]& e^{*}\\
    x\ar@{|->}[r]& (g^{*}x)_{g}\ar@{|->}[r]& ((g^{-1})^{*}x)_{g}\ar@{|->}[r]&
    \sum_{g\in G}g^{*}(g^{-1})^{*}x=\#G\cdot x\mathrlap{.}}
  \end{align*}

  If $\#G$ is invertible in $R_{\D}$ then this morphism and (again,
  by~\ref{der:conservative}) also~(\ref{eq:mult-G}) is invertible, in
  particular there is, for every $B\in\D(G)_{0}$, a factorization of
  the identity morphism of $B$: $B\to e_{*}e^{*}B\to B$. For any
  $A\in\D(G)_{0}$, this factorization in turn induces the horizontal
  arrows in the following commutative diagram (${\cal
    C}=\mathbf{CAT}(G^{\op},\D(\pt))$, $d=\mathrm{dia}_{G}$):
  \begin{equation*}
    \xymatrix{\D(G)(A,B)\ar[r]\ar[d]_{d}&\D(G)(A,e_{*}e^{*}B)\ar[r]\ar[d]_{d}&\D(G)(A,B)\ar[d]^{d}\\
      {\cal C}(dA,dB)\ar[r]&{\cal C}(dA,de_{*}e^{*}B)\ar[r]&{\cal C}(dA,dB)\mathrlap{.}}
  \end{equation*}
  The first top horizontal arrow is injective hence if the middle
  vertical arrow is injective then so is the left vertical
  one. Similarly, the second bottom horizontal arrow is surjective
  hence if the middle vertical arrow is surjective then so is the
  right vertical one. Consequently, to prove fully faithfulness of
  $\mathrm{dia}_{G}$ it suffices to prove bijective the middle
  vertical arrow (for all $A$ and $B$). Now, the source of this map
  can be identified with $\D(\pt)(e^{*}A,e^{*}B)$ by adjunction, while
  the target is the set of $G^{\op}$-morphisms in $\D(\pt)$ from
  $e^{*}A$ to the left regular representation associated to
  $e^{*}B$ --- which is also $\D(\pt)(e^{*}A,e^{*}B)$.

  It remains to show essential surjectivity of
  $\mathrm{dia}_{G}$. Given an object $A\in\D(\pt)_{0}$ with a
  $G^{\op}$-action $\rho$, consider the two morphisms
  \begin{equation*}
    \xymatrix@R=1pt{A\ar[r]^-{\alpha}&\prod_{G}A&\text{and}&\prod_{G}A\ar[r]^-{\beta}&A\\
      x\ar@{|->}[r]&(\rho(g)x)_{g}&&(x_{g})_{g}\ar@{|->}[r]&\frac{1}{\#G}\sum_{g\in
      G}\rho(g^{-1})x_{g}\mathrlap{.}}
  \end{equation*}
  They give rise to a $G^{\op}$-equivariant decomposition of the
  identity on $A$:
  \begin{equation*}
    \catid_{A}:A\xrightarrow{\alpha}\mathrm{dia}_{G}(e_{*}A)\xrightarrow{\beta}A.
  \end{equation*}
  By fullness of $\mathrm{dia}_{G}$ proved above, there exists
  $p\in\D(G)(e_{*}A,e_{*}A)$ with
  $\mathrm{dia}_{G}(p)=\alpha\beta$. By faithfulness also proved
  above, the equality
  \begin{equation*}
    \mathrm{dia}_{G}(p^{2})=\mathrm{dia}_{G}(p)^{2}=(\alpha\beta)^{2}=\alpha\beta=\mathrm{dia}_{G}(p)
  \end{equation*}
  implies that $p$ is a projector, and therefore if $\D(G)$ is
  pseudo-abelian then there is a decomposition
  \begin{equation*}
    e_{*}A=\ker(p)\oplus \im(p).
  \end{equation*}
  Let $\alpha':\im(p)\to e_{*}A$ be the inclusion, and
  $\beta':e_{*}A\to\im(p)$ the projection. Then
  \begin{align*}
    (\mathrm{dia}_{G}(\beta')\alpha)(\beta\mathrm{dia}_{G}(\alpha'))&=\mathrm{dia}_{G}(\beta')\mathrm{dia}_{G}(p)\mathrm{dia}_{G}(\alpha')\\
    &=\mathrm{dia}_{G}(\beta'p\alpha')\\
    &=\mathrm{dia}_{G}(\catid_{\im(p)})\\
    &=\catid_{\mathrm{dia}_{G}(\im(p))},\intertext{and}
    (\beta\mathrm{dia}_{G}(\alpha'))(\mathrm{dia}_{G}(\beta')\alpha)&=\beta\mathrm{dia}_{G}(\alpha'\beta')\alpha\\
    &=\beta \mathrm{dia}_{G}(p)\alpha\\
    &=\beta\alpha\beta\alpha\\
    &=\catid_{A}.
  \end{align*}
  We conclude that $A\cong \mathrm{dia}_{G}(\im(p))$.
\end{proof}
\phantomsection
\addcontentsline{toc}{section}{References}
\printbibliography{}
\end{document}